\documentclass[11pt,letterpaper]{amsart}

\usepackage{amsmath,amssymb,amsfonts}
\usepackage[foot]{amsaddr}

\usepackage{mathtools}
\usepackage[shortlabels]{enumitem}
\setlist{leftmargin=*}
\usepackage[showdow]{datetime2}
 \usepackage{tikz-cd}

\usepackage[textsize=tiny]{todonotes}
\usepackage[colorlinks=true, linkcolor=blue]{hyperref}

\usepackage[sorted]{amsrefs}

\newtheorem{thm}{Theorem}[section]
\newtheorem*{thm*}{Theorem}
\newtheorem{cor}[thm]{Corollary}
\newtheorem{lem}[thm]{Lemma}
\newtheorem{prop}[thm]{Proposition}
\newtheorem{prob}[thm]{Problem}

\newtheorem{claim}[thm]{Claim}
\newtheorem{fact}[thm]{Fact}
\newtheorem{ass}{Assumption}

\theoremstyle{definition}
\newtheorem{defn}[thm]{Definition}
\theoremstyle{remark}
\newtheorem{ntn}[thm]{Notation}
\newtheorem{rem}[thm]{Remark}

\newtheorem{sample}[thm]{Example} \numberwithin{equation}{section}
    {\medskip\begingroup\leftskip 0.5cm\rightskip 0.5cm\noindent\begin{small}{\bf Remark.}}
    {\end{small}\par\endgroup}
{\begin{list}{$\bullet$}
 {\settowidth{\labelwidth}{\textsf{$\bullet$}} \setlength{\leftmargin}{10pt}}}
{\end{list}}
\newtheorem{theorem}{Theorem}[section]
\newtheorem{lemma}[theorem]{Lemma}

\newcounter{ssample}[section]

{\color{Bittersweet} \noindent Example \refstepcounter{ssample}\hbox{\bf \arabic{section}.\arabic{ssample}.}}
{}

\newcounter{insertcount}

    {\color{blue} \medskip\begingroup\noindent\begin{small}{\color{blue} \stepcounter{insertcount}
          {
            \bf Insert \arabic{insertcount}. #1.}
            \addcontentsline{toc}{subsection}{{\ \ \small  Insert \arabic{insertcount}: #1}}
               \leavevmode  }
           }
    {\end{small}\par\endgroup}

\newcommand{\mrmk}[1]
{{\tiny$^{\spadesuit}$}\marginpar{\fbox{\footnotesize #1}}}
\def\strutdepth{\dp\strutbox}%
\def\marginalnote#1{\strut\vadjust{\kern-\strutdepth\specialnote{#1}}}%
\def\specialnote#1{\vtop to \strutdepth{\baselineskip%
\strutdepth\vss\llap{\hbox{\scriptsize \bf #1}}\null}}%

\newcommand{\CN}{\mathcal N}

\newcommand{\RR}{\mathbb{R}}
\newcommand{\CC}{\mathbb{C}}

\def\CF{\mathcal F}

\newcommand{\NN}{\mathbb N}
\newcommand{\rarr}{\rightarrow}

\newcommand{\co}{\circ}

\def\tp{\mathrm{tp}}
\def\dcl{\mathrm{dcl}}
\newcommand{\sub}{\subseteq}

\reversemarginpar

\def\CL{\mathcal{L}}

\def\CG{\mathcal G}

\newcommand{\0}{\emptyset}

\def\CM{\mathcal{M}}

\def\CN{\mathcal{N}}

\def\cL{\mathcal{L}}

\newcommand*\bbar[1]{%
  \hbox{%
    \vbox{%
      \hrule height 0.5pt 
      \kern0.5ex
      \hbox{%
        \kern-0.1em
        \ensuremath{#1}%
        \kern-0.1em
      }%
    }%
  }%
}

\gdef\eq{\mathrm{eq}}

\def\Ind#1#2{#1\setbox0=\hbox{$#1x$}\kern\wd0\hbox to 0pt{\hss$#1\mid$\hss}
\lower.9\ht0\hbox to 0pt{\hss$#1\smile$\hss}\kern\wd0}

\def\ind{\mathop{\mathpalette\Ind{}}}

\def\notind#1#2{#1\setbox0=\hbox{$#1x$}\kern\wd0
\hbox to 0pt{\mathchardef\nn=12854\hss$#1\nn$\kern1.4\wd0\hss}
\hbox to 0pt{\hss$#1\mid$\hss}\lower.9\ht0 \hbox to 0pt{\hss$#1\smile$\hss}\kern\wd0}

\title[Elekes-Szab\'o for stable and o-minimal hypergraphs]{Model-theoretic Elekes-Szab\'o for stable and o-minimal hypergraphs}

\makeatletter
\renewcommand{\email}[2][]{%
  \ifx\emails\@empty\relax\else{\g@addto@macro\emails{,\space}}\fi%
  \@ifnotempty{#1}{\g@addto@macro\emails{\textrm{(#1)}\space}}%
  \g@addto@macro\emails{#2}%
}
\makeatother

\author{Artem Chernikov}
\address{Department of Mathematics, University of California, Los Angeles, Los Angeles, CA 90095, USA}
\email[Corresonding]{chernikov@math.ucla.edu}

\author{Ya'acov Peterzil}
\address{Department of Mathematics, University of Haifa, Haifa, Israel}
\email{kobi@math.haifa.ac.il}

\author{Sergei Starchenko}
\address{Department of Mathematics, University of Notre Dame, Notre Dame,
IN, 46656, USA}
\email{sstarche@nd.edu}

\begin{document}

\maketitle

 \gdef\tY{\tilde{Y}}
  \gdef\Th{\operatorname{Th}}
\gdef\tZ{\tilde{Z}}
\gdef\RM{\operatorname{RM}}
\gdef\mfp{{\mathfrak p}}
\gdef\dimp{\dim_\mfp}
\gdef\multp{\operatorname{mult}_\mfp}
\gdef\MR{\operatorname{MR}}
\gdef\ttimes{{\times}}
\gdef\MM{\mathbb{M}}
\gdef\tU{\widetilde{U}}
\gdef\tV{\widetilde{V}}
\gdef\vfs{{\varphi^*}}
\gdef\cp{\operatorname{cp}}
\gdef\acl{\operatorname{acl}}
\gdef\alg{\operatorname{alg}}
\gdef\dcl{\operatorname{dcl}}
\gdef\omin{\operatorname{o-min}}
\gdef\mrp{\mathrm{p}}
\gdef\al{\textcolor{red}}
\def\tdt{\times\dotsb\times}
\newcommand{\nospaceperiod}{\makebox[0pt][l]{\,.}}
\gdef\ACF{\operatorname{ACF}}

\begin{abstract}
A theorem of Elekes and Szab\'{o} recognizes algebraic groups among certain complex algebraic varieties with maximal size intersections with finite grids.
We establish a generalization to relations of any arity and dimension, definable in: 1) stable structures with distal expansions (includes algebraically and differentially closed fields of characteristic $0$); and 2) $o$-minimal expansions of groups.
Our methods provide explicit bounds on the power saving exponent in the non-group case. Ingredients of the proof include: a higher arity generalization of the abelian group configuration theorem in stable structures, along with a purely combinatorial variant characterizing Latin hypercubes that arise from abelian groups; and Zarankiewicz-style bounds for hypergraphs definable in distal structures.
\end{abstract}

\tableofcontents

\section{Introduction}
\subsection{History, and a special case of the main theorem}

Erd\H{o}s and Szemer\'edi \cite{erdHos1983sums} observed the following sum-product phenomenon: there exists $c \in \mathbb{R}_{>0}$ such that for any finite set $A \subseteq \mathbb{R}$,
$$\max\left\{|A+A|, |A\cdot A|\right\} \geq  |A|^{1+c}.$$
They conjectured that this holds with $c = 1- \varepsilon$ for an arbitrary $\varepsilon \in \mathbb{R}_{>0}$, and
by the work of Solymosi \cite{MR2538014} and Konyagin and Shkredov \cite{MR3488800} it is known to hold with $c = \frac{1}{3} + \varepsilon$ for some sufficiently small $\varepsilon$. Elekes and R\'onyai \cite{elekes2000combinatorial} generalized this by showing that for any polynomial $f(x,y) \in \mathbb{R}[x,y]$ there exists $c >0$ such that for every finite set $A \subseteq \mathbb{R}$ we have
$$|f(A \times A)|\geq |A|^{1+c},$$
 unless $f$ is either additive or multiplicative, i.e.~of the form $g(h(x) + i(y))$ or $g(h(x) \cdot i(y))$ for some univariate polynomials $g,h,i$ respectively. The bound was improved to $\Omega_{\deg f} \left( |A|^{\frac{11}{6}} \right)$ in \cite{raz}.

Elekes and Szab\'o \cite{ES} established a conceptual generalization of this result explaining the exceptional role played by the additive and multiplicative
forms: for any irreducible polynomial $Q(x,y,z)$ over $\mathbb{C}$ depending on all of its coordinates and such that its set zero set has dimension $2$, either there exists some $\varepsilon >0$ such that $F$ has at most $O(n^{2 - \varepsilon})$ zeroes on all finite $n\times n \times n$ grids, or $F$ is in a coordinate-wise finite-to-finite correspondence with the graph of multiplication of an algebraic group (see Theorem (B) below for a more precise statement). In the special Elekes-R\'onyai case above, taking $Q$ to be the graph of the polynomial function $f$, the resulting group is either the additive or the multiplicative group of the field.
 Several generalizations, refinements and variants of this influential result were obtained recently \cite{MR3577370, ES4d, jing2019minimal, MR4181764, bukh2012sum, tao2015expanding, hrushovski2013pseudo}, in particular for complex algebraic relations of higher dimension and arity by Bays and Breuillard \cite{Bays}.

In this paper we obtain a generalization of the Elekes-Szab\'o theorem to hypergraphs of any arity and dimension definable in stable structures admitting distal expansions (this class includes algebraically and differentially closed fields of characteristic $0$ and compact complex manifolds); as well as for arbitrary $o$-minimal structures.
Before explaining our general theorems, we state two very special corollaries.

\begin{thm*}[A](Corollary \ref{cor: 1-dim o-min})
Assume $s \geq 3$ and $Q \subseteq  \mathbb{R}^s$ is semi-algebraic, of description complexity $D$ (i.e.~given by at most $D$ polynomial (in-)equalities, with all polynomials of degree at most $D$, and $s \leq D$), such that the projection of $Q$ to any $s-1$ coordinates is finite-to-one. Then exactly one  of the
following  holds.
\begin{enumerate}
\item There exists a constant $c$, depending only on $s$ and $D$, such that: for any $n \in \mathbb{N}$ and  finite $A_i \subseteq \mathbb{R}$ with $|A_i| = n$ for $i \in [s]$ we have
$$|Q \cap (A_1 \times \ldots \times A_s)| \leq c n^{s - 1 - \gamma},$$
where $\gamma = \frac{1}{3}$ if $s \geq 4$, and $\gamma = \frac{1}{6}$ if $s=3$.

\item There exist open sets $U_i \subseteq \mathbb{R}, i \in [s]$, an open set $V \subseteq \mathbb{R}$ containing $0$, and analytic bijections with analytic inverses $\pi_i: U_i \to V$ such that
$$\pi_1(x_1)+\cdots+\pi_s(x_s)=0\Leftrightarrow Q(x_1,\ldots,x_s)$$
for all $x_i \in U_i, i \in [s]$.
\end{enumerate}
\end{thm*}

\begin{thm*}[B](Corollary \ref{cor: ES for 1 dim ACF})
	Assume $s \geq 3$, and let $Q \subseteq \mathbb{C}^{s}$  be an irreducible algebraic variety so that for each $i \in [s]$, the projection of $Q$ to any $s-1$ coordinates is generically finite.
	Then exactly one of the following holds.
	\begin{enumerate}
		\item There exist $c$ depending only on $s, \deg(Q)$ such that:
	 for any $n \in \mathbb{N}$ and $A_i \subseteq \CC_i$ with $|A_i| = n$ for $i \in [s]$ we have
	 $$|Q \cap (A_1 \times \ldots \times A_s)| \leq c n^{s-1 -\gamma}$$
	 where $\gamma = \frac{1}{11}$ if $s \geq 4$, and $\gamma =  \frac{1}{22}$ if $s=3$.
		\item For $G$ one of $(\mathbb{C}, +)$, $(\mathbb{C}, \times)$ or an elliptic curve group, $Q$ is in coordinate-wise correspondence (see Section \ref{sec: stab applics}) with
		$$Q' := \left\{(x_1, \ldots, x_s) \in G^s : x_1 \cdot \ldots \cdot x_s = 1_G \right\}.$$
	\end{enumerate}
	
\end{thm*}

\begin{rem}
Theorem (B) is similar to the codimension $1$ case of \cite[Theorem 1.4]{Bays}, however our method provides an explicit bound on the exponent in Clause (1).
\end{rem}

\begin{rem}
Theorems (A) and (B) correspond to the $1$-dimensional case of Corollaries \ref{thm:real o-min1} and \ref{cor: ES for any dim in ACF}, respectively, which allow $Q \subseteq \prod_{i \in [s]} X_i$ with $\dim(X_i) = d$ for an arbitrary $d \in \mathbb{N}$.
\end{rem}

\begin{rem}
	Note the important difference --- Theorem (A) is \emph{local}, i.e.~we can only obtain a correspondence of $Q$ to a subset of a group after restricting to \emph{some} open subsets $U_i$. This is unavoidable in an ordered structure since the high count in Theorem (A.2) might be the result of a local phenomenon in $Q$. E.g.~when $Q$ is the union of $Q_1=\{\bar x: x_1+\cdots+x_s=0\} \cap (-\varepsilon, \varepsilon)^s$, for some $\varepsilon > 0$, and another set  $Q_2$ for which the count is low.
\end{rem}

\subsection{The Elekes-Szab\'o principle}
We now describe the general setting of our main results. We let $\CM = (M, \ldots)$ be an arbitrary first-order structure, in the sense of model theory, i.e.~a set $M$ equipped with some distinguished functions and relations. As usual, a subset of $M^d$ is definable if it is the set of tuples satisfying a formula (with parameters). Two key examples to keep in mind are $(\mathbb{C}, +, \times, 0,1)$ (in which definable sets are exactly the constructible ones, i.e.~boolean combinations of the zero-sets of polynomials, by Tarski's quantifier elimination) and $(\mathbb{R}, +, \times, <, 0,1)$ (in which definable sets are exactly the semialgebraic ones, by Tarski-Seidenberg quantifier elimination). We refer to \cite{MR1924282} for an introduction to model theory and the details of the aforementioned quantifier elimination results.

From now on, we fix a structure $\mathcal{M}$,  $s \in \mathbb{N}$, definable sets $X_i \subseteq M^{d_i}, i \in [s]$,
and a definable relation $Q\subseteq\bar{X}=X_{1}\times\ldots\times X_{s}$.
We write $A_{i}\subseteq_{n}X_{i}$ if $A_{i}\subseteq X_{i}$ with $\left|A_{i}\right| \leq n$. By a \emph{grid on $\bar{X}$ }we mean a set $\bar{A}\subseteq\bar{X}$
with $\bar{A}=A_{1}\times\ldots\times A_{s}$ and $A_{i}\subseteq X_{i}$.
By an \emph{$n$-grid on $\bar{X}$} we mean a grid $\bar{A}=A_{1}\times\ldots\times A_{s}$
with $A_{i}\subseteq_{n}X_{i}$.

\begin{defn}\label{def: fiber alg}
For $d\in \NN$,   we say that  a relation $Q\subseteq X_1\times X_2 \times \dotsc \times X_s$
is \emph{fiber-algebraic, of degree $d$}  if  for any
$i\in [s]$ we have
\begin{gather*}
	\forall x_1 \in X_1 \dotsc \forall x_{i-1} \in X_{i-1} \forall x_{i+1} \in X_{i+1} \dotsc \forall x_s \in X_s \\
\exists^{\leq d}
  x_i \in X_i   \, \, (x_1,\dotsc,x_s)\in Q.
\end{gather*}
We say that $Q\subseteq X_1\times X_2 \times \dotsc \times X_s$ is 
\emph{fiber-algebraic} if it is fiber-algebraic of degree  $d$ for some
$d\in \NN$.
\end{defn}

In other words, fiber algebraicity means that the projection of $Q$ onto any $s-1$ coordinates is finite-to-one.
For example, if $Q\subseteq X_{1}\times X_{2}\times X_{3}$ is fiber-algebraic of degree $d$,
then for any $A_{i}\subseteq_{n}X_{i}$ we have $\left|Q\cap A_{1}\times A_{2}\times A_{3}\right| \leq d n^{2}$. Conversely, let $Q\subseteq\mathbb{C}^{3}$ be given by $x_{1}+x_{2}-x_{3}=0$, and let $A_{1}=A_{2}=A_{3}=\left\{ 0,\ldots,n-1\right\} $. Then $\left|Q\cap A_{1}\times A_{2}\times A_{3}\right|=\frac{n\left(n+1\right)}{2}=\Omega\left(n^{2}\right)$. This indicates that the upper and lower bounds match for the graph of addition in an abelian group (up to a constant) --- and the Elekes-Szab\'o principle suggests that in many situations this is the only possibility. Before making this precise, we introduce some notation.

\subsubsection{Grids in general position.}
From now on we will assume that $\mathcal{M}$ is equipped with some notion of integer-valued dimension on definable sets, to be specified later. A good example to keep in mind is Zariski dimension on constructible subsets of $\mathbb{C}^d$, or the topological dimension on semialgebraic subsets of $\RR^d$.
\begin{defn}\label{defn: gen pos intro}
	\begin{enumerate}
		\item Let  $X$ be a definable set in $\mathcal{M}$, and  let $\mathcal{F}$ be a definable family of subsets of $X$. For $\nu \in \mathbb{N}$, we say that a set $A\subseteq X$ is \emph{in $(\mathcal{F},\nu)$-general position}
if  $|A\cap F|\leq \nu$ for every $F\in \mathcal{F}$ with $\dim(F) < \dim(X)$.

\item Let $X_i$,  $i=1, \ldots ,s$,  be definable sets in $\mathcal{M}$.
Let $\bar{\mathcal{F}} = (\mathcal{F}_1, \ldots, \mathcal{F}_s)$, where $\mathcal{F}_i$ is a definable family of subsets of $X_i$.
For  $\nu \in \mathbb{N}$ we say
that a grid $\bar{A}$ on $\bar{X}$
is in
\emph{$(\bar{\mathcal{F}}, \nu)$-general position} if each $A_i$ is in
  $(\mathcal{F}_i,\nu)$-general position.
	\end{enumerate}
\end{defn}

For example, when $\CM$ is the field $\CC$, a subset of $\mathbb{C}^d$ is in a $(\CF, \nu)$-general position if any variety of smaller dimension and bounded degree (determined by the formula defining $\CF$) can cut out only $\nu$ points from it (see the proof of  Corollary \ref{cor: ES for any dim in ACF}). Also, if  $\mathcal{F}$  is any definable family of subsets of $\CC$, then   for any large
  enough $\nu$,  every  $A\subseteq X$  is in $(\mathcal{F},\nu)$-general position.  On the other hand, let $X = \mathbb{C}^2$ and let $\mathcal{F}_d$ be the family of algebraic curves of degree less than $d$. If $\nu \leq d+1$, then any set $A \subseteq X$ with $|A| \geq \nu$ is not in $(\mathcal{F}_d, \nu-1 )$-general position.

\subsubsection{Generic correspondence with group multiplication.}
We assume that $\CM$ is a sufficiently saturated structure, and let $Q \subseteq \bar{X}$ be a definable relation and $(G, \cdot, 1_G)$ a connected type-definable group in $\CM^{\textrm{eq}}$. Type-definability means that the underlying set $G$  of the group is given by the intersection of a small (but possibly infinite) collection of definable sets, and the multiplication and inverse operations are relatively definable. Such a group is connected if it contains no proper type-definable subgroup of small index (see e.g.~\cite[Chapter 7.5]{MR1924282}). And $\CM^{\eq}$ is the structure obtained from $\CM$ by adding sorts for the quotients of definable sets by definable equivalence relations in $\CM$ (see e.g.~\cite[Chapter 1.3]{MR1924282}). In the case when $\CM$ is the field $\mathbb{C}$, connected type-definable groups are essentially just the complex algebraic groups connected in the sense of Zariski topology (see Section \ref{sec: stab applics} for a discussion and further references).

\begin{defn}\label{def: gen corresp intro}
We say that $Q$ is in a \emph{generic correspondence with multiplication in $G$} if there exist a small set $A \subseteq M$ and elements $g_1, \ldots, g_s \in G$ such that:
\begin{enumerate}
\item   $g_1\cdot
  \dotsc\cdot  g_s = 1_G$;
 \item  $g_1,\dotsc,g_{s-1}$ are independent generics in $G$ over $A$ (i.e.~each $g_i$ does not belong to any definable set of dimension smaller than $G$ definable over $A \cup \{g_1, \ldots, g_{i-1}, g_{i+1}, \ldots, g_{s-1} \}$);
 \item  For each $i=1,\dotsc, s$ there is a generic element $a_i \in X_i$   inter-algebraic with $g_i$ over $\mathcal{A}$ (i.e.~$a_i \in \acl(g_i,A)$ and $g_i \in \acl(a_i,A)$, where $\acl$ is the model-theoretic algebraic closure),  such that $(a_1, \ldots, a_s) \in Q$.
\end{enumerate}
\end{defn}

\begin{rem}
	There are several variants of ``generic correspondence with a group'' considered in the literature around the Elekes-Szab\'o theorem. The one that we use arises naturally at the level of generality we work with, and as we discuss in Sections \ref{sec: stab applics} and \ref{sec: apps in o-min} it easily specializes to the notions considered previously in several cases of interest (e.g.~the algebraic coordinate-wise finite-to-finite correspondence in the case of constructible sets in Theorem (B), or coordinate-wise analytic bijections on a neighborhood in the case of semialgebraic sets in Theorem (A)).
\end{rem}

\subsubsection{The Elekes-Szab\'o principle}
Let $s \geq 3, k \in \mathbb{N}$ and $X_1, \ldots, X_s$ be definable sets in a sufficiently saturated structure $\mathcal{M}$ with $\dim(X_i) = k$.
\begin{defn}
We say that $X_1, \ldots, X_s$ satisfy the \emph{Elekes-Szab\'o principle} if for any fiber-algebraic definable relation $Q \subseteq \bar{X}$, one of the following holds:
\begin{enumerate}
	\item $Q$ \emph{admits power saving}: there exist some $\gamma \in \mathbb{R}_{>0}$ and some definable families $\mathcal{F}_i$ on $X_i$ such that: for any $\nu \in \mathbb{N}$ and any $n$-grid $\bar{A} \subseteq \bar{X}$ in $(\bar{F},\nu)$-general position, we have $|Q \cap \bar{A}| = O_{\nu} \left( n^{(s-1) - \gamma} \right)$;
	\item there exists a type-definable subset of $Q$ of full dimension that is in a generic correspondence with multiplication in some type-definable \emph{abelian} group $G$  of dimension $k$.
\end{enumerate}
\end{defn}

The following are the previously known cases of the Elekes-Szab\'o principle:
\begin{enumerate}
	\item \cite{ES} $\mathcal{M} = \left(\mathbb{C}, +, \times \right)$, $s=3$, $k$ arbitrary (no explicit exponent $\gamma$ in power saving; no abelianity of the algebraic group for $k > 1$);
	\item \cite{MR3577370} $\mathcal{M} = \left(\mathbb{C}, +, \times \right)$, $s = 3$, $k = 1$ (explicit $\gamma$ in power saving);

	\item \cite{ES4d} $\mathcal{M} =  \left(\mathbb{C}, +, \times \right)$, $s = 4$, $k = 1$ (explicit $\gamma$ in power saving);
	\item  \cite{MR4181764} $\mathcal{M} =  \left(\mathbb{C}, +, \times \right)$, $k = 1$, $Q$ is the graph of an  $s$-ary polynomial function for an arbitrary $s$ (i.e.~this is a generalization of Elekes-R\'onyai to an arbitrary number of variables);
	\item \cite{Bays} $\mathcal{M} = \left(\mathbb{C}, +, \times \right)$, $s$ and $k$ arbitrary, abelianity of the group for $k > 1$ (they work with a more relaxed notion of general position and arbitrary codimension, however no bounds on $\gamma$);
	\item \cite{StrMinES} $\mathcal{M}$ is any strongly minimal structure interpretable in a distal structure (see Section \ref{sec: distal Zarank}), $s = 3$, $k=1$.
\end{enumerate}
In the first five cases the dimension is the Zariski dimension, and in the sixth case the Morley rank.

\subsection{Main theorem}

We can now state the main result of this paper.
\begin{thm*}[C]
The Elekes-Szab\'o principle holds in the following two cases:
\begin{enumerate}
\item (Theorem \ref{thm: stab main ineff}) $\mathcal{M}$ is a stable structure interpretable in a distal structure, with respect to $\mathfrak{p}$-dimension (see Section \ref{sec:notion-mfp-dimension}, and below).
	\item (Theorem \ref{thm: main thm o-min1}) $\mathcal{M}$ is an $o$-minimal structure expanding a group, with respect to the topological dimension. In this case, on  a type-definable generic subset of $\bar{X}$, we get a definable coordinate-wise bijection of $Q$ with the graph of multiplication of an abelian type-definable group $G$ (we stress that this $G$ is a priori unrelated to the underlying group that $\mathcal{M}$ expands).
\end{enumerate}
Moreover, the power saving bound is explicit in (2) (see the statement of Theorem \ref{thm: main thm o-min1}), and is explicitly calculated from a given distal cell decomposition for $Q$ in (1) (see Theorem  \ref{rem: power saving bound, main stable}).
\end{thm*}

Examples of structures satisfying the assumption of Theorem (C.1) include: algebraically closed fields of characteristic $0$, differentially closed fields of characteristic $0$ with finitely many commuting derivations, compact complex manifolds. In particular, Theorem (B) follows from Theorem (C.1) with $k=1$, combined with some basic model theory of algebraically closed fields (see Section \ref{sec: stab applics}). We refer to \cite{pillay1996geometric} for a detailed treatment of stability, and to \cite[Chapter 8]{tent2012course} for a quick introduction. See Section \ref{sec: distal Zarank} for a discussion of distality.

Examples of $o$-minimal structures include real closed fields (in particular, Theorem (A) follows from Theorem (C.2) with $k=1$ combined with some basic $o$-minimality, see Section \ref{sec: apps in o-min}), $\mathbb{R}_{\operatorname{exp}} = (\mathbb{R}, +, \times, e^x)$, $\RR_{\operatorname{an}} = \left( \RR, +, \times, f \restriction_{[0,1]^k} \right)$ for $k \in \mathbb{N}$ and $f$ ranging over all functions real-analytic on some neighborhood of $[0,1]^k$, or the combination of both $\RR_{\operatorname{an, exp}}$. We refer to \cite{MR1633348} for a detailed treatment of $o$-minimality, or to \cite[Section 3]{MR3728313} and reference there for a quick introduction.

\begin{rem}
	The assumption that $\mathcal{M}$ is an $o$-minimal expansion of a group in Theorem (C.2) can be relaxed to the more general assumption that $\mathcal{M}$ is an $o$-minimal structure with \emph{definable Skolem functions} (see e.g.~\cite{dinis2022definable} for a detailed discussion of Skolem functions and related notions), but possibly with 
	 a weaker bound on the power saving exponent than the one stated in Theorem \ref{thm: main thm o-min1}.
	 Indeed, the $\gamma$ in the $\gamma$-power saving stated in Theorem \ref{thm: main thm o-min1} depends on $\gamma$ in the $\gamma$-ST property, and hence on $t = 2d_1 - 2$, in Fact \ref{o-min cutting}(2) --- the proof of which uses that $\mathcal{M}$ is an $o$-minimal expansion of a group. However, Fact \ref{o-min cutting}(2) is known to  hold in an arbitrary $o$-minimal structure with (at least) the weaker bound $t = 2d_1 - 1$ (see \cite[Theorem 4.1]{DistCellDecompBounds}).
	   To carry out the rest of the arguments in the proof of Theorem \ref{thm: main thm o-min1} in Section \ref{sec: main thm omin} we only use the existence of definable Skolem functions. Thus any $o$-minimal structure with definable Skolem functions satisfies the conclusion of Theorem \ref{thm: main thm o-min1} with $\gamma = \frac{1}{8m-3}$ if $s \geq 4$ and $\gamma = \frac{1}{16m - 6}$ if $s=3$.
	    \end{rem}


%
%

\subsection{Outline of the paper}
In this section we outline the structure of the paper, and highlight some of the key ingredients of the proof of the main theorem. The proofs of (1) and (2) in Theorem (C) have similar strategy at the general level, however there are considerable technical differences.
In each of the cases, the proof consists of the following key ingredients.
\begin{enumerate}
	\item Zarankiewicz-type bounds for distal relations (Section \ref{sec: distal Zarank}, used for both Theorem (C.1) and (C.2)).
	\item A higher arity generalization of the abelian group configuration theorem (Section \ref{sec: group config omin} for the $o$-minimal case Theorem (C.2), and Section \ref{sec: group config stable} for the stable case Theorem (C.1)).
	\item The dichotomy between an incidence configuration, in which case the bounds from (1) give power saving, and existence of a family of functions (or finite-to-finite correspondences) associated to $Q$ closed under generic composition, in which case a correspondence of $Q$ to an abelian group is obtained using (2). This is Section \ref{sec: main thm stable} for the stable case (C.1) and Section \ref{sec: main thm omin} for the $o$-minimal case (C.2).
\end{enumerate}

We provide some further details for each of these ingredients, and discuss some auxiliary results of independent interest.

\subsubsection{Zarankiewicz-type bounds for distal relations (Section \ref{sec: distal Zarank})}

Distal structures constitute a subclass of purely unstable NIP structures \cite{simon2013distal} that contains all $o$-minimal structures, various expansions of the field $\mathbb{Q}_p$, and many other valued fields and related structures \cite{DistValFields} (we refer to the introduction of \cite{distal} for a general discussion of distality in connection to combinatorics and references). Distality of a graph can be viewed as a strengthening of finiteness of its VC-dimension retaining stronger combinatorial properties of semialgebraic graphs.
In particular, it is demonstrated in \cite{chernikov2015externally, distal, chernikov2016cutting} that many of the results in semialgebraic incidence combinatorics generalize to relations definable in distal structures. In Section \ref{sec: distal Zarank} we discuss distality, in particular proving  the following generalized ``Szemer\'edi-Trotter'' theorem:

\begin{thm*}[D](Theorem \ref{thm: distal incidence bound})
For every $d \in \mathbb{N}, t \in \mathbb{N}_{\geq 2}$ and $c \in \mathbb{R}$ there exists some $C = C(d,t,c) \in \mathbb{R}$ satisfying the following.

Assume that $E \subseteq X \times Y$ admits a distal cell decomposition $\mathcal{T}$ such that $|\mathcal{T}(B)| \leq c |B|^t$ for all finite $B \subseteq Y$. Then, taking $\gamma_1 := \frac{(t-1)d}{td-1}, \gamma_2 := \frac{td-t}{td-1}$ we have: for all $\nu \in \mathbb{N}_{\geq 2}$ and $A\subseteq_m X,B\subseteq_n Y$ such that $E \cap (A \times B)$ is $K_{d,\nu}$-free,
$$\left|E \cap (A \times B)\right| \leq C \nu \left(m^{\gamma_1} n^{\gamma_2} + m + n \right).$$
\end{thm*}
In particular, if $E \subseteq U \times V$ is a binary relation definable in a distal structure and $E$ is $K_{s,2}$-free for some $s \in \mathbb{N}$,
then there is some $\gamma>0$ such that: for
all $A\subseteq_n U,B\subseteq_n V$
we have $\left|E \cap A \times B\right| = O( n^{\frac{3}{2}-\gamma} )$. The exponent strictly less that $\frac{3}{2}$ requires distality, and is strictly better than e.g.~the optimal bound $\Omega(n^{\frac{3}{2}})$ for the point-line incidence relation on the affine plane over a field of positive characteristic. In the proof of Theorem (C), we will see how this  $\gamma$  translates to the power saving exponent in the non-group case.
More precisely, for our analysis of the higher arity relation $Q$, we introduce the so-called \emph{$\gamma$-Szemer\'edi-Trotter property}, or \emph{$\gamma$-ST property} (Definition \ref{def: gamma ST property}),  capturing an iterated variant of Theorem (D), and show in Proposition \ref{prop: ind ES} that Theorem (D) implies that every binary relation definable in a distal structure satisfies the $\gamma$-ST property for some $\gamma >0$ calculated in terms of its distal cell decomposition.
We conclude Section \ref{sec: distal Zarank} with a discussion of the explicit bounds on $\gamma$ for the $\gamma$-ST property in several particular structures of interest needed to deduce the explicit bounds on the power saving in Theorems (A) and (B).

\subsubsection{Reconstructing an abelian group from a family of bijections (Section \ref{sec: group config omin})}

Assume that $(G, +, 0)$ is an abelian group, and consider the $s$-ary relation $Q \subseteq \prod_{i \in [s]}G$ given by $x_1 + \ldots + x_s = 0$. Then $Q$ is easily seen to satisfy the following two properties, for any permutation of the variables of $Q$:
\begin{gather*}
	\tag{P1} \forall x_1, \ldots, \forall x_{s-1} \exists ! x_s Q(x_1, \ldots, x_s), \\
	\tag{P2} \forall x_1, x_2 \forall y_3, \ldots y_{s} \forall y'_3, \ldots, y'_s \Big( Q(\bar{x}, \bar{y}) \land Q(\bar{x}, \bar{y}') \rightarrow \\
	\big( \forall x'_1, x'_2 Q(\bar{x}', \bar{y}) \leftrightarrow Q(\bar{x}', \bar{y}') \big) \Big).
\end{gather*}
In Section  \ref{sec: group config omin} we show a converse, assuming $s \geq 4$:
\begin{thm*}[E](Theorem \ref{thm: main group config bijections})
Assume $s \in \mathbb{N}_{\geq 4}$, $X_1, \ldots, X_s$ and $Q \subseteq \prod_{i \in [s]} X_i$ are sets, so that $Q$ satisfies (P1) and (P2) for any permutation of the variables.
Then there exists an abelian group $(G,+,0_G)$ and bijections $\pi_i: X_i \to G$ such that for every $(a_1, \ldots, a_s) \in \prod_{i \in [s]} X_i$ we have
$$Q(a_1, \ldots, a_s) \iff \pi_1(a_1) + \ldots + \pi_s(a_s) = 0_G.$$
Moreover, if $Q$ is definable and $X_i$ are type-definable in a sufficiently saturated structure $\CM$, then we can take $G$ to be type-definable and the bijections $\pi_i$ relatively definable in $\CM$.
\end{thm*}
On the one hand, this can be viewed as a purely combinatorial  higher arity variant of the Abelian Group Configuration theorem (see below) in the case when the definable closure in $\CM$ is equal to the algebraic closure (e.g.~when $\CM$ is $o$-minimal). On the other hand, if $X_1 = \ldots = X_s$, property (P1) is equivalent to saying that the relation $Q$ is an  \emph{$(s-1)$-dimensional permutation} on the set $X_1$, or a \emph{Latin $(s-1)$-hypercube}, as studied by Linial and Luria in \cite{linial2014upper, linial2016discrepancy} (where Latin $2$-hypercube is just a Latin square).
Thus the condition (P2) in Theorem (E) characterizes, for $s \geq 3$, those Latin $s$-hypercubes that are given by the relation ``$x_1 + \ldots + x_{s-1} = x_s$'' in an abelian group. We remark that for $s=2$ there is a known ``quadrangle condition'' due to Brandt characterizing those Latin squares that represent the multiplication table of a group, see e.g.~\cite[Proposition 1.4]{gowers2020partial}.

\subsubsection{Reconstructing a group from an abelian $s$-gon in stable structures (Section \ref{sec: group config stable})}
Here we consider a generalization of the group reconstruction method from a fiber-algebraic $Q$ of degree $1$ to a fiber-algebraic $Q$ of arbitrary degree, which moreover only satisfies (P2) \emph{generically}, and restricting to $Q$ definable in a stable structure.

Working in a stable theory, it is convenient to formulate this in the language of generic points. By an \emph{$s$-gon} over a set of parameters $A$ we mean a tuple $a_1, \ldots, a_s$ such that any $s-1$ of its elements are (forking-) independent over $A$,  and any element in it is in the algebraic closure of the other ones and $A$.
We say that an $s$-gon is \emph{abelian} if, after any permutation of its elements, we have
$$a_1 a_2 \ind_{\acl_A(a_1 a_2) \cap \acl_A(a_3 \ldots a_m)} a_3 \ldots a_m.$$ Note that this condition corresponds to the definition of a \emph{$1$-based}  stable theory, but localized to the elements of the $s$-gon.

If $(G, +)$ is a type-definable abelian group, $g_1, \ldots, g_{s-1}$ are independent generics in $G$ and $g_s := g_1 + \ldots + g_{s-1}$, then $g_1, \ldots, g_s$ is an abelian $s$-gon (associated to $G$). In Section \ref{sec: group config stable} we prove a converse:
\begin{thm*}[F](Theorem \ref{thm: main ab mgon gives grp})
Let $a_1, \ldots, a_s$ be an abelian $s$-gon, $s \geq 4$, in a sufficiently saturated stable structure $\CM$. Then there is a type-definable (in $\mathcal{M}^{\textrm{eq}}$) connected abelian group $(G, +)$ and an abelian $s$-gon $g_1, \ldots, g_s$ associated to $G$, such that after a base change each $g_i$ is inter-algebraic with $a_i$.
\end{thm*}

It is not hard to see that a $4$-gon is essentially equivalent to the usual abelian group configuration, so Theorem (F) is a higher arity generalization. After this work was completed, we have learned that independently Hrushovski obtained a similar (but incomparable) unpublished result \cite{HrushovskiUnpubl, HruOber}.

\subsubsection{Elekes-Szab\'o principle in stable structures with distal expansions --- proof of Theorem (C.1) (Section \ref{sec: main thm stable})}
We introduce and study the notion of \emph{$\mfp$-dimension} in Section \ref{sec:notion-mfp-dimension}, imitating the topological definition of dimension in $o$-minimal structures, but localized at a given tuple of commuting definable global types. Assume we are given \emph{$\mathfrak{p}$-pairs} $(X_i, \mathfrak{p}_i)$ for $1 \leq i \leq s$, i.e.~$X_i$ is an
$\mathcal{M}$-definable set and  $\mathfrak{p}_i \in S(\mathcal{M})$ is a complete stationary type on $X_i$ for each $1 \leq i \leq s$ (see Definition \ref{def: p-pairs}). We say that a definable set $Y \subseteq X_1 \times \ldots \times X_s$ is \emph{$\mathfrak{p}$-generic}, where $\mfp$ refers to the tuple $(\mfp_1, \ldots, \mfp_s)$, if $Y \in \left( \mathfrak{p}_1 \otimes \ldots \otimes \mathfrak{p}_s \right)|_{\mathbb{M}}$. Finally, we define the $\mathfrak{p}$-dimension via $\dim_{\mathfrak{p}}(Y) \geq k$ if for some projection $\pi$ of $\bar{X}$ onto $k$ components, $\pi(Y)$ is $\mathfrak{p}$-generic.
We show that $\mathfrak{p}$-dimension enjoys definability and additivity properties crucial for our arguments that may fail for Morley rank in general $\omega$-stable theories such as $\textrm{DCF}_0$. However, if $X$ is a definable subset of finite Morley rank $k$ and degree one, taking $\mathfrak{p}_X$ to be the unique type on $X$ of Morley rank $k$, we have that $k \cdot \dim_{\mathfrak{p}}  = \textrm{MR}$ (this is used to deduce Theorem (B) from Theorem (C.1)).

In Section \ref{sec: pirreducible} we consider the notion of irreducibility and show that every fiber-algebraic relation is a union of finitely many absolutely $\mfp$-irreducible sets.
In Section \ref{sec:general-position-1} we consider finite grids in general position with respect to $\mfp$-dimension and prove some preliminary power-saving bounds. In Section \ref{sec:setting-3} we state a more informative version of Theorem (C.1) (Theorem \ref{thm: stab main ineff} + Theorem  \ref{rem: power saving bound, main stable} concerning the bound $\gamma$ in power saving) and make some preliminary reductions.
In particular, we may assume $\dim(Q) = s-1$, and let $\bar{a} = (a_1, \ldots, a_s)$ be a generic tuple in $Q$. As $Q$ is fiber-algebraic, $\bar{a}$ is an $s$-gon. We then establish the following key structural dichotomy.
\begin{thm*}[G](Theorem \ref{thm:ser-main} and its proof)
	Assuming $s \geq 3$, one of the following is true:
	\begin{enumerate}
		\item For $u = (a_1,a_2)$ and $v=(a_3, \ldots, a_s)$ we have $u \ind_{\acl(u) \cap \acl(v)} v$.
		\item $Q$, as a binary relation on $U \times V$, for $U = X_1 \times X_2$ and $V=X_3 \times \ldots \times X_s$, is a ``pseudo-plane''. By which we mean here that, ignoring a smaller dimensional ($\dim_{\mathfrak{p}}<s-2$) set of $v \in V$, every fiber $Q_v \subseteq U$ has a zero-dimensional intersection $Q_v \cap Q_{v'}$ for all $v'\in V$ outside a smaller dimensional set (more precisely, the $\mfp$-dimension of the set $Z$ defined in terms of $Q$ in Section \ref{sec:proof-theor-ser-main} is $< s-2$).
	\end{enumerate}
\end{thm*}
This notion of a ``pseudo-plane'' generalizes the usual definition requiring that any two ``lines'' in $V$ intersect on finitely many ``points'' in $U$, viewing $Q$ as the incidence relation.

In case (2) the relation $Q$ satisfies the assumption on the intersection of its fibers in Definition \ref{def: gamma ST property}, hence the incidence bound from Theorem (D) can be applied inductively to obtain power saving for $Q$ (see Section \ref{sec:proof-theor-ser-main}).
Thus we may assume that for any permutation of $\{1, \ldots, s\}$ we have
$$a_1 a_2 \ind_{\acl(a_1 a_2) \cap \acl(a_3 \ldots a_s)} a_3 \ldots a_s,$$
i.e.~the $s$-gon $\bar{a}$ is abelian. Assuming that $s \geq 4$, Theorem (F) can be applied to establish a generic correspondence with a  type-definable abelian group (Section \ref{sec: MainThm for s geq 4}). The case $s=3$ of Theorem (C.1) is treated separately in Section \ref{sec: stable main thm ternary} by reducing it to the case $s=4$ (similar to the approach in \cite{MR3577370}).

In Section \ref{sec: stab applics} we combine Theorem (C.1) with some standard model theory of algebraically closed fields to deduce Theorem (B) and its higher dimensional version.

\subsubsection{Elekes-Szab\'o principle in $o$-minimal structures --- proof of Theorem (C.2) (Section \ref{sec: main thm omin})}

Our proof of the $o$-minimal case is overall similar to the stable case, but is independent from it. In Section \ref{sec: o-min main thm} we formulate a more informative version of Theorem (C.2) with explicit bounds on power saving (Theorem \ref{thm: main thm o-min1}) and reduce it to Theorem \ref{main} --- which is an  appropriate analog of Theorem (G): (1) either $Q$ is a ``pseudo-plane'', or (2) it contains a subset $Q^*$ of full dimension so that the property (P2) from Theorem (E) holds in a neighborhood of every point of $Q^*$. In Case (1), considered in Section \ref{sec: power saving case proof omin}, we show that $Q$ satisfies the required power saving using Theorem (D) (or rather, its refinement for $o$-minimal structures from Fact \ref{o-min cutting}). In Case (2), we show in Section \ref{sec: obt nice Q rel omin} that one can choose a generic tuple $(a_1, \ldots, a_s)$ in $Q$ and (type-definable) infinitesimal  neighborhoods $\mu_i$ of $a_i$ so that the relation $Q \cap (\mu_1 \times \ldots \times \mu_s)$ satisfies (P1) and (P2) from Theorem (E) --- applying it we obtain a generic correspondence with a type-definable abelian group, concluding the proof of Theorem (C.2) for $s \geq 4$. The case $s=3$ is reduced to $s = 4$ similarly to the stable case.

Finally, in Section \ref{sec: apps in o-min} we obtain a Corollary of Theorem (C.2) that holds in an arbitrary $o$-minimal structure, not necessarily a saturated one - replacing a type-definable group by a definable \emph{local} group (Theorem \ref{thm:local o-min1}). Combining this with the solution of the Hilbert's 5th problem for local groups \cite{goldbring2010hilbert} (in fact, only in the much easier abelian case, see Theorem 8.5 there), we can improve ``local group'' to a ``Lie group'' in the case when the underlying set of the $o$-minimal structure $\CM$ is $\mathbb{R}$ and deduce Theorem (A) and its higher dimensional analog (Theorem \ref{thm:real o-min1}, see also Remark \ref{rem: getting analytic}). We also observe that  for semi-linear relations, in the non-group case we have $(1-\varepsilon)$-power saving for any $\varepsilon > 0$ (Remark \ref{rem: optimal power save semilin}).
%
%

\subsection{Acknowledgements}
We are very grateful to the referees for their detailed
and insightful reports and many valuable suggestions on improving the paper.
We thank Saugata Basu, Martin Bays, Emmanuel Breuillard, Jim Freitag, Rahim Moosa, Tom Scanlon, Pierre Simon, Chieu-Minh Tran and Frank Wagner for some helpful conversations.
We thank Institut Henri Poincar\'e in Paris for its hospitality during
the ``Model theory, Combinatorics and valued fields'' term in the Spring trimester
of 2018.
Chernikov was partially supported by the NSF
CAREER grant DMS-1651321 and by a Simons Fellowship. He thanks Lior Pachter, Michael Kinyon, and Math Twitter for the motivation in the final effort of finishing this paper. Peterzil was partially supported by the Israel Science Foundation grant 290/19. Starchenko was partially supported by the NSF grant DMS-1800806. 
The results of this paper were announced in \cite{breuillard2021model}.

\section{Zarankiewicz-type bounds for distal relations}\label{sec: distal Zarank}

We begin by recalling some of the notions and results about distality and generalized ``incidence bounds'' for distal relations from \cite{chernikov2016cutting}, and refer to that article for further details. The following definition captures a combinatorial ``shadow'' of the existence of a nice topological cell decomposition (as e.g.~in $o$-minimal theories or in the $p$-adics).

 \begin{defn}\label{def: distal cell decomp}\cite[Section 2]{chernikov2016cutting} Let $X,Y$ be infinite sets, and $E\subseteq X\times Y$ a binary relation.
 \begin{enumerate}
 	\item Let $A\subseteq X$. For $b\in Y$, we say that $E_{b} = \left\{a \in X : (a,b) \in E \right\}$ \emph{crosses} $A$ if $E_{b}\cap A\neq\emptyset$ and $\left(X\setminus E_{b}\right)\cap A\neq\emptyset$.
 	\item A set $A\subseteq X$ is \emph{$E$-complete over $B\subseteq Y$} if $A$ is not crossed by any $E_{b}$ with $b\in B$.
 	\item A family $\mathcal{F}$ of subsets of $X$ is a \emph{cell decomposition for $E$ over $B\subseteq Y$} if $X\subseteq\bigcup\mathcal{F}$ and every $A\in\mathcal{F}$ is $E$-complete over $B$.
 	\item A \emph{cell decomposition for $E$} is a map $\mathcal{T}: B \mapsto \mathcal{T}(B)$ such that for each finite $B\subseteq Y$, $\mathcal{T}\left(B\right)$ is a cell decomposition for $E$ over $B$.
 	\item A cell decomposition $\mathcal{T}$ is \emph{distal} if there exist $k\in\mathbb{N}$ and a relation $D\subseteq X\times Y^{k}$ such that for all finite $B\subseteq Y$, $\mathcal{T}\left(B\right)=\{D_{\left(b_{1},\ldots,b_{k}\right)}:b_{1},\ldots,b_{k}\in B\mbox{ and }D_{\left(b_{1},\ldots,b_{k}\right)}\mbox{ is }E\mbox{-complete over }B\}$.
 	\item For $t \in \mathbb{R}_{>0}$, we say that a cell decomposition $\mathcal{T}$ has \emph{exponent $\leq t$} if there exists some $c \in \mathbb{R}_{>0}$ such that $|\mathcal{T}(B)| \leq c |B|^{t}$ for all finite sets $B \subseteq Y$.
 \end{enumerate}

 \begin{rem}
 	Note that if $\mathcal{T}$ is a distal cell decomposition, then it has exponent $\leq k$ for $k$ as in Definition \ref{def: distal cell decomp}(5).
 \end{rem}
 
 \begin{rem}\label{rem: dist cell decomp fibers}
Assume that the binary relation $E \subseteq X \times (Y \times Z)$ admits a distal cell decomposition $\mathcal{T}$ with $|\mathcal{T}(\hat{B})| \leq c |\hat{B}|^{t}$ for every finite $\hat{B} \subseteq Y \times Z$. Then for every $z \in Z$, the binary relation $E_z \subseteq X \times Y$ admits a distal cell decomposition $\mathcal{T}_z$ with $|\mathcal{T}_z(B)| \leq c |B|^{t}$ for all finite $B \subseteq Y$.
 \end{rem}
 \begin{proof}
 	Indeed, assume that $D \subseteq X \times (Y \times Z)^{k}$ is witnessing that $\mathcal{T}$ is distal, i.e.~for any finite $\hat{B} \subseteq Y \times Z$ we have 
 	$$\mathcal{T}\left(\hat{B}\right)=\{D_{\left(b_{1},\ldots,b_{k}\right)}:b_{1},\ldots,b_{k}\in \hat{B}\mbox{ and }D_{\left(b_{1},\ldots,b_{k}\right)}\mbox{ is }E\mbox{-complete over }\hat{B}\}.$$
 	 Fix $z \in Z$, and let 
 	$$D_z := \left\{ (x;y_1, \ldots, y_k) \in X \times Y^k : (x; y_1, z, \ldots, y_k,z) \in D \right\} \subseteq X \times Y^k.$$
 	Given a finite $B \subseteq Y$, we define $\mathcal{T}_z\left(B\right)$ as 
 	\begin{gather*}
 	\left\{\left(D_z\right)_{\left(b_{1},\ldots,b_{k}\right)}:b_{1},\ldots,b_{k}\in B \mbox{ and } \left(D_z \right)_{\left(b_{1},\ldots,b_{k}\right)}\mbox{ is }E_z\mbox{-complete over }B\right\}	.
 	\end{gather*}
 Then $\mathcal{T}_z\left(B\right) = \mathcal{T}\left(B \times \{z\}\right)$, hence $\mathcal{T}_z$ is a distal cell decomposition for $E_z$ and $|\mathcal{T}_z(B)| = |\mathcal{T}(B \times \{z\})| \leq   c |B|^{t}$.
 \end{proof}

Existence of ``strong honest definitions'' established in \cite{chernikov2015externally} shows that every relation definable in a distal structure admits a distal cell decomposition (of some exponent).
 	
 \end{defn}
\begin{fact}\label{fac: def in dis impl distal cell decomp}(see \cite[Fact 2.9]{chernikov2016cutting})
Assume that the relation $E$ is definable in a distal structure $\mathcal{M}$. Then $E$ admits a distal cell decomposition (of some exponent $t \in \mathbb{N}$).
Moreover, in this case the relation $D$ in Definition \ref{def: distal cell decomp}(5) is also definable in $\mathcal{M}$.
	
\end{fact}

The following definition abstracts from the notion of cuttings in incidence geometry (see the introduction of \cite{chernikov2016cutting} for an extended discussion).
\begin{defn}\label{def: r-cutting}
Let $X,Y$ be infinite sets, $E\subseteq X\times Y$.
We say that
  $E$ \emph{admits cuttings with exponent $t \in \mathbb{R}$} if there is some constant $c \in \mathbb{R}_{>0}$ satisfying the following. For any $B \subseteq Y$ with $\left| B \right| = n$ and any $r \in \mathbb{R}$ with $1 < r < n$ there are some sets $X_1, \ldots, X_s \subseteq X$ covering $X$ with $s \leq c r ^t$ and such that for each $i \in [s]$ there are at most $\frac{n}{r}$ elements $b \in B$ so that  $X_i$ is crossed by $E_b$.
  \end{defn}

In the case $r > n$ in Definition \ref{def: r-cutting}, an $r$-cutting is equivalent to a distal cell decomposition (sets in the covering are not crossed at all). And for $r$ varying between $1$ and $n$, $r$-cutting allows to control the trade-off between the number of cells in a covering and the number of times each cell is allowed to be crossed.

\begin{fact}(Distal cutting lemma, \cite[Theorem 3.2]{chernikov2016cutting})\label{fac: distal implies cutting}
Assume $E\subseteq X\times Y$ admits a distal cell decomposition $\mathcal{T}$ of exponent $ \leq t$. Then $E$ admits cuttings with exponent $ \leq t$ and with the constant coefficient depending only on $t$ and the constant coefficient of $\mathcal{T}$ (the latter is not stated there explicitly, but follows from the proof). Moreover, every set in this cutting is an intersection of at most two cells in $\mathcal{T}$.
\end{fact}

\begin{rem}
	We stress that in the Definition \ref{def: r-cutting} of an $r$-cutting, some of the fibers $E_b, b \in B$ might be equal to each other. This is stated correctly on page $2$ of the introduction of \cite{chernikov2016cutting}, but is ambiguous in \cite[Definition 3.1]{chernikov2016cutting} (the family $\mathcal{F}$ there is allowed to have repeated sets, so it is a multi-set of sets) and in the statement of \cite[Theorem 3.2]{chernikov2016cutting} (again, the family $\left\{ \varphi(M;a) : a \in H \right\}$ there should be viewed as a family of sets with repetitions --- this is how it is understood in the proof of Theorem 3.2 there).
\end{rem}

%
%


The next theorem can be viewed as an abstract variant of the Szemer\'edi-Trotter theorem. It generalizes (and strengthens) the incidence bound
due to Elekes and Szab\'o \cite[Theorem 9]{ES} to arbitrary graphs admitting a distal cell decomposition, and is crucial to obtain power saving in the non-group case of our main theorem. Our proof below closely follows the proof of \cite[Theorem 2.6]{StrMinES} (which in turn is a generalization of \cite[Theorem 3.2]{fox2014semi} and \cite[Theorem 4]{pach1992repeated})  making the dependence on $s$ explicit. We note that the fact that the bound in Theorem \ref{thm: distal incidence bound} is sub-linear in $s$ was first observed in a special case in \cite{shefferincidence}.

As usual, given $d,\nu \in \mathbb{N}$ we say that a bipartite graph $E \subseteq U \times V$ is \emph{$K_{d,\nu}$-free} if it does not contain a copy of the complete bipartite graph $K_{d,\nu}$ with its parts of size $d$ and $\nu$, respectively.

\begin{thm}
\label{thm: distal incidence bound}
For every $d , t \in \mathbb{N}_{\geq 2}$ and $c \in \mathbb{R}_{>0}$ there exists some $C = C(d,t,c) \in \mathbb{R}$ satisfying the following.

Assume that $E \subseteq X \times Y$ admits a distal cell decomposition $\mathcal{T}$ such that $|\mathcal{T}(B)| \leq c |B|^t$ for all finite $B \subseteq Y$. Then, taking $\gamma_1 := \frac{(t-1)d}{td-1}, \gamma_2 := \frac{td-t}{td-1}$ we have: for all $\nu \in \mathbb{N}_{\geq 2}$ and $A\subseteq_m X,B\subseteq_n Y$ such that $E \cap (A \times B)$ is $K_{d,\nu}$-free,
$$\left|E \cap (A \times B)\right| \leq C \nu \left(m^{\gamma_1} n^{\gamma_2} + m + n \right).$$
\end{thm}
\noindent Before giving its proof we recall a couple of weaker general bounds  that will be used. First, a classical fact from \cite{kovari1954problem} giving a bound on the number of edges in $K_{d,\nu}$-free graphs without any additional assumptions (see e.g.~\cite[Chapter VI.2, Theorem 2.2]{MR2078877} for the stated version):
\begin{fact}\label{fac: Kovari} Assume $E \subseteq A\times B$ is $K_{d, \nu}$-free, for some $d,\nu \in \mathbb{N}_{\geq 1}$ and $A,B$ finite. Then $|E \cap A \times B| \leq \nu^{\frac{1}{d}} |A||B|^{1 - \frac{1}{d}} + d |B|$.
\end{fact}

Given a set $Y$ and a family $\CF$ of subsets of $Y$, the \emph{shatter function} $\pi_{\CF}: \mathbb{N} \to \mathbb{N}$ of $\CF$ is defined as
$$ \pi_{\CF}(z) := \max \{ |\CF \cap B| : B \subseteq Y, |B| = z \}, $$
where $\CF \cap B = \{ S \cap B : S \in \CF \}$.

 Second, the following bound for graphs of bounded VC-density is only stated in \cite{fox2014semi} for $K_{d,\nu}$-free graphs with $d=\nu$ (and with the sides of the bipartite graph exchanged), but the more general statement below, as well as the linear dependence of the bound on $\nu$, follow from its proof.
\begin{fact} \label{VCBoundOnEdges}
\cite[Theorem 2.1]{fox2014semi} For every $c \in \mathbb{R}$ and  $t, d  \in \mathbb{N}$ there is some constant $C = C(c,t,d)$ such that the following holds.

  Let $E \subseteq X \times Y$ be a bipartite graph such that the family $\CF = \{ E_a : a \in X \}$ of subsets of $Y$ satisfies $\pi_{\mathcal{F}}(z) \leq c z^t$ for all $z \in \mathbb{N}$ (where $E_a = \{ b \in Y : (a,b) \in E \}$). Then, for any $A \subseteq_m X, B \subseteq_n Y$ so that $E \cap (A \times B)$ is $K_{d, \nu}$-free, we have
$$|E \cap (A \times B)| \leq C \nu (m^{1- \frac{1}{t}} n +m).$$
\end{fact}

\begin{rem}\label{rem: distal decomp bounds dual VC density}
If $E \subseteq X \times Y$ admits a distal cell decomposition $\mathcal{T}$ with $|\mathcal{T}(B)| \leq c |B|^t$ for all $B \subseteq Y$, then for $\CF = \{ E_a : a \in X \}$ we have $\pi_{\mathcal{F}}(z) \leq c z^t$ for all $z \in \mathbb{N}$.

Indeed, by Definition \ref{def: distal cell decomp}, given any finite $B \subseteq Y$ and $\Delta \in \mathcal{T}(B)$, $B \cap E_a = B \cap E_{a'}$ for any $a,a' \in \Delta$ (and the sets in $\mathcal{T}(B)$ give a covering of $X$), hence at most $|\mathcal{T}(B)|$ different subsets of $B$ are cut out by the fibers of $E$.
\end{rem}

\begin{proof}[Proof of Theorem \ref{thm: distal incidence bound}]
Let $A \subseteq_m X, B \subseteq_n Y$ so that $E \cap (A \times B)$ is $K_{d, \nu}$-free be given.

If $n \geq m^d$, then by Fact \ref{fac: Kovari}  we have
\begin{gather}
|E \cap (A \times B)| \leq \nu^{\frac{1}{d}}m n^{1 - \frac{1}{d}}  + d n \leq d \nu (n^{\frac{1}{d}} n^{1 - \frac{1}{d}} + n)  =  2 d \nu n. \label{eq: Zar1}	
\end{gather}
Hence we assume $n < m^d$ from now on.

Let $r := \frac{m^{\frac{d}{td-1}}}{n^{\frac{1}{td-1}}}$ (note that $r > 1$ as $m^d > n$), and consider the family $\Sigma = \left( E_b : b \in B \right)$ of subsets of $X$ (some of the sets in it might be repeated).
By assumption and Fact \ref{fac: distal implies cutting}, there is a family $\mathcal{C}$ of subsets of $X$ giving a $\frac{1}{r}$-cutting for the family $\Sigma$. That is, $X$ is covered by the union of the sets in $\mathcal{C}$, any of the sets $C \in \mathcal{C}$ is crossed by at most $|B|/r$ elements from $\Sigma$, and $|\mathcal{C}| \leq \alpha_1 r^t$ for some $\alpha_1 = \alpha_1(c,t)$.

Then there is a set $C \in \mathcal{C}$ containing at least $\frac{m}{\alpha_1 r^t} = \frac{n^{\frac{t}{td-1}}}{\alpha_1 m ^{\frac{1}{td-1}}}$ points from $A$. Let $A' \subseteq A \cap C$ be a subset of size exactly $\left\lceil \frac{n^{\frac{t}{td-1}}}{\alpha_1 m ^{\frac{1}{td-1}}} \right\rceil$.

If $|A'| \leq d$, we have $\frac{n^{\frac{t}{td-1}}}{\alpha_1 m ^{\frac{1}{td-1}}} \leq |A'| \leq d$, so $n \leq d^{\frac{td-1}{t}} \alpha_1^{\frac{td-1}{t}}m^{\frac{1}{t}}$.
By assumption, Remark \ref{rem: distal decomp bounds dual VC density} and Fact   \ref{VCBoundOnEdges}, for some $\alpha_2 = \alpha_2(c,t,d)$ we have
$$|E \cap (A \times B)| \leq \alpha_2 \nu (n m^{1 - \frac{1}{t}} + m) \leq \alpha_2 \nu (d^{\frac{td-1}{t}} \alpha_1^{\frac{td-1}{t}}m^{\frac{1}{t}} m^{1 - \frac{1}{t}} + m),$$
hence
\begin{gather}
|E \cap (A \times B)| \leq \alpha_3 \nu m	\textrm{ for some  }\alpha_3 = \alpha_3 (c, t,d).\label{eq: Zar2}	
\end{gather}

Hence from now on we assume that $|A'| > d$.
Let $B'$ be the set of all points $b \in B$ such that $E_b$ crosses $C$. We know that
$$|B'| \leq \frac{|B|}{r} \leq  \frac{n n^{\frac{1}{td-1}}}{m^{\frac{d}{td-1}}} = \frac{n^{\frac{td}{td-1}}}{m^{\frac{d}{td-1}}} \leq \alpha_1^d |A'|^d.$$
Again by Fact \ref{fac: Kovari} we get
\begin{gather*}
	|E \cap (A' \times B')| \leq   d\nu (|A'| |B'|^{1-\frac{1}{d}} + |B'|)\\
	\leq d\nu (|A'|\alpha_1^{d-1} |A'|^{d-1} + \alpha_1^d |A'|^d) \leq \alpha_4 \nu |A'|^d
\end{gather*}
for some $\alpha_4 = \alpha_4(c,t,d)$. Hence there is a point $a \in A'$ such that $|E_a \cap B'| \leq \alpha_4 \nu |A'|^{d-1}$.

Since $E \cap (A \times B)$ is $K_{d,\nu}$-free, there are at most $\nu-1$ points in $B\setminus B'$ from $E_a$ (otherwise, since none of the sets $E_b, b \in B \setminus B'$ crosses $C$ and $C$ contains $A'$, which is of size $\geq d$, we would have a copy of $K_{d,\nu}$). And we have $|A'| \leq  \frac{n^{\frac{t}{td-1}}}{\alpha_1 m ^{\frac{1}{td-1}}} + 1  \leq \frac{2}{\alpha_1} \frac{n^{\frac{t}{td-1}}}{ m ^{\frac{1}{td-1}}}$ as $|A|' > d \geq 1$. Hence
\begin{gather*}
	|E_a \cap B| \leq |E_a \cap B'| + |E_a \cap (B \setminus B')| \leq  \alpha_4 \nu |A'|^{d-1} + (\nu-1)\\
	 \leq \frac{\alpha_4 2^{d-1}}{\alpha_1^{d-1}} \nu \frac{n^{\frac{t(d-1)}{td-1}}}{m^{\frac{d-1}{td-1}}} + (\nu-1) \leq \alpha_5 \nu \frac{n^{\frac{t(d-1)}{td-1}}}{m^{\frac{d-1}{td-1}}} + (\nu-1)
\end{gather*}
for some $\alpha_5 := \alpha_5(c,t,d)$. We remove $a$ and repeat the argument until \eqref{eq: Zar1}	or \eqref{eq: Zar2}	applies. This shows:
%
\begin{gather*}
|E \cap (A \times B)| \leq (2d\nu + \alpha_3 \nu)(n+m) + \sum_{i=n^{\frac{1}{d}}}^{m} \left( \alpha_5 \nu \frac{n^{\frac{t(d-1)}{td-1}}}{i^{\frac{d-1}{td-1}}} + (\nu-1) \right)\\
\leq (2 d + \alpha_3)\nu(n+m) + \alpha_5 \nu n^{\frac{t(d-1)}{td-1}} \sum_{i=n^{\frac{1}{d}}}^{m} \frac{1}{i^{\frac{d-1}{td-1}}} + (\nu-1) m .
\end{gather*}
Note that
\begin{gather*}
	\sum_{i=n^{\frac{1}{d}}}^{m} \frac{1}{i^{\frac{d-1}{td-1}}} \leq \int_{n^{\frac{1}{d}} -1}^{m} \frac{dx}{x^{\frac{d-1}{td-1}}} = \frac{m^{1 - \frac{d-1}{td-1}}}{1-\frac{d-1}{td-1}} - \frac{\left(n^{\frac{1}{d}} -1 \right)^{1 - \frac{d-1}{td-1}}}{1 - \frac{d-1}{td-1}}\\
	\leq \frac{td-1}{(t-1)d} m^{1 - \frac{d-1}{td-1}}
\end{gather*}
using $d,t \geq 2$ and that the second term is non-negative for $n \geq 1$.

Taking $C := 3 \max \{2d+\alpha_3, \frac{td-1}{(t-1)d} \alpha_5 \}$ --- which only depends on $c,t,d$ --- we thus have
$$|E \cap (A \times B)| \leq \frac{C}{3} \nu (n+m) + \frac{C}{3}  \nu n^{\frac{t(d-1)}{td-1}} m^{1-\frac{d-1}{td-1}} + \frac{C}{3} \nu m$$
$$ \leq C \nu ( m^{\frac{(t-1)d}{td-1}} n^{\frac{td-t}{td-1}} + m + n)$$
for all $m,n$.
\end{proof}

For our applications to hypergraphs, we will need to consider a certain iterated variant of the bound in Theorem \ref{thm: distal incidence bound}.

\begin{defn}\label{def: gamma ST property}
Let $\mathcal{E}$ be a family of subsets of $X \times Y$ and $\gamma \in \mathbb{R}$. We say that $\mathcal{E}$ satisfies the \emph{$\gamma$-Szemer\'edi-Trotter property}, or \emph{$\gamma$-ST property}, if for any function $C: \mathbb{N} \to \mathbb{N}_{\geq 1}$ there exists a function $C': \mathbb{N} \to \mathbb{N}_{\geq 1}$ so that: for every $E \in \mathcal{E}$, $s \in \mathbb{N}_{\geq 4}, \nu \in \mathbb{N}_{\geq 2}, n \in \mathbb{N}$ and $A \subseteq X, B \subseteq Y$ with $|A| \leq n^{s-2}, |B| \leq n^2$, if for every $a \in A$ there are at most $C(\nu) n^{s-4}$ elements $a' \in A$ with $|E_a \cap E_{a'} \cap B| \geq \nu$, then  $|E \cap (A \times B)| \leq  C'(\nu) n^{(s-1) - \gamma}$.

We say that a relation $E \subseteq X \times Y$ satisfies the \emph{$\gamma$-ST property} if the family $\mathcal{E} := \{E\}$ does.
\end{defn}

\begin{lemma}\label{lem: gamma-ST prop props}
	Assume that $\mathcal{E}$ is a family of subsets of $X \times Y$ and $\gamma \in \mathbb{R}$.
	\begin{enumerate}
		\item Assume that $X',Y'$ are some sets and $f: X \to X', g: Y \to Y'$ are bijections. For $E \in \mathcal{E}$, let $E' := \left\{(x,y) \in X' \times Y' : \left(f^{-1}(x), g^{-1}(y) \right) \in E\right\}$, and let $\mathcal{E}' := \left\{E' : E \in \mathcal{E} \right\}$, a family of subsets of $X' \times Y'$. Then $\mathcal{E}$ satisfies the $\gamma$-ST property if and only if $\mathcal{E}'$ satisfies the $\gamma$-ST property.
		\item Assume that  for some $k, \ell \in \mathbb{N}$ we have $X = \bigsqcup_{i \in [k]}X_i, Y = \bigsqcup_{j \in [\ell]} Y_i$, and let  $E_{i,j}  := E \cap (X_i \times Y_j)$, $\mathcal{E}_{i,j} := \left\{ E_{i,j} : E \in \mathcal{E} \right\}$. Assume that each $\mathcal{E}_{i,j}$ satisfies the $\gamma$-ST property. Then $\mathcal{E}$ also satisfies the $\gamma$-ST property.
	\end{enumerate}
\end{lemma}
\begin{proof}
	(1) is immediate from the definition. In (2), given $C: \mathbb{N} \to \mathbb{N}_{\geq 1}$, assume $C'_{i,j}: \mathbb{N} \to \mathbb{N}_{\geq 1}$ witnesses that $\mathcal{E}_{i,j}$ satisfies the $\gamma$-ST property. Then $C' := \sum_{(i,j) \in [k] \times [\ell]} C'_{i,j}$ witnesses that $\mathcal{E}$ satisfies the $\gamma$-ST property.
\end{proof}

\begin{lem}\label{prop: Holder iteration}
	Assume that $\mathcal{E} \subseteq \mathcal{P} \left( X \times Y \right)$, $\gamma_1, \gamma_2 \in \mathbb{R}_{>0}$ with $\gamma_1, \gamma_2 \leq 1, \gamma_1 + \gamma_2 \geq 1$ and $C_0: \mathbb{N} \to \mathbb{R}$ satisfy:
	\begin{itemize}
	\item[$(*)$] for every $E \in \mathcal{E}$, $\nu \in \mathbb{N}_{\geq 2}$ and finite $A \subseteq_m X, B \subseteq_n Y$, if $E \cap (A \times B)$ is $K_{2,\nu}$-free, then $|E \cap (A \times B)| \leq C_0(\nu) (m^{\gamma_1} n^{\gamma_2} + m + n)$.
	\end{itemize}

	Then $\mathcal{E}$ satisfies the $\gamma$-ST property with $\gamma := 3 - 2(\gamma_1 + \gamma_2) \leq 1$ and $C'(\nu) := 2 C_0(\nu) (C(\nu) + 2)$.
	\end{lem}

\begin{proof}

Given $E \in \mathcal{E}$ and finite sets $A,B$ satisfying the assumption of the $\gamma$-ST property, we consider the finite graph with the vertex set $A$ and the edge relation $R$ defined by $aRa' \iff |E_a \cap E_{a'} \cap B| \geq \nu$ for all $a,a'\in A$. By the assumption of the $\gamma$-ST property, this graph has degree at most $r := C(\nu) n^{s-4}$, so it is $(r+1)$-colorable by a standard fact in graph theory. For each $ i \in  [r+1]$, let $A_i \subseteq A$ be the set of vertices corresponding to the $i$th color. Then the sets $A_i$ give a partition of $A$, and for each $ i \in [r+1]$ the restriction of $E$ to $A_i \times B$ is $K_{2, \nu}$-free.

 For any fixed $i$, applying the assumption on $E$ to $A_i \times B$, we have
$$|E \cap (A_i \times B)| \leq C_0(\nu) \left( |A_i|^{\gamma_1} |B|^{\gamma_2} + |A_i| + |B| \right).$$
Then we have
\begin{gather}
	|E \cap (A \times B)| \leq \sum_{i \in [r+1]} |E \cap (A_i \times B)| \nonumber \\
	\leq \sum_{i \in [r+1]}C_0(\nu) \left( |A_i|^{\gamma_1} |B|^{\gamma_2} + |A_i| + |B|  \right) \nonumber\\
	 \leq C_0(\nu) \left( \sum_{i \in  [r+1]}  |A_i|^{\gamma_1} |B|^{\gamma_2} + \sum_{i \in  [r+1]}|A_i| + \sum_{i \in  [r+1]}|B| \right). \label{eq: Zar3}	
\end{gather}

For the first sum, applying H\"older's inequality with $p= \frac{1}{\gamma_1}$, we have
\begin{gather*}
	\sum_{i \in  [r+1]} |A_i|^{\gamma_1} |B|^{\gamma_2}  = |B|^{\gamma_2} \sum_{i \in  [r+1]} |A_i|^{\gamma_1}\\
	 \leq |B|^{\gamma_2} \left( \sum_{i \in  [r+1]}|A_i| \right)^{\gamma_1} \left( \sum_{i \in [r+1]}1 \right)^{1 - \gamma_1}\\
	 = |B|^{\gamma_2} |A|^{\gamma_1} (r+1)^{1 - \gamma_1} \leq n^{2 \gamma_2} n^{(s-2)\gamma_1} \left( C(\nu) n^{s-4} + 1 \right)^{1 - \gamma_1}\\
	 \leq n^{2 \gamma_2} n^{(s-2)\gamma_1} \left( C(\nu) +1  \right)^{1 - \gamma_1} n^{(s-4)(1 - \gamma_1)} \\
	  \leq ( C(\nu) + 1) n^{(s-4) + 2(\gamma_1 + \gamma_2)} = ( C(\nu) + 1)  n^{(s-1) - \gamma}
\end{gather*}
for all $n$ (by definition of $\gamma$ and as $s \geq 4, C(\nu) \geq 1, 0< \gamma_1 \leq 1$).

For the second sum, we have
$$ \sum_{i \in [r+1]}|A_i| = |A| \leq n^{s-2}$$
for all $n$. For the third sum we have
$$\sum_{i \in  [r+1]}|B| \leq (r+1)|B| \leq (C({\nu}) n^{s-4} + 1) n^2 \leq (C(\nu)+1) n^{s-2}$$
for all $n$.
Substituting these bounds into \eqref{eq: Zar3}, as $\gamma \leq 1$ we get
$$|E \cap (A \times B)| \leq 2 C_0(\nu) ( C(\nu) + 2) n^{(s-1)-\gamma}. \qedhere$$
\end{proof}
We note that the $\gamma$-ST property is non-trivial only if $\gamma >0$. Lemma \ref{prop: Holder iteration} shows that if $\mathcal{E}$ satisfies the condition in Lemma \ref{prop: Holder iteration}$(*)$ with   $\gamma_1 + \gamma_2 < \frac{3}{2}$, then $\mathcal{E}$ satisfies the $\gamma$-ST property for some $\gamma > 0$. By Theorem \ref{thm: distal incidence bound} this condition on $\gamma_1 + \gamma_2$ is satisfied for any relation admitting a distal cell decomposition, leading to the following.

\begin{prop}\label{prop: ind ES}
\begin{enumerate}
	\item Assume that $t \in \mathbb{N}_{\geq 2}$ and $E \subseteq X \times Y$ admits a distal cell decomposition $\mathcal{T}$ such that $|\mathcal{T}(B)| \leq c |B|^t$ for all finite $B \subseteq Y$. Then $E$ satisfies the $\gamma$-ST property  with $\gamma := \frac{1}{2t-1} > 0 $ and $C': \mathbb{N} \to \mathbb{N}_{\geq 1}$ depending only on $t,c,C$.
	\item In particular, if the binary relation $E \subseteq X \times (Y \times Z)$ admits a distal cell decomposition $\mathcal{T}$ of exponent $t$, then the family of fibers 
$$\mathcal{E} := \left\{ E_z \subseteq X \times Y : z \in Z \right\} \subseteq \mathcal{P}(X \times Y)$$
 satisfies the $\gamma$-ST property $\gamma := \frac{1}{2t-1}$.
\end{enumerate}

\end{prop}

\begin{proof}

(1)	By assumption and Theorem \ref{thm: distal incidence bound} with $d := 2$, there exists some $c' = c'(t,c) \in \mathbb{R}$ such that, taking $\gamma_1 := \frac{2t-2}{2t-1}, \gamma_2 := \frac{t}{2t-1}$, for all $\nu \in \mathbb{N}_{\geq 2}, m,n \in \mathbb{N}$ and $A\subseteq_m X,B\subseteq_n Y$ with $E \cap (A \times B)$ is $K_{2,\nu}$-free we have
$$\left|E \cap (A \times B)\right| \leq c' \nu \left(m^{\gamma_1} n^{\gamma_2} + m + n \right).$$
Then, by Lemma \ref{prop: Holder iteration}, $E$ satisfies the $\gamma$-ST property $\gamma := 3 - 2(\gamma_1 + \gamma_2) = 3 - 2 \frac{3t-2}{2t-1} = \frac{1}{2t-1} > 0$ and $C'(\nu) := 2 c' \nu (C(\nu) + 2)$.

(2) Combining (1) and Remark \ref{rem: dist cell decomp fibers}.
\end{proof}

The $\gamma$ in  Proposition \ref{prop: ind ES} will correspond to the power saving in the main theorem.
Stronger upper bounds on $\gamma_1, \gamma_2$ in Lemma \ref{prop: Holder iteration}$(*)$ (than the ones given by Theorem \ref{thm: distal incidence bound}) are known in some particular distal structures of interest and can be used to improve the bound on $\gamma$ in Proposition \ref{prop: ind ES}, and hence in the main theorem. We summarize some of these results relevant for our applications.

\begin{fact} \label{o-min cutting}Let $\mathcal{M} = (M, <, \ldots)$ be an $o$-minimal expansion of a group.
\begin{enumerate}
\item 
Let $\mathcal{E}$ be a definable family of subsets of $M^2 \times M^{d_2}, d_2 \in \mathbb{N}$, i.e.~$\mathcal{E} = \{E_b : b \in Z\}$ for some $d_3 \in \mathbb{N}$ and definable sets $E \subseteq M^{2} \times M^{d_2} \times M^{d_3}, Z \subseteq M^{d_3}$. 
The definable relation $E$ viewed as a binary relation on $M^2 \times M^{d_2+d_3}$ admits a distal cell decomposition with exponent $t=2$ by \cite[Theorem 4.1]{chernikov2016cutting}. 
Then Proposition \ref{prop: ind ES}(2) implies that $\mathcal{E}$ satisfies the $\gamma$-ST property with $\gamma := \frac{1}{3}$. (See also \cite{basu2017minimal} for an alternative approach.)

\item For general $d_1,d_2 \in \mathbb{N}_{\geq 2}$, every definable relation $E \subseteq M^{d_1}\times M^{d_2 + d_3}$ admits a distal cell decomposition with exponent $t=2d_1-2$ by \cite{DistCellDecompBounds} (this improves on the weaker bound in \cite[Section 4]{barone2013some} and generalizes the semialgebraic case in \cite{chazelle1991singly}). As in (1), Proposition \ref{prop: ind ES}(2) implies that any definable family $\mathcal{E}$ of subsets of $M^{d_1} \times M^{d_2}$ satisfies the $\gamma$-ST property with $\gamma := \frac{1}{4d_1-5}$.
\end{enumerate}
\end{fact}
In particular this implies the following bounds for semialgebraic and constructible sets of bounded description complexity:
\begin{cor}\label{fac: algebraic ST}
\begin{enumerate}
\item If $d_1,d_2, D \in \mathbb{N}_{\geq 2}, $ and $\mathcal{E}_D$ is the family of semialgebraic subsets of $ \mathbb{R}^{d_1} \times \mathbb{R}^{d_2}$ of description complexity $D$ (i.e.~every $E \in \mathcal{E}$ is defined by a Boolean combination of at most $D$ polynomial (in-)equalities with real coefficients, with all polynomials of degree at most $D$), then $\mathcal{E}_D$ satisfies the $\gamma$-ST property with $\gamma := \frac{1}{4d_1-5}$ (noting that for a fixed $D$, the family $\mathcal{E}_D$ is definable in the $o$-minimal structure $\left(\mathbb{R}, + ,\times, < \right)$ and  using Fact \ref{o-min cutting}(2)).

\item If $d_1,d_2, D \in \mathbb{N}_{\geq 2}$ and $\mathcal{E}_D$ is the family of constructible subsets of $ \mathbb{C}^{d_1} \times \mathbb{C}^{d_2}$ of description complexity $D$  (i.e.~every $E \in \mathcal{E}$ is defined by a Boolean combination of at most $D$ polynomial equations with complex coefficients, with all polynomials of degree at most $D$), then $\mathcal{E}_D$ satisfies the $\gamma$-ST property with $\gamma := \frac{1}{8d_1-5}$ (noting that for a fixed $D$, every $E \in \mathcal{E}_{D}$ can be viewed as a constructible, and hence semialgebraic, subset of $\mathbb{R}^{2d_1} \times \mathbb{R}^{2d_2}$ of description complexity $D$, and using (1)).
%
%
\end{enumerate}
\end{cor}
We note that a stronger bound is known for algebraic sets over $\mathbb{R}$ and $\mathbb{C}$, however in the proof of the main theorem over $\mathbb{C}$ we require a bound for more general families of constructible sets:
\begin{fact}\label{fac: stronger bounds for algebraic}
\begin{enumerate}
\item (\cite[Theorem 1.2]{fox2014semi}, \cite[Corollary 1.7]{walsh2018polynomial}) If $d_1,d_2 \in \mathbb{N}_{\geq 2}$ and $E \subseteq \mathbb{R}^{d_1} \times \mathbb{R}^{d_2}$ is algebraic with each $E_b, b \in \mathbb{R}^{d_2}$ an algebraic variety of degree $D$ in $\mathbb{R}^{d_1}$, then $E$ satisfies the condition in Lemma \ref{prop: Holder iteration}$(*)$ with $\gamma_1 = \frac{2(d_1-1)}{2d_1-1}, \gamma_2 = \frac{d_1}{2d_1-1}$ and some function $C_0$ depending on $d_2, D$. Hence,  by Lemma \ref{prop: Holder iteration}, $E$ satisfies the $\gamma$-ST property with $\gamma := \frac{1}{2d_1-1}$.
\item If $d_1,d_2 \in \mathbb{N}_{\geq 2}$ and $E \subseteq \mathbb{C}^{d_1} \times \mathbb{C}^{d_2}$ is algebraic with each $E_b, b \in \mathbb{C}^{d_2}$ an algebraic variety of degree $D$, it can be viewed as an algebraic subset of $\mathbb{R}^{2d_1} \times \mathbb{R}^{2d_2}$ with all fibers algebraic varieties of fixed degree, which implies by (1) that $E$ satisfies the $\gamma$-ST property with $\gamma := \frac{1}{4d_1-1}$. (This improves the bound in \cite[Theorem 9]{ES}.)
\end{enumerate}
\end{fact}

\begin{prob}
We expect that the same bound on $\gamma$ as in Fact \ref{fac: stronger bounds for algebraic}(2)  should hold  for an arbitrary constructible family $\mathcal{E}_D$ over $\mathbb{C}$ in  Corollary \ref{fac: algebraic ST}(2), and the same bound on $\gamma$ as in Fact \ref{fac: stronger bounds for algebraic}(1)  should hold  for an arbitrary definable family $\mathcal{E}$ in an $o$-minimal structure in Fact \ref{o-min cutting}(2). However, the polynomial method used to obtain these stronger bounds for high dimensions in the algebraic case does not immediately generalize to constructible sets, and is not available for general $o$-minimal structures (see \cite{basu2018zeroes}).
\end{prob}
\begin{fact}\label{fac: semilin}
	Assume that $d_1,d_2,s \in \mathbb{N}$ and $\mathcal{E}$ is a family of semilinear subsets of $\mathbb{R}^{d_1} \times \mathbb{R}^{d_2}$  so that each $E \in \mathcal{E}$ is defined by a Boolean combination of $s$ linear equalities and  inequalities (with real coefficients). Then by \cite[Theorem (C)]{basit2020zarankiewicz}, for every $\varepsilon \in \mathbb{R}_{>0}$ the family $\mathcal{E}$ satisfies the condition in Lemma \ref{prop: Holder iteration}$(*)$ with $\gamma_1 + \gamma_2 \leq 1 + \varepsilon$ (and some function $C_0$ depending on $s$ and $\varepsilon$). It follows that $\mathcal{E}$ satisfies the $(1-\varepsilon)$-ST property for every $\varepsilon > 0$ (which is the best possible bound up to $\varepsilon$).
\end{fact}

\begin{fact}\label{fac: DCF0 dist exp}
It has been shown in \cite{DistValFields} that every differentially closed field (with one or several commuting derivations) of characteristic $0$ admits a distal expansion. Hence by Fact \ref{fac: def in dis impl distal cell decomp}, every definable relation admits a distal cell decomposition of some finite exponent $t$, hence  by Proposition \ref{prop: ind ES}(2) any definable family $\mathcal{E}$ of subsets of $ \subseteq M^{d_1} \times M^{d_2}$ in a differentially closed field $\mathcal{M}$ of characteristic $0$ satisfies the $\gamma$-ST property for some $\gamma > 0$.
\end{fact}

\begin{fact}\label{fac: CCM dist exp}
The theory of compact complex manifolds, or CCM, is the theory of the structure containing a separate sort for each compact complex variety, with each Zariski closed subset of the cartesian products of the sorts named by a predicate (see \cite{moosa} for a survey). This is an $\omega$-stable theory of finite Morley rank, and it is interpretable in the $o$-minimal structure $\mathbb{R}_{\textrm{an}}$. Hence, by Fact \ref{fac: def in dis impl distal cell decomp} and Proposition \ref{prop: ind ES}(2), every definable family $\mathcal{E}$ admits a distal cell decomposition of some finite exponent $t$, and hence satisfies the $\gamma$-ST property for some $\gamma > 0$.
\end{fact}

We remark that in differentially closed fields it is not possible to bound $t$ in terms of $d_1$ alone. Indeed, the dp-rank of the formula ``$x=x$'' is $\geq n$ for all $n \in \mathbb{N}$ (since the field of constants is definable, and $M$ is an infinite dimensional vector space over it, see \cite[Remark 5.3]{chernikov2014valued}). This implies that the VC-density of a definable relation $E\subseteq M \times M^{n}$ cannot be bounded  independent of $n$ (see e.g.~\cite{kaplan2013additivity}), and since $t$ gives an upper bound on the VC-density (see Remark \ref{rem: distal decomp bounds dual VC density}), it cannot be bounded either.

\begin{prob}\label{prob: cell decomp DCF0}
Obtain explicit bounds on the distal cell decomposition and incidence counting for relations $E$ definable in DCF$_0$ (e.g., are they bounded in terms of the Morley rank of the relation $E$?).
\end{prob}

\section{Reconstructing an abelian group from a family of bijections}\label{sec: group config omin}
In this and the following sections we provide two higher arity variants of the group configuration theorem of Zilber-Hrushovski (see e.g~\cite[Chapter 5.4]{pillay1996geometric}).
From a model-theoretic point of view, our result can be viewed as a construction of a type-definable abelian group in the non-trivial \emph{local} locally modular case, i.e.~local modularity is only assumed for the given relation, as opposed to the whole theory, based on a relation of arbitrary arity $\geq 4$.

In this section, as a warm-up, we begin with a purely combinatorial abelian group configuration for the case of bijections as opposed to finite-to-finite correspondences. It  illustrates some of the main ideas and is sufficient for the application in the $o$-minimal case of the main theorem in Section \ref{sec: main thm omin}.

 In the next Section \ref{sec: group config stable}, we generalize the construction to allow finite-to-finite correspondences instead of bijections (model-theoretically, algebraic closure instead of the definable closure) in the stable case, which requires additional forking calculus arguments.

\subsection{$Q$-relations or arity $4$}\label{sec: dcl Q-rels of arity 4}
Throughout this section, we fix some sets $A,B,C,D$ and a quaternary relation $Q \subseteq A\ttimes B \ttimes C  \ttimes D$. We assume that $Q$ satisfies the following two properties.
\begin{itemize}
\item[(P1)]  If we fix any $3$ variables,
  then there is exactly one value for the 4th variable satisfying $Q$.
  \item[(P2)] If
\[ (\alpha,\beta;\gamma,\delta), (\alpha',\beta'; \gamma,\delta),
  (\alpha',\beta'; \gamma',\delta') \in Q, \]
then
\[ (\alpha,\beta;\gamma',\delta')\in Q ;\]
and the same is true under any other partition of the variables into two groups each of size two.
\end{itemize}
Intuitively, the first condition says that $Q$ induces a family of bijective functions between any two of its coordinates, and the second condition says that this family of bijections satisfies  the ``abelian group configuration'' condition in a strong sense. Our goal is to show that under these assumption there exist an  abelian group for which $Q$ is in a coordinate-wise bijective correspondence with the relation defined by $\alpha \cdot \beta = \gamma \cdot \delta$.

First, we can view the relation $Q$ as a $2$-parametric family of bijections as follows. Note that for every pair $(c,d)\in C\times D$, the corresponding fiber $\{(a,b) \in A \times B : (a,b,c,d) \in Q\}$ is the graph of a function from $A$ to $B$ by (P1). Let $\CF$ be the  set of all functions from $A$ to $B$ whose graph
is a fiber of $Q$.

Similarly, let $\CG$ be the  set of all functions from $C$ to $D$ whose graph is a fiber of $Q$ (for some $(a,b) \in A \times B$). Note that all functions in $\CF$ and in $\CG$ are bijections, again by (P1).

\begin{claim}\label{claim:regul}
  For every $(a,b)\in A\times B$  there is a unique $f\in \CF$ with
  $f(a)=b$, and similarly for $\CG$.
\end{claim}
\begin{proof}  We only check this for $\CF$, the argument for $\CG$ is analogous.  Let $(a,b) \in A \times B$ be fixed. Existence: let $c \in C$ be arbitrary, then by (P1) there exists some $d \in D$ with $(a,b,c,d) \in Q$, hence the function corresponding to the fiber of $Q$ at $(c,d)$ satisfies the requirement.
Uniqueness follows from (P2) for the appropriate partition of the variables: if $(a,b; c,d), (a,b;
  c_1,d_1)\in Q$ for some $(c,d), (c_1,d_1) \in C \times D$, then for all $(x,y) \in A \times B$ we have $(x,y, c,d) \in Q \iff  (x,y; c_1,d_1) \in Q$.
\end{proof}

\begin{claim}\label{claim:perp}

For every $f\in \CF$ and $(x,u)$ in $A \times C$ there exists a unique $g\in \CG$ such that $(x,f(x), u,g(u) ) \in Q$ (which then satisfies $(x',f(x'),u',g(u')) \in Q$ for all $(x',u') \in A \times C$).

And similarly exchanging the roles of $\CF$ and $\CG$.
\end{claim}
\begin{proof}
  As $x,f(x),u$ are given, by (P1) there is a unique choice for the fourth coordinate of a tuple in $Q$ determining the image of $g$ on $u$. There is only one such $g \in \CG$ by  Claim \ref{claim:regul} with respect to $\CG$.
\end{proof}

For $f\in \CF$, we will denote by $f^\perp$ the unique $g\in \CG$ as in Claim \ref{claim:perp}. Similarly, for $g\in \CG$, we will denote  by $g^\perp$
the unique $f\in \CF$ as in Claim \ref{claim:perp}.

\begin{rem}\label{rem: perp is a bijection}
Note that $(f^\perp)^\perp=f$ and
$(g^\perp)^\perp=g$ for all $f \in \CF, g \in \CG$.
\end{rem}

\begin{claim}\label{claim:comp}
  Let  $f_1 ,f_2,f_3 \in \CF$,  and $g_i := f_i^\perp \in \CG$ for $i \in [3]$.  Then
$f_3\co f_2^{-1}
  \co f_1 \in \CF$,  $g_3\co g_2^{-1}
  \co g_1 \in \CG$ and $(f_3\co f_2^{-1}\co f_1)^\perp = g_3\co g_2^{-1}
  \co g_1$.
\end{claim}
\begin{proof}
  We first observe the following: given any $a\in A$ and $c\in C$, if we take $b := (f_3\co f_2^{-1}
  \co f_1) (a) \in B$ and $ d := (g_3\co g_2^{-1}
  \co g_1) (c) \in D$, then   $(a,b,c,d) \in Q$. Indeed, let $b_1 :=f_1(a)$, $a_2 :=f_2^{-1} (b_1)$, then $b=f_3(a_2)$.
Similarly, let $d_1 :=g_1(c)$, $c_2 :=g_2^{-1} (d_1)$, then $d=g_3(c_2)$.
By the definition of $\perp$  we then have
\[ (a,b_1, c,d_1) \in Q, (a_2,b_1,c_2,d_1) \in Q, (a_2,b, c_2,d) \in Q. \]
Applying (P2)  for the partition $\{1,3\} \cup \{2,4\}$, this implies
$(a,b, c,d) \in Q$, as wanted.

Now fix an arbitrary $c \in C$ and take the corresponding $d$, varying $a \in A$ the observation implies that the graph of $f_3\co f_2^{-1}\co f_1$ is given by the fiber $Q_{(c,d)}$. Similarly, the graph of $g_3\co g_2^{-1}
  \co g_1$ is given by the fiber $Q_{(a,b)}$ for an arbitrary $a \in A$ and the corresponding $b$; and $(f_3\co f_2^{-1}\co f_1)^\perp = g_3\co g_2^{-1}
  \co g_1$ follows.
\end{proof}

\begin{claim}\label{claim:com}
  For any $f_1,f_2, f_3\in \CF$ we have $f_3\co f_2^{-1}\co f_1=f_1\co
  f_2^{-1}\co f_3$, and similarly for $\CG$.
 \end{claim}
 \begin{proof}   Let $a\in A$ be arbitrary. We define  $b_1 :=f_1(a)$, $a_2 :=f_2^{-1}(b_1)$
   and $b_3 :=f_3(a_2)$, so we have $(f_3\co f_2^{-1}\co f_1)(a)=b_3$.
Let also $b_4 := f_3(a)$, $a_4 := f_2^{-1}(b_4)$
   and $b_5 := f_1(a_4)$, so we have $(f_1\co f_2^{-1}\co f_3)(a)=b_5$.

We need to show that $b_5=b_3$.

Let $c_1 \in C$ be arbitrary. By (P1) there exists some $d_1 \in D$ such that
\begin{gather}
	(a,b_1, c_1, d_1) \in Q.   \label{eq: GrpConfDcl 1}
\end{gather}
 Applying (P1) again, there exists some $c_2 \in C$ such that
\begin{gather}
	(a_2,b_1, c_2, d_1) \in Q,   \label{eq: GrpConfDcl 2}
\end{gather}
and then some $d_2 \in D$ such that
\begin{gather}
	(a_2,b_3, c_2, d_2) \in Q.   \label{eq: GrpConfDcl 3}
\end{gather}
Using (P2) for the partition $\{1,3\} \cup \{2,4\}$, it follows from \eqref{eq: GrpConfDcl 1}, \eqref{eq: GrpConfDcl 2}, \eqref{eq: GrpConfDcl 3}  that
\begin{equation}
  \label{eq: GrpConfDcl 4}  (a,b_3,c_1,d_2) \in Q.
\end{equation}

 On the other hand, by the choice of $b_1, a_2, b_3$, \eqref{eq: GrpConfDcl 1}, \eqref{eq: GrpConfDcl 2}, \eqref{eq: GrpConfDcl 3} and Claim \ref{claim:regul} we have:  $Q_{
(c_1,d_1)}$ is the graph of $f_1$, $Q_{(
c_2,d_1)}$ is the graph of $f_2$ and $Q{(
c_2,d_2)}$ is the graph of $f_3$. Hence we also have
\[ (a,b_4, c_2, d_2) \in Q,   (a_4,b_4, c_2, d_1) \in Q,  (a_4,b_5, c_1, d_1) \in Q.  \]
Applying (P2) for the partition $\{1,4\} \cup \{2,3\}$ this implies
\[ (a,b_5,c_1,d_2) \in Q, \]
and combining with \eqref{eq: GrpConfDcl 4} and (P1) we obtain $b_3=b_5$.
 \end{proof}

\begin{claim} \label{cla: Q 4 F is an ab grp}
	
Given an arbitrary element $f_0\in \CF$, for every pair $f,f' \in \CF$ we define
$$f + f' := f\co f_0^{-1}\co f'.$$
Then $(\CF,+)$ is an abelian group, with the identity element $f_0$.
\end{claim}
\begin{proof}
	Note that for every $f,f' \in \CF$, $f+f' \in \CF$ by Claim \ref{claim:comp}.
	Associativity follows from the associativity of the composition of functions. For any $f \in \CF$ we have $f + f_0 = f \co f^{-1}_0 \co f_0 = f$, $f_0 \co f^{-1} \co f_0 \in \CF$ by Claim \ref{claim:comp} and $f + (f_0 \co f^{-1} \co f_0) = f \co f_0^{-1} \co ( f_0 \co f^{-1} \co f_0) = f_0$, hence $f_0$ is the right identity and $f_0 \co f^{-1} \co f_0$ is the right inverse of $f$. Finally, by Claim \ref{claim:com} we have $f + f' = f' + f$ for any $f,f' \in \CF$, hence $(\CF, +)$ is an abelian group.
\end{proof}

\begin{rem}\label{rem: perp is an iso}
	Moreover, if we also fix $g_0 := f_0^\perp$ in $\CG$, then similarly we
obtain an abelian group on $\CG$ with the identity element $g_0$, so that $(\mathcal{F},+)$ is isomorphic to $(\CG,+)$ via the map $f \mapsto f^{\perp}$ (it is a homomorphism as for any $f_1, f_2 \in \CF$ we have $(f_1 \co f_0^{-1} \co f_2)^\perp = f_1^{\perp} \co g_0^{-1}
  \co f_2^{\perp}$ by Claim \ref{claim:comp}, and its inverse is $g \in \CG \mapsto g^{\perp}$ by Remark \ref{rem: perp is a bijection}).
\end{rem}

Next we establish a connection of these groups and the relation $Q$. We fix arbitrary $a_0\in A$, $b_0\in B$, $c_0\in C$ and $d_0\in D$ with
$(a_0,b_0,c_0,d_0) \in Q$.  By Claim \ref{claim:regul}, let $f_0\in \CF$ be unique with
$f_0(a_0)=b_0$, and let $g_0\in \CG$ be unique with
$g_0(c_0)=d_0$. Then $g_0 = f_0^\perp$ by Claim \ref{claim:perp}, and by Remark \ref{rem: perp is an iso} we have isomorphic
groups on $\CF$ and on $\CG$. We will denote this common group by $G := (\mathcal{F},+)$.

We consider the following bijections between each of
$A,B,C,D$ and $G$, using our identification of $G$ with both $\CF$ and $\CG$ and  Claim \ref{claim:regul}:
\begin{itemize}
	\item let $\pi_A \colon A \to \CF$ be the bijection that assigns to $a\in
A$ the unique $f_a\in \CF$ with $f_a(a)=b_0$;
\item let $\pi_B \colon B \to \CF$ be the bijection that assigns to $b\in
B$ the unique $f_b\in \CF$ with $f_b(a_0)=b$;
\item let $\pi_C \colon C \to \CG$ be the bijection that assigns to $c\in
C$ the unique $g_c\in \CG$ with $g_c(c)=d_0$;
\item let $\pi_D \colon D \to \CG$ be the bijection that assigns to $d\in
D$ the unique $g_d\in \CG$ with $g_d(c_0)=d$.
\end{itemize}

\begin{claim}\label{claim:ab}   For any $a\in A$ and $b\in B$,
  $\pi_A(a)+\pi_B(b)$ is the unique function $f\in \CF$ with
$f(a)=b$.

Similarly, for any $c\in C$ and $d\in D$,  $\pi_C(a)+\pi_D(b)$ is the unique function $g \in \CG$ with
$g(c)=d$.
\end{claim}
\begin{proof}
  Let $(a,b) \in A \times B$ be arbitrary, and let $f := \pi_A(a)+\pi_B(b) = \pi_B(b) + \pi_A(a) = \pi_B(b)\co f_0^{-1} \co \pi_A(a)$.
Note that, from the definitions, $\pi_A(a)\colon
a\mapsto b_0$, $f_0^{-1}\colon b_0 \mapsto a_0$ and $\pi_B(b)\colon
a_0\mapsto b$, hence $f(a) = b$. The second claim is analogous.
\end{proof}

\begin{prop}\label{prop: quaternary dcl group} For any $(a,b,c,d) \in A \times B \times C \times D$, $(a,b,c,d) \in Q$ if and only if $\pi_A(a) +\pi_B(b)=
  \pi_C(c)^{\perp}+\pi_D(d)^{\perp}$ (in $G$).
\end{prop}
\begin{proof}
	Given $(a,b,c,d)$, by Claim \ref{claim:ab} we have: $\pi_A(a) +\pi_B(b)$ is the function
  $f\in \CF$ with $f(a)=b$, and $ \pi_C(c)+\pi_D(d)$ is the function $g\in
  \CG$ with $g(c)=d$. Then, by Claim \ref{claim:perp}, $(a,b,c,d) \in Q$ if and only if $f = g^{\perp}$, and since $\perp$ is an isomorphism this  happens if and only $f = \pi_C(c)^{\perp}+\pi_D(d)^{\perp}$.
\end{proof}

\subsection{$Q$-relation of any arity for dcl}\label{sec: bij Q any arity}

Now we extend the construction of an abelian group to relations of arbitrary arity $\geq 4$. Assume that we are given $m \in \mathbb{N}_{\geq 4}$, sets $X_1, \ldots, X_m$ and a relation $Q \subseteq X_1 \times \cdots \times X_m$ satisfying the following two conditions (corresponding to the conditions in Section \ref{sec: dcl Q-rels of arity 4} when $m=4$).
\begin{itemize}
\item[(P1)] For any permutation of the variables of $Q$ we have:
 $$\forall x_1, \ldots, \forall x_{m-1} \exists ! x_m Q(x_1, \ldots, x_m).$$
\item[(P2)] For any permutation of the variables of $Q$ we have:
\begin{gather*}
	\forall x_1, x_2 \forall y_3, \ldots y_{m} \forall y'_3, \ldots, y'_m \Big( Q(\bar{x}, \bar{y}) \land Q(\bar{x}, \bar{y}') \rightarrow \\
	\big( \forall x'_1, x'_2 Q(\bar{x}', \bar{y}) \leftrightarrow Q(\bar{x}', \bar{y}') \big) \Big),
	\end{gather*}
	where $\bar{x} = (x_1, x_2), \bar{y} = (y_3, \ldots, y_m)$, $Q(\bar{x},\bar{y})$ evaluates $Q$ on the concatenated tuple $(x_1, x_2, y_3, \ldots, y_m)$, and similarly for $\bar{x}', \bar{y}'$.
\end{itemize}

\noindent We let $\mathcal{F}$ be the set of all functions $f: X_1 \to X_2 $ whose graph is given by the set of pairs $(x_1, x_2) \in X_1 \times X_2$ satisfying $Q(x_1, x_2, \bar{b})$ for some $\bar{b} \in X_3 \times \ldots \times X_m$.

\begin{rem}\label{rem: Q-rel any arity uniq}
\begin{enumerate}
\item Every $f \in \mathcal{F}$ is a bijection, by (P1).
\item For every $a_1 \in X_1, a_2 \in X_2$ there exists a unique $f \in \mathcal{F}$ such that $f(a_1) = a_2$ (existence by (P1), uniqueness by (P2)). We will denote it as $f_{a_1, a_2}$.
\end{enumerate}
\end{rem}

\begin{lem}\label{lem: dcl Q any f}
For every $c_i \in X_i, 4 \leq i \leq m$ and $f \in \mathcal{F}$ there exists some $c_3 \in X_3$ such that $Q(x_1, x_2, c_3, c_4, \ldots, c_m)$ is the graph of $f$.
\end{lem}
\begin{proof}
Choose any $a_1 \in X_1$, let $a_2 := f(a_1)$. Choose $c_3 \in X_3$ such that $Q(a_1, a_2, c_3, \ldots, c_m)$ holds by (P1). Then $Q(x_1, x_2, c_3, c_4, \ldots, c_m)$ defines the graph of $f$ by Remark \ref{rem: Q-rel any arity uniq}(2).
\end{proof}

\begin{lem}\label{lem: Q any arity ab}
For any $f_1, f_2, f_3 \in \mathcal{F}$ there exists some $f_4 \in \mathcal{F}$ such that $f_1  \circ f_2^{-1} \circ f_3 = f_3 \circ f_2^{-1} \circ f_1 = f_4$.
\end{lem}
\begin{proof}
Choose any $c_i \in X_i, 5 \leq i \leq m$ and consider the quaternary relation $Q' \subseteq X_1 \times \cdots \times X_4$ defined by $Q'(x_1, \ldots, x_4) := Q(x_1, \ldots, x_4, \bar{c})$. Hence $Q'$ also satisfies  (P1) and (P2), and the graph of every $f \in \CF$ is given by $Q'(x_1,x_2,b_3,b_4)$ for some $b_3 \in X_3, b_4 \in X_4$, by Lemma \ref{lem: dcl Q any f}.
Then the conclusion of the lemma follows from Claims \ref{claim:comp} and \ref{claim:com} applied to $Q'$.
\end{proof}

\begin{defn}\label{def: choice of ei comb group constr}
	We fix arbitrary elements $e_i \in X_i, i=1, \ldots, m$ so that $Q(e_1, \ldots, e_m)$ holds.
Let $f_0 \in \mathcal{F}$ be the function whose graph is given by $Q(x_1, x_2, e_3, \ldots, e_m)$, i.e.~$f_0 = f_{e_1, e_2}$.
We define $+: \mathcal{F} \times \mathcal{F} \to \mathcal{F}$ by taking $f_1 + f_2 := f_1 \circ f_0^{-1} \circ f_2$.
\end{defn}
As in Claim \ref{cla: Q 4 F is an ab grp}, from Lemma \ref{lem: Q any arity ab} we get:

\begin{lem}
$G := (\mathcal{F}, +)$ is an abelian group with the identity element $f_0$.
\end{lem}

\begin{defn}
We define the map $\pi_1: X_1 \to G$ by $\pi_1(a) := f_{a, e_2}$ for all $a \in X_1$, and the map $\pi_2: X_2 \to G$ by $\pi_2(b) := f_{e_1, b}$ for all $b \in X_2$.
\end{defn}

Note that both $\pi_1$ and $\pi_2$ are bijections by Remark \ref{rem: Q-rel any arity uniq}.

\begin{lem}\label{lem: FirstTwoCoords}
For any $a \in X_1$ and $b \in X_2$ we have $\pi_1(a) + \pi_2(b) = f_{a,b}$.
\end{lem}
\begin{proof}
Let $f_1 := \pi_1(a), f_2 := \pi_2(b)$. Note that $f_1(a) = e_2$, $f_0^{-1}(e_2) = e_1$ and $f_2(e_1) = b$, hence $(f_1 + f_2 )(a) = (f_2 + f_1 )(a) = f_2 \circ f_0^{-1}\circ f_1 (a) = b$, so $f_1 + f_2 = f_{a,b}$.
\end{proof}

\begin{defn}
For any set $S \subseteq \{3, \ldots, m \}$, we define the map $\pi_S: \prod_{i \in S} X_i \to G$ as follows: for $\bar{a} = (a_i : i \in S) \in \prod_{i \in S} X_i$, let $\pi_S(\bar{a})$ be the function in $\CF$ whose graph is given by $Q(x_1, x_2, c_3, \ldots, c_m)$ with $c_j := a_j$ for $j \in S$ and $c_j := e_j$ for $j \notin S$. We write $\pi_j$ for $\pi_{\{j \}}$.
\end{defn}

\begin{rem}
	For each $i \in \{3, \ldots, m\}$, the map $\pi_i : X_i \to G$ is a bijection (by (P2)).
\end{rem}

\begin{lem}\label{lem: TheOtherCoords}
Fix some $S \subsetneq \{3, \ldots, m \}$ and $j_0 \in \{3, \ldots, m \} \setminus S$. Let $S_0 := S \cup \{j_0 \}$. Then for any $\bar{a} \in \prod_{i \in S} X_i$ and $a_{j_0} \in X_{j_0}$ we have $\pi_S(\bar{a}) + \pi_{j_0}(a_{j_0}) = \pi_{S_0}(\bar{a}^{\frown}a_{j_0})$.
\end{lem}

\begin{proof}
Without loss of generality we have $S = \{3, \ldots, k \}$ and $j_0 = k+1 \leq m$ for some $k$.
Take any $\bar{a} = (a_3, \ldots, a_k) \in \prod_{3 \leq i \leq k} X_i$ and $a_{k+1} \in X_{k+1}$. Then, from the definitions:
\begin{itemize}
\item the graph of $f_1 := \pi_S(\bar{a})$ is given by $Q(x_1, x_2, a_3, \ldots, a_k, e_{k+1}, \bar{e}')$, where $\bar{e}' := (e_{k+2}, \ldots, e_{m})$;
\item the graph of $f_2 := \pi_{k+1}(a_{k+1})$ is given by $Q(x_1, x_2, e_3, \ldots, e_k, a_{k+1}, \bar{e}')$;
\item the graph of $f_3 := \pi_{S_0}(\bar{a}^{\frown}a_{k+1})$ is given by $Q(x_1, x_2, a_3, \ldots, a_k, a_{k+1}, \bar{e}')$.
\end{itemize}

Let $c_1 \in X_1$ be such that $f_1(c_1) = e_2$ and let $c_2 \in X_2$ be such that $f_2(e_1) = c_2$. Then $(f_1+f_2)(c_1) = (f_2+f_1)(c_1) = f_2 \circ f_0^{-1} \circ f_1 (c_1) = c_2$. On the other hand, the following also hold:

\begin{itemize}
\item $Q(c_1, e_2, a_3, \ldots, a_k, e_{k+1}, \bar{e}')$;
\item $Q(e_1, e_2, e_3, \ldots, e_k, e_{k+1}, \bar{e}')$;
\item $Q(e_1, c_2, e_3, \ldots, e_k, a_{k+1}, \bar{e}')$.
\end{itemize}

Applying (P2) with respect to the coordinates $(2, k+1)$ and the rest, this implies that $Q(c_1, c_2, a_3, \ldots, a_k, a_{k+1}, \bar{e}')$ holds, i.e.~$f_3(c_1) = c_2$. Hence $f_1 + f_2 = f_3$ by Remark \ref{rem: Q-rel any arity uniq}(2), as wanted.

\end{proof}

\begin{prop}\label{prop: almost final of bij group constr}
	For any $\bar{a} = (a_1, \ldots, a_m) \in \prod_{i \in [m]} X_i$ we have
	$$Q(a_1, \ldots, a_m) \iff \pi_1(a_1) + \pi_2(a_2) = \pi_3(a_3) + \ldots + \pi_m(a_m).$$
\end{prop}
\begin{proof}
	Let $\bar{a} = (a_1, \ldots, a_m) \in \prod_{i \in [m]} X_i$ be arbitrary. 	By Lemma \ref{lem: FirstTwoCoords}, $\pi_1(a_1) + \pi_2(a_2) = f_{a_1, a_2}$. Applying Lemma \ref{lem: TheOtherCoords} inductively, we have
	$$\pi_{3, \ldots, m}(a_3, \ldots, a_m) = \pi_3(a_3) + \ldots + \pi_m(a_m).$$
 And by definition, the graph of the function $\pi_{3, \ldots, m}(a_3, \ldots, a_m)$ is given by $Q(x_1, x_2, a_3, \ldots, a_m)$. Combining and using Remark \ref{rem: Q-rel any arity uniq}(2), we get $Q(a_1, \ldots, a_m) \iff \pi_1(a_1) + \pi_2(a_2) = \pi_{3, \ldots, m}(a_3, \ldots, a_m) \iff \pi_1(a_1) + \pi_2(a_2) = \pi_3(a_3) + \ldots + \pi_m(a_m)$.
\end{proof}

We are ready to prove the main theorem of the section.
\begin{thm}\label{thm: main group config bijections}
Given $m \in \mathbb{N}_{\geq 4}$, sets $X_1, \ldots, X_m$ and $Q \subseteq \prod_{i \in [m]} X_i$ satisfying (P1) and (P2), there exists an abelian group $(G,+,0_G)$ and bijections $\pi'_i: X_i \to G$ such that for every $(a_1, \ldots, a_m) \in \prod_{i \in [m]} X_i$ we have
$$Q(a_1, \ldots, a_m) \iff \pi'_1(a_1) + \ldots + \pi'_m(a_m) = 0_G.$$
Moreover, if we have first-order structures $\CM \preceq \CN$ so that $\CN$ is $|\CM|^+$-saturated, each $X_i, i \in [m]$ is type-definable (respectively, definable) in $\CN$ over $\CM$ and $Q = F \cap \prod_{i \in [m]} X_i$ for a relation $F$ definable in $\CN$ over  $\CM$, then given an arbitrary tuple $\bar{e} \in Q$, we can take $G$ to be type-definable (respectively, definable) and the bijections $\pi'_i, i \in [m]$ to be definable in $\CN$, in both cases only using parameters from $\CM$ and $\bar{e}$, so that $\pi'_i(e_i) = 0_G$ for all $i\in [m]$.
\end{thm}
\begin{proof}

By Proposition \ref{prop: almost final of bij group constr}, for any $\bar{a} = (a_1, \ldots, a_m) \in \prod_{i \in [m]} X_i$ we have
\begin{gather}
	Q(a_1, \ldots, a_m) \iff\\
	\pi_1(a_1) + \pi_2(a_2) = \pi_3(a_3) + \ldots + \pi_m(a_m) \iff \nonumber \\
	\pi_1(a_1) + \pi_2(a_2) + (-\pi_3(a_3)) + \ldots + (-\pi_m(a_m))= 0_G, \nonumber
\end{gather}
hence the bijections $\pi'_1 := \pi_1, \pi'_2 := \pi_2$ and $\pi'_i: X_i \to G, \pi'_i(x) := - \pi_i(x)$ for $3 \leq i \leq m$ satisfy the requirement.

Assume now that, for each $i \in [m]$, $X_i$ is type-definable in $\CN$ over $\CM$, i.e.~$X_i$ is the set of solutions in $\CN$ of some partial type $\mu_i(x_i)$ over $\CM$; and that $Q = F \cap \prod_{i \in [m]}X_i$ for some $\CM$-definable relation $F$.
Then from (P1) and (P2) for $Q$, for any permutation of the variables of $Q$ we have in $\CN$:
\begin{gather*}
	\mu_m(x_m) \land \mu_m(x'_m) \land \bigwedge_{1 \leq i \leq m-1} \mu_{i} \left(x_{i} \right) \land \\
	\land F(x_1, \ldots, x_{m-1}, x_m)
	 \land F(x_1, \ldots, x_{m-1}, x'_m) \rightarrow x_m = x'_m,\\
	\bigwedge_{i \in [m]} \mu_{i} \left(x_{i} \right) \land \bigwedge_{i \in [m]} \mu_{i} \left(x'_{i} \right) \land F(x_1, x_2, x_3, \ldots, x_m) \land F(x_1, x_2, x'_3, \ldots, x'_m) \land \\
	\land F(x'_1, x'_2, x_3, \ldots, x_m) \rightarrow F(x'_1, x'_2, x'_3, \ldots, x'_m).
\end{gather*}

By $|\CM|^+$-saturation of $\CN$, in each of these implications $\mu_i$ can be replaced by a finite conjunction of formulas in it. Hence, taking a finite conjunction over all permutations of the variables, we conclude that there exist some $\CM$-definable sets $X'_i \supseteq X_i, i \in [m]$ so that $Q' := F \cap \prod_{i \in [m]} X'_i$ satisfies (P2) and
\begin{itemize}
\item[(P1$'$)] For any permutation of the variables of $Q'$, for any $x_i \in X'_i, 1 \leq i \leq m-1$, there exists at most one (but possibly none) $x_m \in X'_m$ satisfying $Q'(x_1, \ldots, x_m)$.
 \end{itemize}
We proceed to type-definability of $G$. Let $(e_1, \ldots, e_m) \in Q$ (so in $\CN$) be as above (see Definition \ref{def: choice of ei comb group constr}). We identify $X_2$ with $\CF$, the domain of $G$, via the bijection $\pi_2$ above mapping $a_2 \in X_2$ to $f_{e_1,a_2}$ (in an analogous manner we could identify the domain of $G$ with any of the type-definable sets $X_i, i \in [s]$).
Under this identification, the graph of addition in $G$ is given by
\begin{gather*}
	R_{+} := \left\{(a_2, a'_2, a''_2) \in X_2 \times X_2 \times X_2 : a''_2 =  f_{e_1,a_2} \circ f_{e_1, e_2}^{-1}\circ f_{e_1, a'_2}(e_1)\right\}\\
	= \left\{(a_2, a'_2, a''_2) \in X_2 \times X_2 \times X_2 : a''_2 =  f_{e_1,a_2} \circ f_{e_1, e_2}^{-1}(a'_2) \right\}.
\end{gather*}

\noindent We have the following claim.
\begin{claim}\label{cla: type-def of bij grp}
\begin{itemize}
	\item For any $a_1 \in X_1, a_2 \in X_2$ and $\bar{b} \in \prod_{3 \leq i \leq m} X'_i$, if $F(a_1,a_2, \bar{b})$ holds then $F_{\bar{b}} \restriction_{X_1 \times X_2}$ defines the graph of $f_{a_1,a_2}$ (since $Q'$ satisfies (P2)).
	\item For any $\bar{b} \in \prod_{3 \leq i \leq m} X'_i$, if $F_{\bar{b}} \restriction_{X_1 \times X_2}$ coincides with the graph of some function $f \in \CF$, then using that $Q'$ satisfies (P1$'$) we have:
	\begin{itemize}
		\item for any $a_1 \in X_1$, $f(a_1)$ is the unique element in $X'_2$ satisfying $F(a_1, x_2, \bar{b})$;
		\item  for any $a_2 \in X_2$, $f^{-1}(a_2)$ is the unique element in $X'_1$ satisfying $F(x_1, a_2, \bar{b})$.
	\end{itemize}
\end{itemize}
\end{claim}

Using Claim \ref{cla: type-def of bij grp}, we have
$$R_{+} = R'_{+} \restriction \prod_{i \in [m]} X_i,$$
where $R'_{+}$ is a definable relation in $\CN$ (with parameters in $\CM \cup \{e_1, e_2\}$) given by
\begin{gather*}
	R'_{+}(x_2, x'_2, x''_2) :\iff
	\exists \bar{y}, \bar{y}', z \Big(\bar{y} \in \prod_{3 \leq i \leq m} X'_i \land \bar{y}' \in \prod_{3 \leq i \leq m} X'_i \land z \in X'_1  \land \\
	  F(e_1, e_2, \bar{y}')
\land  F(z,x'_2, \bar{y}') \land F(e_1,x_2, \bar{y}) \land F(z,x''_2,\bar{y}) \Big).
\end{gather*}
This shows that $(G,+)$ is type-definable over $\CM \cup \{e_1, e_2\}$. It remains to show definability of the bijections $\pi'_i: X_i \to \CF$, where $\CF$ is identified with $X_2$ as above (i.e.~to show that the graph of $\pi'_i$ is given by some $\CN$-definable relation $P_i(x_i,x_2)$ intersected with $X_i \times X_2$).

We have $\pi'_1: a_1 \in X_1 \mapsto f_{a_1, e_2} \in \CF$, hence we need to show that the relation
$$\left\{(a_1,a_2) \in X_1 \times X_2 : f_{a_1, e_2}(e_1) = a_2\right\}$$
is of the form $P_1(x_1, x_2) \restriction X_1 \times X_2$ for some relation $P_1$ definable in $\CN$. Using Claim \ref{cla: type-def of bij grp}, we can take
\begin{gather*}
	P_1(x_1,x_2) :\iff \exists \bar{y} \left( \bar{y} \in \prod_{3 \leq i \leq m} X'_i \land  F(x_1, e_2, \bar{y}) \land F(e_1, x_2, \bar{y}) \right).
\end{gather*}

We have $\pi'_2: a_2 \in X_2 \mapsto f_{e_1,a_2} \in \CF$, hence the corresponding definable relation $P_2(x_2,x_2)$ is just the graph of the equality.

Finally, given $3 \leq i \leq m$, $\pi_i$ maps $a_i \in X_i$ to the function in $\CF$ with the graph given by $Q(x_1,x_2, e_3, \ldots, e_{i-1}, a_i, e_{i+1}, \ldots, e_m)$. Hence, remembering that the identity of $G$ is  $f_{e_1,e_2}$, which corresponds to $e_2 \in X_2$, and using Claim \ref{cla: type-def of bij grp}, the graph of $\pi'_i: a_i \in X_i \mapsto - \pi_i(a_i)(e_1) \in X_2$ is given by the intersection of $X_i \times X_2$ with the definable relation
\begin{gather*}
	P_i(x_i, x_2) :\iff \exists z \Big( z \in X'_2 \land  F(e_1, z, e_3, \ldots, e_{i-1}, x_i, e_{i+1}, \ldots, e_m) \land \\
	R'_{+}(x_2, z, e_2) \Big).\qedhere
\end{gather*}
\end{proof}

\section{Reconstructing an abelian group from an abelian $m$-gon}\label{sec: group config stable}
Let $T= T^{\eq}$ be a stable theory in a language $\cL$ and $\mathbb{M}$ a monster model of $T$. By ``independence'' we mean independence in the sense of forking, unless stated otherwise, and write $a \ind_c b$  to denote that $\tp(a/bc)$ does not fork over $c$. We assume some familiarity with the properties of forking in stable theories (see e.g.~\cite{MR3888974} for a concise introduction to model-theoretic stability, and \cite{pillay1996geometric} for a detailed treatment). We say that a subset $A$ of $\CM$ is \emph{small} if $|A| \leq |\CL|$.

\subsection{Abelian $m$-gons}

For a small set $A$, as usual  by its  \emph{$\acl_A$-closure} we mean the algebraic closure over $A$,
i.e.~for a set $X$ its $\acl_A$-closure is $\acl_A(X) := \acl(A\cup X)$.


\begin{defn}\footnote{An analogous notion in the context of geometric theories was introduced in \cite{berenstein2016geometric} under the name of an \emph{algebraic $m$-gon}, and it was also  used in \cite[Section 7]{chernikov2019n}.}
We say that a tuple $(a_1,\ldots, a_m)$ is \emph{an $m$-gon over a set $A$}  if
each type $\tp(a_i/A)$ is not algebraic,  any $m-1$ elements of the tuple are
independent over $A$, and  every element is in the $\acl_A$-closure of the rest.
 We refer to a $3$-gon as a \emph{triangle}.
\end{defn}

\begin{defn}\label{def: m-gon}
  We say that an $m$-gon $(a_1,\ldots,a_m)$ over $A$ with $m\geq 4$ is \emph{abelian}  if   for any $i{\neq j} \in[m]$, taking $\bar a_{ij} :=(a_k)_{k\in [m]\setminus \{i,j\}}$, we have
\[  a_i a_j \ind_{\acl_A(a_i a_j)\cap \acl_A(\bar a_{ij})} \bar a_{ij}. \]
\end{defn}

\begin{sample} Let $A$ be a small set and
let $(G, \cdot, 1_G)$ be an abelian group type-definable over $A$.
Let  $g_1, \dotsc, g_{m-1} \in G$ be independent generic elements over $A$, and let $g_m$ be such that $g_1{\cdot} \dotsc  {\cdot g_m} = 1_G$.
Then $(g_1,\dotsc, g_m)$ is an abelian $m$-gon over $A$  \emph{associated to $G$}.

Indeed, by assumption we have $g_1 \cdot g_2 \in \dcl(g_1, g_2) \cap \dcl(g_3, \ldots, g_m)$. Also $g_1 g_2 \ind_{A} g_3 \ldots g_{m-1}$, hence $g_1 g_2 \ind_{A, g_1 \cdot g_2} g_3 \ldots g_{m-1}$, which together with $g_m \in \dcl(g_1 \cdot g_2, g_3, \ldots, g_{m-1})$ implies $g_1 g_2 \ind_{A, g_1 \cdot g_2} g_3 \ldots g_m$. As  the group $G$ is abelian, the same holds for any $i \neq j \in [m]$ instead of $i=1, j=2$.
\end{sample}

\begin{defn}
  Given two tuples $\bar a=(a_1,\dots,a_m)$, $\bar
  a'=(a_1,\dotsc,a_m)$ and a small set $A$ we say that $\bar a$ and
  $\bar a'$ are \emph{$\acl$-equivalent  over $A$} if
  $\acl_A(a_i)=\acl_A(a_i')$  for all $i\in [m]$. As usual if
  $A=\emptyset$ we omit it.
\end{defn}

\begin{rem}
	Note that the condition ``$\bar{a}, \bar{a}'$ are $\acl$-equivalent'' is stronger than ``the tuples $\bar{a}, \bar{a}'$ are inter-algebraic'', as it
requires inter-algebraicity component-wise.
\end{rem}

In this section we prove the following theorem.

\begin{thm}\label{thm: main ab mgon gives grp}
Let $\bar a=(a_1, \ldots, a_m)$ be an abelian $m$-gon, over some small set
$A$.  Then there is a finite  set $C$ with $\bar a
\ind_A{C}$, a type-definable (in $\MM^{\eq}$) over $\acl(C\cup A)$
connected (i.e.~$G = G^0$) abelian group $(G, \cdot)$ and an abelian $m$-gon $\bar g=(g_1, \ldots,
g_m)$ over $\acl(C\cup A)$ associated to $G$  such that $\bar a$ and $\bar g$ are
$\acl$-equivalent over $\acl(C\cup A)$.
\end{thm}

\begin{rem}
	After this work was completed, we have learned that independently Hrushovski   obtained a similar (but incomparable) result \cite{HrushovskiUnpubl, HruOber}.
\end{rem}

\begin{rem}
	In the case $m=4$, Theorem~\ref{thm: main ab
  mgon gives grp} follows from the Abelian Group Configuration Theorem (see
\cite[Theorem C.2]{bays2017model}).
\end{rem}

In the rest of the section we prove Theorem \ref{thm: main ab mgon gives grp}, following the  presentation of Hrushovski's Group Configuration Theorem in \cite[Theorem 6.1]{bays2018geometric} with appropriate modifications.

First note that, adding to the language new constants naming the elements of $\acl(A)$, we may assume without loss of generality that $A=\emptyset$ in Theorem~\ref{thm: main ab mgon gives grp}, and that all types over the
empty set are stationary.

	Given a
tuple $\bar a=(a_1,\dotsc a_m)$ we will often modify it  by applying
the following two operations:
\begin{itemize}
\item  for a \textbf{finite} set $B$ with $\bar a \ind B$  we expand
  the language by  constants  for the elements of $\acl(B)$, and refer to this as \emph{``base change to $B$''}.
\item  we replace $\bar a$ with an $\acl$-equivalent tuple $\bar a'$ (over $\emptyset$), and refer
  to this as \emph{``inter-algebraic replacing''}.
\end{itemize}
It is not hard to see that these two operations transform an (abelian)
$m$-gon to an (abelian) $m$-gon, and we will freely apply them to the
$m$-gon $\bar a$ in the proof of  Theorem~\ref{thm: main ab mgon gives
  grp}.
\begin{defn}
We say that a tuple $(a_1,\dotsc, a_m, \xi)$ is an
\emph{expanded abelian $m$-gon} if $(a_1,\dotsc, a_m)$ is an
abelian $m$-gon, $\xi \in \acl(a_1,a_2)\cap  \acl(a_3,\dotsc,a_m)$ and
$a_1a_2 \ind_\xi a_3\dotsc a_m$. 
\end{defn}
\noindent We remark that the tuple $\xi$ might be infinite even if all of the tuples $a_i$'s are finite. Similarly, base change and inter-algebraic
replacement transform an expanded abelian
$m$-gon to an expanded abelian $m$-gon.

\medskip

From now on, we fix an abelian $m$-gon $\vec{a} = (a_1,\dotsc,a_m)$. We also fix
 $\xi \in \acl(a_1,a_2)\cap  \acl(a_3,\dotsc,a_m)$ such that
$a_1a_2 \ind_\xi a_3\dotsc a_m$ (exists by the definition of abelianity).

\begin{claim}\label{claim:both-gons}
  $(a_1,a_2,\xi)$ is a triangle and $(\xi,a_3,\dotsc,a_m)$ is an $(m-1)$-gon.
\end{claim}
\begin{proof} For $i=1,2$, since $a_i\ind a_3,\dotsc,a_m$ and $\xi\in
  \acl(a_3,\dotsc,a_m)$  we have $a_i \ind \xi$. Also $a_1\ind
  a_2$. Thus the set $\{a_1,a_2,\xi\}$ is pairwise independent.
We also have $\xi\in \acl(a_1,a_2)$.  From $a_1a_2\ind_\xi a_3\dotsc a_m$
we obtain $a_1\ind_{\xi a_2} a_3\dotsc a_m$.  Since $a_1\in
\acl(a_2,\dotsc,a_m)$ we obtain $a_1\in \acl(\xi, a_2)$. Similarly
$a_2\in \acl(\xi,a_1)$, thus $(a_1,a_2,\xi)$ is a triangle.

The proof that $(\xi,a_3,\dotsc,a_m)$ is an $(m-1)$-gon is similar.
\end{proof}

\subsection{Step 1. Obtaining a pair of interdefinable elements}
\noindent \emph{After applying finitely many base changes and  inter-algebraic
  replacements  we may assume that $a_1$ and $a_2$ are interdefinable
  over  $\xi$, i.e.~$a_1\in \dcl(\xi , a_2)$ and $a_2\in \dcl(\xi, a_1).$}

\medskip

Our proof of Step~1 follows closely the proof of the corresponding
step in the proof of \cite[Theorem 6.1]{bays2018geometric}, but in order to keep track of the additional parameters we work with enhanced  group configurations.

\begin{defn}
\emph{An enhanced group configuration} is a tuple
$$(a,b,c,x,y,z,d,e)$$
satisfying the following diagram.
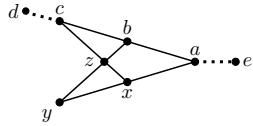
\[
\begin{tikzpicture}[scale=0.9]
\draw[black,line  width=0.2mm] (0,0) -- (1,0.3);
\draw[black,line  width=0.2mm] (1,0.3) -- (2,0.6);
\draw[black,line  width=0.2mm] (2,0.6) -- (1,0.9);
\draw[black,line  width=0.2mm] (1,0.9) -- (0,1.2);
\draw[black,line  width=0.2mm] (0,0) -- (1, 0.9);
\draw[black,line  width=0.2mm] (0,1.2) -- (1, 0.3);

\draw[fill=black] (2,0.6) circle (1.5pt);
\node[scale=0.7, above] at (2, 0.6) {$a$};
\draw[fill=black] (1,0.9) circle (1.5pt);
\node[scale=0.7, above] at (1, 0.9) {$b$};
\draw[fill=black] (0,1.2) circle (1.5pt);
\node[scale=0.7, above] at (0, 1.2) {$c$};
\draw[fill=black] (1,0.3) circle (1.5pt);
\node[scale=0.7, below] at (1, 0.3) {$x$};
\draw[fill=black] (0,0) circle (1.5pt);
\node[scale=0.7, below left] at (0, 0) {$y$};
\draw[fill=black] (0.66,0.6) circle (1.5pt);
\node[scale=0.7, left] at (0.62, 0.6) {$z$};

\draw[black,line  width=0.4mm, dotted] (0, 1.2) -- (-0.5, 1.33);
\draw[fill=black] (-0.5,1.35) circle (1.33pt);
\node[scale=0.7, left] at (-0.5, 1.35) {$d$};

\draw[black,line  width=0.4mm,dotted] (2,0.6) -- (2.6, 0.6);
\draw[fill=black] (2.6,0.6) circle (1.5pt);
\node[scale=0.7,right] at (2.6, 0.6) {$e$};

\end{tikzpicture}
\]
That is,
\begin{itemize}
\item $(a,b,c)$ is a triangle over $de$;
\item $(c, z, x)$ is a triangle over $d$;
\item $(y,x,a)$ is a triangle over $e$;
\item $(y,z,b)$  is a triangle;
\item for any non-collinear triple in $(a,b,c,x,y,z)$, the set given
  by it and $de$ is independent over $\emptyset$.
\end{itemize}

\medskip
If $e=\emptyset$ we omit it from the diagram:
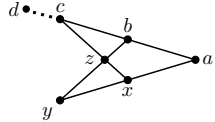
\[
 \begin{tikzpicture}[scale=0.9]
\draw[black,line  width=0.2mm] (0,0) -- (1,0.3);
\draw[black,line  width=0.2mm] (1,0.3) -- (2,0.6);
\draw[black,line  width=0.2mm] (2,0.6) -- (1,0.9);
\draw[black,line  width=0.2mm] (1,0.9) -- (0,1.2);
\draw[black,line  width=0.2mm] (0,0) -- (1, 0.9);
\draw[black,line  width=0.2mm] (0,1.2) -- (1, 0.3);

\draw[fill=black] (2,0.6) circle (1.5pt);
\node[scale=0.7, right] at (2, 0.6) {$a$};
\draw[fill=black] (1,0.9) circle (1.5pt);
\node[scale=0.7, above] at (1, 0.9) {$b$};
\draw[fill=black] (0,1.2) circle (1.5pt);
\node[scale=0.7, above] at (0, 1.2) {$c$};
\draw[fill=black] (1,0.3) circle (1.5pt);
\node[scale=0.7, below] at (1, 0.3) {$x$};
\draw[fill=black] (0,0) circle (1.5pt);
\node[scale=0.7, below left] at (0, 0) {$y$};
\draw[fill=black] (0.66,0.6) circle (1.5pt);
\node[scale=0.7, left] at (0.62, 0.6) {$z$};

\draw[black,line  width=0.4mm, dotted] (0, 1.2) -- (-0.5, 1.33);
\draw[fill=black] (-0.5,1.35) circle (1.33pt);
\node[scale=0.7, left] at (-0.5, 1.35) {$d$};

\end{tikzpicture}
\]
\end{defn}

\

In order to complete Step 1 we first show a few  lemmas.

\begin{lem} \label{lem: conj is interalg}
  Let $(a,b,c,x,y,z,d,e)$ be an enhanced group configuration.
  Let $\tilde{z} \in \MM^{\eq}$ be the  imaginary representing
  the finite set $\{z_1, \ldots, z_k\}$ of all conjugates of $z$ over
  $bcxyd$. Then $\tilde{z}$ is inter-algebraic with $z$.
\end{lem}
\begin{proof}
It suffices to show that $\acl(z_i) = \acl(z_j)$ for all $1\leq i,j
\leq k$.
Indeed, then $\tilde{z}\in\acl(z_1, \ldots, z_k) = \acl(z)$, and $z \in \acl(\tilde{z})$ as it satisfies the algebraic formula ``$z \in \tilde{z}$''.

We have $c d \ind y z$, so $c d  \ind_z y $, so $c d x \ind_z b y$.
Let $B := \acl(cdx) \cap \acl(by)$, then $B \ind_z B$, so $B \subseteq
\acl(z)$. But $z \in B$, so $B = \acl(z)$.
Then we also have $\acl(z_i) = B$ since for each $z_i$
there is an automorphism $\sigma$ of $\MM$ with $\sigma(z)=z_i$ and
$\sigma(B)=B$.
\end{proof}

\begin{lem}\label{lem: getting one dcl}
  Assume that $(a,b,c,x,y,z,d,e)$ is an enhanced group configuration.
  Then after a base change it is $\acl$-equivalent to an enhanced group
  configuration
  $(a,b_1,c,x,y_1,z_1,d,e)$ such that $z_1 \in
  \dcl(b_1y_1)$. Moreover,
   $b\in\dcl(b_1)$ and $y\in \dcl(y_1)$.
\end{lem}
\begin{proof}
Recall  that by our assumption all types over the empty set are stationary.

Let $a' d' e' \models \tp(ade)|_{abcdexyz}$. We have $a d e\ind y z $,
hence $ad e\ind y z  b$. Then by stationarity we have  $a' d' e'
\equiv_{yzb} ade$. Let $x', c'$ be such that $a' d' e' x' c'
\equiv_{yzb} a d e x c$. So $(a', b, c', x', y, z,d',e')$ is also an
enhanced group configuration. Applying Lemma \ref{lem: conj is
  interalg} to it, the set $\tilde z'$ of conjugates of $z$ over
$ybx'c'd'$ is inter-algebraic with $z$, and $\tilde z' \in \dcl(ybx'c'd')$.

We add $\acl(a'd'e')$ to the base, and take
$y_1 := yx'$,  $b_1 := bc'$,  $z_1 := \tilde z'$.
Then $(a,b_1,c,x,y_1,z_1,d,e)$ is an enhanced group
configuration
satisfying the conclusion of the lemma.
\end{proof}

\begin{lem}\label{lem:enhanc-group}
Let $(a,b,c,x,y,z,d,e)$ be an enhanced group configuration with $e\in \dcl(\emptyset)$.
Then, applying finitely many base changes and inter-algebraic replacements, it can be
transformed to
a configuration 
$$(a_1,b_1,c_1,x_1,y_1,z_1,d,e)$$
such that $y_1$ and $z_1$ are \emph{interdefinable} over
$b_1$. (Notice that $d$  and $e$ remain unchanged.)
\end{lem}

\begin{proof}
Applying Lemma \ref{lem: getting one dcl}, after a base change and
an inter-algebraic replacement  we  may
assume $z \in \dcl(by)$.

Next observe that, since $e \in \dcl(\emptyset)$, the tuple $(b, a, c, z, y, x, d, e)$ is also an enhanced group configuration.

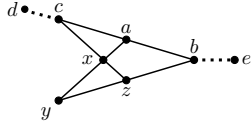
\[
\begin{tikzpicture}[scale=0.9]
\draw[black,line  width=0.2mm] (0,0) -- (1,0.3);
\draw[black,line  width=0.2mm] (1,0.3) -- (2,0.6);
\draw[black,line  width=0.2mm] (2,0.6) -- (1,0.9);
\draw[black,line  width=0.2mm] (1,0.9) -- (0,1.2);
\draw[black,line  width=0.2mm] (0,0) -- (1, 0.9);
\draw[black,line  width=0.2mm] (0,1.2) -- (1, 0.3);

\draw[fill=black] (2,0.6) circle (1.5pt);
\node[scale=0.7, above] at (2, 0.6) {$b$};
\draw[fill=black] (1,0.9) circle (1.5pt);
\node[scale=0.7, above] at (1, 0.9) {$a$};
\draw[fill=black] (0,1.2) circle (1.5pt);
\node[scale=0.7, above] at (0, 1.2) {$c$};
\draw[fill=black] (1,0.3) circle (1.5pt);
\node[scale=0.7, below] at (1, 0.3) {$z$};
\draw[fill=black] (0,0) circle (1.5pt);
\node[scale=0.7, below left] at (0, 0) {$y$};
\draw[fill=black] (0.66,0.6) circle (1.5pt);
\node[scale=0.7, left] at (0.62, 0.6) {$x$};

\draw[black,line  width=0.4mm, dotted] (0, 1.2) -- (-0.5, 1.33);
\draw[fill=black] (-0.5,1.35) circle (1.33pt);
\node[scale=0.7, left] at (-0.5, 1.35) {$d$};

\draw[black,line  width=0.4mm,dotted] (2,0.6) -- (2.6, 0.6);
\draw[fill=black] (2.6,0.6) circle (1.5pt);
\node[scale=0.7,right] at (2.6, 0.6) {$e$};

\end{tikzpicture}
\]

By  Lemma \ref{lem: getting one dcl}, after a base change,  it  is $\acl$-equivalent
to a configuration $(b,a_1,c,z,y_1,x_1,d,e)$ with
$x_1\in \dcl(a_1,y_1)$ and $y\in \dcl(y_1)$.
Thus after an inter-algebraic replacement we may assume that
$x \in \dcl(ay)$  and $z \in \dcl(by)$.

Finally, observe that $(c,b,a,x,z,y,e,d)$ is an enhanced group  configuration.

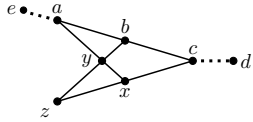
\[
\begin{tikzpicture}[scale=0.9]
\draw[black,line  width=0.2mm] (0,0) -- (1,0.3);
\draw[black,line  width=0.2mm] (1,0.3) -- (2,0.6);
\draw[black,line  width=0.2mm] (2,0.6) -- (1,0.9);
\draw[black,line  width=0.2mm] (1,0.9) -- (0,1.2);
\draw[black,line  width=0.2mm] (0,0) -- (1, 0.9);
\draw[black,line  width=0.2mm] (0,1.2) -- (1, 0.3);

\draw[fill=black] (2,0.6) circle (1.5pt);
\node[scale=0.7, above] at (2, 0.6) {$c$};
\draw[fill=black] (1,0.9) circle (1.5pt);
\node[scale=0.7, above] at (1, 0.9) {$b$};
\draw[fill=black] (0,1.2) circle (1.5pt);
\node[scale=0.7, above] at (0, 1.2) {$a$};
\draw[fill=black] (1,0.3) circle (1.5pt);
\node[scale=0.7, below] at (1, 0.3) {$x$};
\draw[fill=black] (0,0) circle (1.5pt);
\node[scale=0.7, below left] at (0, 0) {$z$};
\draw[fill=black] (0.66,0.6) circle (1.5pt);
\node[scale=0.7, left] at (0.62, 0.6) {$y$};

\draw[black,line  width=0.4mm, dotted] (0, 1.2) -- (-0.5, 1.33);
\draw[fill=black] (-0.5,1.35) circle (1.33pt);
\node[scale=0.7, left] at (-0.5, 1.35) {$e$};

\draw[black,line  width=0.4mm,dotted] (2,0.6) -- (2.6, 0.6);
\draw[fill=black] (2.6,0.6) circle (1.5pt);
\node[scale=0.7,right] at (2.6, 0.6) {$d$};

\end{tikzpicture}
\]

Applying the proof of Lemma \ref{lem: getting one dcl} to it, after
base change to an independent copy $c'd'e'$ of $cde$, let $a'x'c' d'
e' \equiv_{ybz} axcd e$, let $\tilde y'$ be the set of conjugates of
$y$ over $ba'zx'e'$, equivalently over $ba'zx'$ since $e' \in \dcl(\emptyset)$. So $y' \in \dcl(ba'zx')$.

Now since $x' \in \dcl(a'y)$ and $z \in \dcl(by)$ (since this was
satisfied on the previous step), we have $zx' \in \dcl(ba'y)$. But
then $zx' \in \dcl(ba' y')$ for any $y'$ a conjugate of $y$ over
$ba'zx'$, and so $zx' \in \dcl(ba'\tilde y')$. We take $b_1 := ba'$,
$z_1 := zx'$ and
$y_1 := \tilde y'$. Then $y_1 \in \dcl(b_1z_1)$, and also $z_1 \in
\dcl(b_1 y_1)$, and the tuple $(a,b_1,c,x,y_1,z_1,d,e)$  satisfies the
conclusion of the lemma.
\end{proof}

We can now  finish Step 1.

Let $(a_1, \dotsc, a_m,\xi)$ be an expanded
abelian $m$-gon.
Let $\tilde{a} := a_5 \ldots a_m$ and
$\eta := \acl(a_1a_3)\cap\acl(a_2a_4\dotsc a_m)$

It is easy to check that  $(a_3, \xi, a_4, \eta, a_1, a_2,
\tilde{a},\emptyset)$
is an enhanced group configuration.

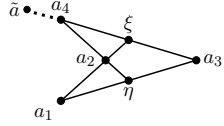
\[
\begin{tikzpicture}[scale=0.9]
\draw[black,line  width=0.2mm] (0,0) -- (1,0.3);
\draw[black,line  width=0.2mm] (1,0.3) -- (2,0.6);
\draw[black,line  width=0.2mm] (2,0.6) -- (1,0.9);
\draw[black,line  width=0.2mm] (1,0.9) -- (0,1.2);
\draw[black,line  width=0.2mm] (0,0) -- (1, 0.9);
\draw[black,line  width=0.2mm] (0,1.2) -- (1, 0.3);

\draw[fill=black] (2,0.6) circle (1.5pt);
\node[scale=0.7, right] at (2, 0.6) {$a_3$};
\draw[fill=black] (1,0.9) circle (1.5pt);
\node[scale=0.7, above] at (1, 0.9) {$\xi$};
\draw[fill=black] (0,1.2) circle (1.5pt);
\node[scale=0.7, above] at (0, 1.2) {$a_4$};
\draw[fill=black] (1,0.3) circle (1.5pt);
\node[scale=0.7, below] at (1, 0.3) {$\eta$};
\draw[fill=black] (0,0) circle (1.5pt);
\node[scale=0.7, below left] at (0, 0) {$a_1$};
\draw[fill=black] (0.66,0.6) circle (1.5pt);
\node[scale=0.7, left] at (0.62, 0.6) {$a_2$};

\draw[black,line  width=0.4mm, dotted] (0, 1.2) -- (-0.5, 1.33);
\draw[fill=black] (-0.5,1.35) circle (1.33pt);
\node[scale=0.7, left] at (-0.5, 1.35) {$\tilde a$};

\end{tikzpicture}
\]

Applying Lemma \ref{lem:enhanc-group}, after a base change it is
$\acl$-equivalent to an enhanced group configuration  $(a_3',
\xi', a_4', \eta', a_1', a_2', \tilde{a},\emptyset)$  such that
$a'_1$ and $a'_2$ are interdefinable over $\xi'$.
Replacing $a_1, a_2, a_3, a_4$ with $a_1', a_2', a_3', a_4'$,
respectively,
and $\xi$ with $\xi'$ we complete Step 1.

\medskip
\noindent\textbf{Reduction 1.} \emph{ From now on we assume that in the
  expanded abelian $m$-gon $(a_1, \ldots, a_m,  \xi\,)$ we have that
  $a_1$ and $a_2$ are interdefinable over
$\xi$. }
\medskip

\subsection{Step 2. Obtaining a group from an expanded
  abelian $m$-gon.}

As in Hrushovski's  Group Configuration Theorem, we will construct a group
using germs of definable functions. We begin by recalling some definitions (see e.g.~\cite[Section 5.1]{bays2018geometric}).

Let $p(x)$ be a stationary type over a set $A$.  By
 \emph{a definable function on $p(x)$} we mean
 a (partial) function $f(x)$   definable  over a set $B$ such that every element
 $a\models p|_{ A B}$ is  in the domain of $f$.

If $f$ and $g$ are two definable functions on $p(x)$, defined over sets $B$
and $C$ respectively,  then we say that \emph{they have the same germ
  at $p(x)$}, and
write $f\sim_p g$, if for all (equivalently, some) $a\models p | _{ABC}$ we have $f(a)=g(a)$.
We may omit $p$ and write $f\sim g$ if no confusion arises.

\emph{The germ of a definable function $f$ at
$p$} is the equivalence class of $f$ under this
equivalence relation, and we denote it by $\tilde f$.

If $p(x)$ and $q(y)$ are stationary types over $\emptyset$, we write $\tilde{f}: p \to q$ if for some (any) representative $f$ of $\tilde{f}$ definable over $B$ and $a \models p|_{B}$ we have $f(a) \models q$. We say that $\tilde{f}$ is \emph{invertible} if  there exists a germ $\tilde{g}: q \to p$ and for some (any) representative $g$ definable over $C$ and $a  \models p|_{BC}$ we have $g(f (a)) = a$. We denote $\tilde{g}$ by $\tilde{f}^{-1}$.

By a \emph{type-definable family of functions from $p$ to $q$} we mean an
$\emptyset$-definable family of functions $f_z$ and a stationary type $s(z)$ over $\emptyset$
such that for any $c\models s(z)$ the function $f_c$  is a definable function on
$p$, and  for any $a\models p|_c$ we have $f_c(a)\models q(y)|_c$.
We will denote such a family as $f_s\colon p\to q$, and the family of
the corresponding germs as $\tilde f_s\colon p\to q$.

Let $p,q,s$ be stationary types over  $\emptyset$ and $
f_s\colon p\to q$ a type-definable family of functions.
This family is \emph{generically transitive} if $f_c(a)\ind a$
for any (equivalently, some) $c\models s$ and $a\models p|c$.
This family is \emph{canonical} if for any $c,c'\models s$ we have
$f_c\sim f_{c'} \Leftrightarrow c=c'$.

\medskip
We now return to our expanded abelian $m$-gon $(\vec a,\xi)$.

Let     $p_i(x_i): = \tp(a_i/\emptyset) $ for $i \in \{1,2 \}$, and let
$q(y) := \tp(\xi/\emptyset)$.

Since $a_1$ and $a_2$ are interdefinable over $\xi$ and $\xi \in \acl(a_1, a_2)$, there exists a formula $\varphi(x_1, x_2, y) \in \tp \left(a_1, a_2, \xi \right)$ such that
\begin{gather}
\nonumber \models \forall y \forall x_1 \exists^{\leq 1}x_2 \varphi(x_1,x_2, y),
 \models \forall y \forall x_2 \exists^{\leq 1}x_1 \varphi(x_1,x_2, y), \\
   \nonumber \models \forall x_1 \forall x_2 \exists^{\leq d}  \varphi(x_1,x_2, y),
\end{gather}
for some $d \in \NN$, and also
\begin{equation}
\varphi(a_1,a_2,y) \vdash \tp(\xi/a_1a_2).\nonumber
\end{equation}
It follows  that $\varphi(x_1, x_2, r), r \models q$  gives a type-definable family of invertible germs
$\tilde f_{q}\colon p_1\to p_2$ with $f_{\xi}(a_1)=a_2$.
\begin{rem}\label{rem: trans inv technical}
Let $r\models q$, $b_1\models p_1|r$ and $b_2 := f_r(b_1)$. By stationarity of types over $\emptyset$ we then have $b_1 r \equiv a_1 \xi$, and as $\varphi(b_1, x_2, r)$ has a unique solution this implies $b_1 b_2 r \equiv a_1 a_2 \xi$, so
$b_1\ind b_2$, $b_1\in \dcl(b_2,r)$ and
$r\in \acl(b_1,b_2)$.	
\end{rem}
 In particular  $\tilde f_{q}\colon
p_1\to p_2$ is a generically transitive invertible  family.

Consider the equivalence relation $E(y,y')$ on the set of realizations of $q$ given by
$r E r' \Leftrightarrow f_{r}\sim f_{r'}$.  By
the definability of types  it is relatively definable, i.e.~it is
an intersection of an $\emptyset$-definable equivalence relation with $q(y)\cup q(y')$.
Assume  $\xi'\models q$ with $\xi E \xi'$.
We choose $b_1\models p_1|\xi\xi'$
and let
$b_2 := f_{\xi}(b_1)=f_{\xi'}(b_1)$.
By the choice of $\varphi$ we have $\xi, \xi' \in \acl(b_1, b_2)$, hence
 $\xi$ and $\xi'$ are inter-algebraic over $b_1$.
Since $b_1\ind \xi \xi'$
it follows that $\xi$ and $\xi'$ are inter-algebraic over $\emptyset$: as $b_1\ind_{\xi} \xi'$ and $\xi' \in \acl(b_1 \xi)$ implies  $\xi' \ind_{\xi} \xi'$, hence $\xi' \in \acl(\xi)$; and similarly $\xi \in \acl(\xi')$.
Hence
the $E$-class of $\xi$  is finite.
Replacing $\xi$ by $\xi/E$, if needed, we will assume that
the family $\tilde f_{q}\colon p_1\to p_2$ is canonical.
\medskip

 \medskip
We now consider the type-definable family of germs $\tilde f_{r_1}^{-1} {\circ}
\tilde f_{r_2} \colon p_1\to p_1$, $(r_1,r_2)\models
q^{(2)}$. Again let $E$ be a relatively definable equivalence
relation on
$q^{(2)}$ defined as $(r_1,r_2)E(r_3,r_4)$ if and only if
$ f_{r_1}^{-1} {\circ}
f_{r_2} \sim  f_{r_3}^{-1} {\circ}
\ f_{r_4}$.  Let $s(z)$ be the type $q^{(2)}/E$.
We then have (by e.g.~\cite[Remark 3.3.1(1)]{MR2264318}) a canonical family of germs $\tilde h_s\colon p_1\to p_1$ such that
for every $(r_1,r_2)\models q^{(2)}$ there is unique $c\models s(z)$
with $\tilde h_c=\tilde f_{r_1}^{-1} {\circ} \tilde f_{r_2}$.
We will denote this $c$ as $c=\lceil  f_{r_1}^{-1} {\circ}
f_{r_2} \rceil$. Clearly $c\in \dcl(r_1,r_2)$,
$r_1\in \dcl(c,r_2)$ and $r_2\in \dcl(c,r_1)$.

\begin{lem}\label{lem:compind}
For  any $(r_1,r_2)\models q^{(2)}$ and $c := \lceil
f_{r_1}^{-1} {\circ} f_{r_2} \rceil$  we have $r_1\ind c$
and  $r_2\ind c$.
\end{lem}
\begin{proof} It is sufficient to prove the lemma for some
  $(r_1,r_2)\models q^{(2)}$. We take $r_1  := \xi$ from
  our abelian expanded $m$-gon $(\vec{a}, \xi)$ and let $r_2\models q |_{a_1,\dotsc,
  a_m}$. Let $c := \lceil  f_{\xi}^{-1} {\circ} f_{r_2} \rceil$.

  Let $\tilde a := (a_5,\dotsc,a_m)$ and $\eta := \acl(a_1a_3)\cap
  \acl(a_2a_4\dotsc a_m)$. We have an enhanced group
  configuration

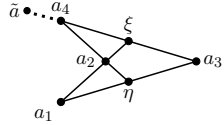
\[
\begin{tikzpicture}[scale=0.9]
\draw[black,line  width=0.2mm] (0,0) -- (1,0.3);
\draw[black,line  width=0.2mm] (1,0.3) -- (2,0.6);
\draw[black,line  width=0.2mm] (2,0.6) -- (1,0.9);
\draw[black,line  width=0.2mm] (1,0.9) -- (0,1.2);
\draw[black,line  width=0.2mm] (0,0) -- (1, 0.9);
\draw[black,line  width=0.2mm] (0,1.2) -- (1, 0.3);

\draw[fill=black] (2,0.6) circle (1.5pt);
\node[scale=0.7, right] at (2, 0.6) {$a_3$};
\draw[fill=black] (1,0.9) circle (1.5pt);
\node[scale=0.7, above] at (1, 0.9) {$\xi$};
\draw[fill=black] (0,1.2) circle (1.5pt);
\node[scale=0.7, above] at (0, 1.2) {$a_4$};
\draw[fill=black] (1,0.3) circle (1.5pt);
\node[scale=0.7, below] at (1, 0.3) {$\eta$};
\draw[fill=black] (0,0) circle (1.5pt);
\node[scale=0.7, below left] at (0, 0) {$a_1$};
\draw[fill=black] (0.66,0.6) circle (1.5pt);
\node[scale=0.7, left] at (0.62, 0.6) {$a_2$};

\draw[black,line  width=0.4mm, dotted] (0, 1.2) -- (-0.5, 1.33);
\draw[fill=black] (-0.5,1.35) circle (1.33pt);
\node[scale=0.7, left] at (-0.5, 1.35) {$\tilde a$};

\end{tikzpicture}
\]

In particular $(a_3, \xi, a_4, \eta, a_1,a_2)$ form a group configuration
\emph{over $\tilde a$}, i.e.~we have a group configuration

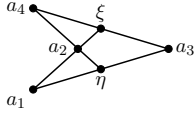
\[
\begin{tikzpicture}[scale=0.9]
\draw[black,line  width=0.2mm] (0,0) -- (1,0.3);
\draw[black,line  width=0.2mm] (1,0.3) -- (2,0.6);
\draw[black,line  width=0.2mm] (2,0.6) -- (1,0.9);
\draw[black,line  width=0.2mm] (1,0.9) -- (0,1.2);
\draw[black,line  width=0.2mm] (0,0) -- (1, 0.9);
\draw[black,line  width=0.2mm] (0,1.2) -- (1, 0.3);

\draw[fill=black] (2,0.6) circle (1.5pt);
\node[scale=0.7, right] at (2, 0.6) {$a_3$};
\draw[fill=black] (1,0.9) circle (1.5pt);
\node[scale=0.7, above] at (1, 0.9) {$\xi$};
\draw[fill=black] (0,1.2) circle (1.5pt);
\node[scale=0.7, left] at (0, 1.2) {$a_4$};
\draw[fill=black] (1,0.3) circle (1.5pt);
\node[scale=0.7, below] at (1, 0.3) {$\eta$};
\draw[fill=black] (0,0) circle (1.5pt);
\node[scale=0.7, below left] at (0, 0) {$a_1$};
\draw[fill=black] (0.66,0.6) circle (1.5pt);
\node[scale=0.7, left] at (0.62, 0.6) {$a_2$};

\end{tikzpicture}
\]
 where any three distinct collinear points form  a triangle over
$\tilde a$,  and any  three distinct non-collinear points form an independent set
over $\tilde a$.

It follows from the proof of the Group Configuration Theorem (e.g.~see
Step (II) in the proof of \cite[Theorem 6.1]{bays2018geometric}) that
$c\ind_{\tilde {a}} \xi$ and $c\ind_{\tilde {a}} r_2$.

We also have $r_2  \ind a_1 \ldots a_m$, hence $r_2 \ind_{a_1 a_2} \tilde{a}$, and as $a_1 a_2 \ind \tilde{a}$ this implies $r_2 a_1 a_2 \ind \tilde{a}$, which together with $\xi \in \acl(a_1a_2)$ implies  $\xi r_2 \ind \tilde a$.
Hence $c\ind \xi$ and $c\ind r_2$.
\end{proof}

This shows that the families of germs $\tilde{f}_{q}:p_1 \to p_2, \tilde{h}_s: p_1 \to p_1$ satisfy the assumptions of the Hrushovski-Weil theorem for bijections
 (see \cite[Lemma 5.4]{bays2018geometric}), applying which
we obtain the following.
\begin{enumerate}[(a)]
\item The family of germs $\tilde h_s\colon p_1\to p_1$ is closed
  under  generic composition and inverse, i.e.~for any independent
  $c_1,c_2\models s(z)$ there exists $c\models s(z)$ with $\tilde h_c=\tilde h_{c_1}{\circ}\tilde h_{c_2}$, and also there is
  $c_3\models s(z)$ with $\tilde h_{c_3}= \tilde h_{c_1}^{-1}$.
  \item  There is a type-definable connected  group $(G, \cdot)$ and a type-definable  set $S$ with a relatively definable faithful transitive action of
        $G$ on $S$ that we will denote by $*: G \times S \to S$, so that $G$, $S$  and the
        action are defined over  the empty set.
      \item There is a definable embedding of $s(z)$ into $G$ as its
        unique generic type, and a definable embedding of $p_1(x_1)$
        into $S$ as its unique generic type, such that the generic
        action of  the family $h_s$ on $p_1$ agrees with that of $G$ on $S$,
        i.e.~for any $c\models s(z)$ and $a\models p_1(x)| c$ we have $h_c(a)=c*a$.
\end{enumerate}

\medskip
\noindent\textbf{Reduction 2.} \emph{Let $r_1,r_2$ be
  independent realizations of $q(y)$, $c := \lceil  f_{r_1}^{-1}
  {\circ} f_{r_2} \rceil$  and $s(z) := \tp(c/\emptyset)$.}

  \emph{From now on we assume that $s(z)$ is the generic type of a
    type-definable
  connected  group $(G, \cdot)$,  the group $G$ relatively
  definably acts faithfully and transitively on a type-definable
   set $S$, the type $p_1(x_1)$ is the generic type of
  $S$,  and generically the action of  $h_s$ on $p_1$ agrees with the
  action of $G$ on $S$,  and $G$, $S$ and the action are definable
  over the empty set.}
\medskip

\subsection{Step 3. Finishing the proof}

We fix an independent  copy  $(\vec e, \xi_e)$ of
 $(\vec a, \xi)$, i.e.~$(\vec e, \xi_e)\equiv (\vec a,  \xi)$
and $\vec e   \xi_e\ind \vec a \xi$.

We denote by $\pi$ the map $\pi\colon q(y)| _{\xi_e} \to s(z)
|_{\xi_e}$ given by $\pi\colon r  \mapsto \lceil  f_{\xi_e}^{-1}
  {\circ} f_{r} \rceil$.  Note that $\pi$ is relatively definable
  over $\acl(\vec e)$.
  Let
  \begin{gather*}
  	t(x_3,\dotsc,x_m) := \tp(a_3,\dotsc,a_m/\emptyset),\\
  	t_\xi(y,x_3,\dotsc,x_m) := \tp(\xi,a_3,\dotsc,a_m/\emptyset).
  \end{gather*}
Note that by Claim~\ref{claim:both-gons} every tuple realizing
$t_\xi$ is an $(m-1)$-gon.

\begin{ntn}
For a tuple $\bar c=(c_3,\dotsc,c_m)$, $j\in \{3,\dotsc,m\}$ and
$\square\in \{ <,\leq,>,\geq\}$, we will denote by $\bar c_{\square j}$
the tuple $\bar c_{\square j}=(c_i :  3\leq i \leq m \land  i\square j)$.
For example, $\bar c_{<j}=(c_3,\dotsc,c_{j-1})$.
We will typically omit the concatenation  sign:  e.g., for $\bar c=(c_3,\dotsc, c_m)$, $\bar
b=(b_3,\dotsc,b_m)$  and $j\in\{3,\dotsc,m\}$ we denote by
$\bar c_{<j},b_j, \bar c_{>j}$ the tuple
$(c_3,\dotsc,c_{j-1},b_j,c_{j+1},\dotsc,c_m)$.

Also in the proof of the next proposition we let $\bar a := (a_3,\dotsc,a_m)$,
$\bar e := (e_3,\dotsc,e_m)$, and continue using $\vec a$ and $\vec e$ to denote the
corresponding $m$-tuples.
\end{ntn}

  \begin{prop}\label{prop:finding-r}
    For each $j \in \{3,\dotsc,m \}$ there exists $r_j\models q(y)|_{\xi_e}$ such that
    $\models t_\xi(r_j,\bar e_{<j}, a_j,\bar e_{>j})$
and  $\pi(\xi)=\pi(r_m)\cdot \pi(r_{m-1})  \cdot \dotsc \cdot \pi(r_3)$.
  \end{prop}

    We will choose such  $r_j$ by reverse induction on $j$. Before proving Proposition~\ref{prop:finding-r} we  first establish the following lemma and its corollary
   that will provide the induction step.

    \begin{lem} \label{lem:r-ind}
      For  $j\in \{4,\dotsc, m\}$ there exist $r_{<j}, r_j, r_{\leq
        j}$, each realizing $q(y)|_{\xi_e}$,
      such that
      $$\models t_\xi(r_{< j}, \bar a_{<j},\bar  e_{\geq j}),
\models t_\xi(r_j,\bar e_{<j},a_j,\bar e_{>j}), 
      \models t_\xi(r_{\leq j}, \bar a_{\leq j}, \bar e_{>j})$$
      and $\pi(r_{\leq j})=\pi(r_j)\cdot \pi(r_{< j})$.
    \end{lem}
    \begin{proof}

First we note that the condition
$r_{<j}, r_j, r_{\leq
        j}\models q(y)|_{\xi_e}$
         can be relaxed to $r_{<j}, r_j, r_{\leq
        j}\models q(y)$ by stationarity of $q$, since for $j\in
      \{4,\dotsc, m\}$ and $r\models q(y)$ satisfying one of
$\models t_\xi(r, \bar a_{<j},\bar  e_{\geq j})$,
$\models t_\xi(r,\bar e_{<j},a_j,\bar e_{>j})$,
      $\models t_\xi(r, \bar a_{\leq j}, \bar e_{>j})$ we have $r\ind \xi_e$.
Indeed, assume e.g.~$\models  t_\xi(r, \bar a_{<j},\bar  e_{\geq
  j})$.  We have $r\in \acl(\bar a_{<j},\bar  e_{\geq
  j})$ and $\xi_e\in \acl(e_3,\dotsc, e_m)$. By assumption 
  $$\{e_3,\dotsc,
e_m,a_3, \dotsc,a_m\}$$
 is an independent set, hence  we obtain $r\ind_{\bar
  e_{\geq j}} \xi_e$. Using $\xi_e\ind \bar e_{\geq j}$ we conclude
$r\ind \xi_e$. The other two cases are similar.

Let
$\eta := \acl_{\bar{e}_{>j}}(e_1,e_j)\cap   \acl_{{\bar{e}_{>j}}}(e_2,e_3,\dotsc,e_{j-1})$.
Note that $\acl(\eta)=\eta$, hence all types over $\eta$ are
stationary, and ${\bar e}_{>j}  \in \eta$.

Then one verifies by basic forking calculus that
\begin{equation}
  \label{eq:r}
  \begin{tikzpicture}[scale=0.9, baseline=(current  bounding  box.center)]
\draw[black,line  width=0.2mm] (0,0) -- (1,0.3);
\draw[black,line  width=0.2mm] (1,0.3) -- (2,0.6);
\draw[black,line  width=0.2mm] (2,0.6) -- (1,0.9);
\draw[black,line  width=0.2mm] (1,0.9) -- (0,1.2);
\draw[black,line  width=0.2mm] (0,0) -- (1, 0.9);
\draw[black,line  width=0.2mm] (0,1.2) -- (1, 0.3);

\draw[fill=black] (2,0.6) circle (1.5pt);
\node[scale=0.7, above] at (2, 0.6) {$e_j$};
\draw[fill=black] (1,0.9) circle (1.5pt);
\node[scale=0.7, above] at (1, 0.9) {$\xi_e$};
\draw[fill=black] (0,1.2) circle (1.5pt);
\node[scale=0.7, above] at (0, 1.2) {$e_{j-1}$};
\draw[fill=black] (1,0.3) circle (1.5pt);
\node[scale=0.7, below] at (1, 0.3) {$\eta$};
\draw[fill=black] (0,0) circle (1.5pt);
\node[scale=0.7, below left] at (0, 0) {$e_1$};
\draw[fill=black] (0.66,0.6) circle (1.5pt);
\node[scale=0.7, left] at (0.62, 0.6) {$e_2$};

\draw[black,line  width=0.4mm, dotted] (0, 1.2) -- (-0.5, 1.33);
\draw[fill=black] (-0.5,1.35) circle (1.33pt);
\node[scale=0.7, left] at (-0.5, 1.35) {$\bar e _{<j-1}$};
\end{tikzpicture}
\end{equation}
is an enhanced group configuration \emph{over ${\bar{e}_{>j}}$}. Namely,
\begin{itemize}
\item $(e_j, \xi_e, e_{j-1})$ and $(\eta, e_2, e_{j-1})$ are
  triangles over $\bar{e}_{<j-1},{\bar{e}_{>j}}$;
\item $(e_1,\eta,e_j)$ and $(e_1, e_2, \xi_e)$ are triangles over ${\bar{e}_{>j}}$;
\item for any non-collinear triple in $e_1,e_2,e_{j-1},e_j, \eta,
  \xi_e$, the set given
  by it and $\bar e_{<j-1}$ is independent over ${\bar{e}_{>j}}$.
\end{itemize}
In addition, $e_1e_2\xi_e \ind {\bar{e}_{>j}}$ and $f_{\xi_e}(e_1)=e_2$.

\medskip
The triple $\eta,e_j,e_{j-1}$ is non-collinear, hence
$\eta\ind_{{\bar{e}_{>j}}} e_3 \dotsc e_j$.  Since 
$${\bar{e}_{>j}}\ind
e_3\dotsc e_j,$$
  this implies  $\eta\ind  e_3 \dotsc e_j$. Since also
$\eta\ind  a_3 \dotsc a_j$, by stationarity of types over $\emptyset$ we have
$a_3\dotsc a_j \equiv_{\eta} e_3\dotsc e_j$. Hence  there exist $r_{\leq j}$,
$b_1$, $b_2$ such that the diagram
\begin{equation}
  \label{eq:r1}
  \begin{tikzpicture}[scale=0.9, baseline=(current  bounding  box.center)]
\draw[black,line  width=0.2mm] (0,0) -- (1,0.3);
\draw[black,line  width=0.2mm] (1,0.3) -- (2,0.6);
\draw[black,line  width=0.2mm] (2,0.6) -- (1,0.9);
\draw[black,line  width=0.2mm] (1,0.9) -- (0,1.2);
\draw[black,line  width=0.2mm] (0,0) -- (1, 0.9);
\draw[black,line  width=0.2mm] (0,1.2) -- (1, 0.3);

\draw[fill=black] (2,0.6) circle (1.5pt);
\node[scale=0.7, above] at (2, 0.6) {$a_j$};
\draw[fill=black] (1,0.9) circle (1.5pt);
\node[scale=0.7, above] at (1, 0.9) {$r_{\leq j}$};
\draw[fill=black] (0,1.2) circle (1.5pt);
\node[scale=0.7, above] at (0, 1.2) {$a_{j-1}$};
\draw[fill=black] (1,0.3) circle (1.5pt);
\node[scale=0.7, below] at (1, 0.3) {$\eta$};
\draw[fill=black] (0,0) circle (1.5pt);
\node[scale=0.7, below left] at (0, 0) {$b_1$};
\draw[fill=black] (0.66,0.6) circle (1.5pt);
\node[scale=0.7, left] at (0.62, 0.6) {$b_2$};

\draw[black,line  width=0.4mm, dotted] (0, 1.2) -- (-0.5, 1.33);
\draw[fill=black] (-0.5,1.35) circle (1.33pt);
\node[scale=0.7, left] at (-0.5, 1.35) {$\bar a_{<j-1}$};
\end{tikzpicture}
\end{equation}
is isomorphic over $\eta$
to the  diagram \eqref{eq:r}. I.e., there is an automorphism of $\MM$ fixing $\eta$ (hence also
${\bar{e}_{>j}}$) and mapping  \eqref{eq:r1} to \eqref{eq:r}.

It follows from  the choice of the tuple $(\vec{e},\xi_e)$, diagrams \eqref{eq:r}, \eqref{eq:r1} and their
isomorphism over $\eta$ that $e_1e_j \ind_\eta e_2\dotsc e_{j-1}$ and
$b_1a_j \equiv_\eta e_1 e_j$.
Since $a_j\ind e_1\dotsc e_m$  we have $a_j\ind_\eta e_2\dotsc  e_{j-1}$. As
$b_1\in \acl(a_j\eta)$, we have 
$$b_1a_j\ind_{\eta}  e_2\dotsc
e_{j-1} .$$
 Since all  types over $\eta$ are stationary,  this implies
$$b_1a_j e_2\dotsc e_{j-1} \equiv_\eta e_1e_j e_2\dotsc  e_{j-1},$$
 hence
there exists $r_j$ such that
the diagram
\begin{equation}
  \label{eq:r2}
  \begin{tikzpicture}[scale=0.9,  baseline=(current  bounding  box.center)]
\draw[black,line  width=0.2mm] (0,0) -- (1,0.3);
\draw[black,line  width=0.2mm] (1,0.3) -- (2,0.6);
\draw[black,line  width=0.2mm] (2,0.6) -- (1,0.9);
\draw[black,line  width=0.2mm] (1,0.9) -- (0,1.2);
\draw[black,line  width=0.2mm] (0,0) -- (1, 0.9);
\draw[black,line  width=0.2mm] (0,1.2) -- (1, 0.3);

\draw[fill=black] (2,0.6) circle (1.5pt);
\node[scale=0.7, above] at (2, 0.6) {$a_j$};
\draw[fill=black] (1,0.9) circle (1.5pt);
\node[scale=0.7, above] at (1, 0.9) {$r_j$};
\draw[fill=black] (0,1.2) circle (1.5pt);
\node[scale=0.7, above] at (0, 1.2) {$e_{j-1}$};
\draw[fill=black] (1,0.3) circle (1.5pt);
\node[scale=0.7, below] at (1, 0.3) {$\eta$};
\draw[fill=black] (0,0) circle (1.5pt);
\node[scale=0.7, below left] at (0, 0) {$b_1$};
\draw[fill=black] (0.66,0.6) circle (1.5pt);
\node[scale=0.7, left] at (0.62, 0.6) {$e_2$};

\draw[black,line  width=0.4mm, dotted] (0, 1.2) -- (-0.5, 1.33);
\draw[fill=black] (-0.5,1.35) circle (1.33pt);
\node[scale=0.7, left] at (-0.5, 1.35) {$\bar e_{<j-1}$};
\end{tikzpicture}
\end{equation}
is isomorphic to the  diagram \eqref{eq:r} over $\eta$.

A similar argument with the roles of the  $a$'s and the $e$'s  interchanged shows
that $e_1 e_j a_2 \dotsc a_{j-i} \equiv_{\eta} b_1 a_j a_2 \dotsc a_{j-1} $, hence there exists $r_{<j}$ such that the diagram
\begin{equation}
  \label{eq:r3}
  \begin{tikzpicture}[scale=0.9, baseline=(current  bounding  box.center)]
\draw[black,line  width=0.2mm] (0,0) -- (1,0.3);
\draw[black,line  width=0.2mm] (1,0.3) -- (2,0.6);
\draw[black,line  width=0.2mm] (2,0.6) -- (1,0.9);
\draw[black,line  width=0.2mm] (1,0.9) -- (0,1.2);
\draw[black,line  width=0.2mm] (0,0) -- (1, 0.9);
\draw[black,line  width=0.2mm] (0,1.2) -- (1, 0.3);

\draw[fill=black] (2,0.6) circle (1.5pt);
\node[scale=0.7, above] at (2, 0.6) {$e_j$};
\draw[fill=black] (1,0.9) circle (1.5pt);
\node[scale=0.7, above] at (1, 0.9) {$r_{<j}$};
\draw[fill=black] (0,1.2) circle (1.5pt);
\node[scale=0.7, above] at (0, 1.2) {$a_{j-1}$};
\draw[fill=black] (1,0.3) circle (1.5pt);
\node[scale=0.7, below] at (1, 0.3) {$\eta$};
\draw[fill=black] (0,0) circle (1.5pt);
\node[scale=0.7, below left] at (0, 0) {$e_1$};
\draw[fill=black] (0.66,0.6) circle (1.5pt);
\node[scale=0.7, left] at (0.62, 0.6) {$b_2$};

\draw[black,line  width=0.4mm, dotted] (0, 1.2) -- (-0.5, 1.33);
\draw[fill=black] (-0.5,1.35) circle (1.33pt);
\node[scale=0.7, left] at (-0.5, 1.35) {$\bar a_{<j-1}$};
\end{tikzpicture}
\end{equation}
is isomorphic  to the  diagram \eqref{eq:r} over $\eta$.

From the choice of $(\vec{e}, \xi_e)$ and the isomorphisms of the diagrams we have
\begin{equation}
  \label{eq:3}
( f_{r_{<j}}\co f_{\xi_e}^{-1}\co f_{r_j}) (b_1) = b_2=f_{r_{\leq
      j}} (b_1).
\end{equation}

We claim that $b_1\ind r_{<j},\xi_e,r_j, r_{\leq j}$. Indeed, as
\begin{gather*}
	r_{<j},\xi_e,r_j, r_{ \leq j}\in
\acl(a_3,\dotsc,a_m,e_3,\dotsc e_m) \textrm{  and}\\
e_2\ind
a_3,\dotsc,a_m,e_3,\dotsc e_m,
\end{gather*}
we obtain
$e_2\ind r_{<j},\xi_e,r_j, r_{ \leq j}$, hence $e_2\ind_{r_j}r_{<j},\xi_e, r_{ \leq j}$.
As $b_1\in \acl(e_2, r_j)$ we have $b_1\ind_{r_j}r_{<j},\xi_e, r_{
  \leq j}$. Using $b_1\ind r_j$ we conclude
\begin{equation}
  \label{eq:2}
b_1\ind r_{<j},\xi_e,r_j,r_{ \leq j}.
\end{equation}

It follows from \eqref{eq:3} and \eqref{eq:2} that
\[ \tilde f_{r_{<j}}\co \tilde f_{\xi_e}^{-1}\co \tilde f_{r_j}
=\tilde f_{r_{\leq
      j}}, \]
and hence
 \begin{equation}
    \label{eq:1}
  \Bigl((\tilde f_{\xi_e}^{-1}\co \tilde  f_{r_{<j}})\co
  (\tilde f_{\xi_e}^{-1}
  \co \tilde f_{r_j})\Bigr)
  =\tilde f_{\xi_e}^{-1}\co  \tilde f_{r_{\leq j}}.
  \end{equation}

As noted at the beginning of the proof, we have  that $r_j,
r_{<j},r_{\leq  j}\models q(y)|_{\xi_e}$, and
we define
$c_0,c_1,c_2 \models s(z)|_{\xi_e}$ as follows:
\begin{gather*}
c_0  := \pi(r_{<j})  =\lceil  f_{\xi_e}^{-1}
{\circ} f_{r_{<j}}\rceil,\\
c_1 := \pi(r_{j}) =\lceil  f_{\xi_e}^{-1}
{\circ} f_{r_{j}}\rceil, \\
c_2  := \pi(r_{\leq j})  =\lceil  f_{\xi_e}^{-1}
{\circ} f_{r_{\leq j}}\rceil.
\end{gather*}

By \eqref{eq:1}, to conclude  that $c_2=c_0\cdot c_1$ in  $G$ and finish the proof
of the lemma it is sufficient to show that $c_0\ind  c_1$.

As $r_{<j}\in \acl({\bar a}_{<j},{\bar e}_{\geq j})$,
$r_j,\xi_e\in \acl(\bar e,a_j)$, and
$\{e_3,\dotsc,e_m,a_j,{\bar a}_{<j}\}$
 is an
independent  set, we have $r_{<j}\ind_{{\bar e}_{\geq j}} r_j\xi_e$.
Since  $r_{<j}\ind {\bar e}_{\geq j}$ (as $(r_{<j}, {\bar a}_{<j},
{\bar e}_{\geq j})$ is an $(m-1)$-gon) we also have $r_{<j}\ind_{\xi_e} r_j$.
It follows then that $c_0\ind_{\xi_e} c_1$.  Since, by
Lemma~\ref{lem:compind}, $c_0\ind \xi_e$ we have $c_0\ind c_1$.

This concludes the proof of Lemma~\ref{lem:r-ind}.
\end{proof}

\begin{cor}
  \label{cor:rrr}
  For  any $j\in \{4,\dotsc, m\}$,
let $r_{\leq j}\models q(y)|_{\xi_e}$
with 
$$\models t_\xi(r_{\leq j}, \bar a_{\leq j}, \bar e_{>j}).$$
Then
there exist $r_{<j}, r_j \models q(y)|_{\xi_e}$
      such that
     $$\models t_\xi(r_{< j}, \bar a_{<j},\bar  e_{\geq j}), \models t_\xi(r_j,\bar e_{<j},a_j,\bar e_{>j})$$
      and $\pi(r_{\leq j})=\pi(r_j)\cdot \pi(r_{< j})$.
\end{cor}

\begin{proof}  It is
  sufficient to show that for any $r,r'$ with
  $\models t_\xi(r, \bar a_{\leq j},\bar  e_{> j})$,
    $\models t_\xi(r', \bar a_{\leq j},\bar  e_{> j})$
    we have $r\bar a\bar e  \equiv r'\bar a\bar e $. Indeed, given any $(r'_{\leq j},  r'_j, r'_{>j})$  satisfying  the conclusion of Lemma~\ref{lem:r-ind}, we then have an automorphism $\sigma$ of $\mathbb{M}$ fixing $\bar{a}\bar{e}$ with $\sigma(r'_{\leq j}) = r_{\leq j}$; as the map $\pi$   is relatively definable over $\acl(\bar{e})$, it  then follows that $r_{<j} := \sigma(r'_{<j}), r_{j} :=  \sigma(r'_j)$ satisfy the requirements.

We have $r  \bar a_{\leq j} \bar  e_{>j} \equiv r'  \bar a_{\leq j} \bar
e_{> j}$. As $\bar{e} \ind \bar{a}$ and each of $\bar{e}, \bar{a}$ is an  $(m-2)$-tuple  from the corresponding $m$-gon,  we get $\bar{a}_{\leq j} \bar{e}_{>j} \ind \bar{a}_{>j} \bar{e}_{\leq j}$. Also $r,r' \in \acl(\bar{a}_{\leq j}  \bar{e}_{>j})$,  as any realization of $t_{\xi}$ is an $(m-1)$-gon, hence 
$$rr'  \bar{a}_{\leq j}   \bar{e}_{>j} \ind\bar{a}_{>j} \bar{e}_{\leq j}.$$
 As all types over the empty set are
stationary, we conclude
$r\bar a\bar e\equiv r'\bar a\bar e$.
 \end{proof}
 We can now finish the proof of Proposition~\ref{prop:finding-r}.

   \begin{proof}[Proof of Proposition~\ref{prop:finding-r}]

 We start with $r_{\leq m} := \xi$. Applying Corollary~\ref{cor:rrr}
 with $j  :=  m$, we obtain $r_m$ and $r_{<m}$ with
   $\pi(\xi)=\pi(r_m)\cdot \pi(r_{<m})$.

 Applying Corollary~\ref{cor:rrr} again with $j := m-1$ and $r_{\leq m-1} :=  r_{<
   m}$
we obtain $r_{m-1}$ and $r_{<m-1}$ with
$\pi(\xi)=\pi(r_m)\cdot \pi(r_{m-1})\cdot\pi(r_{<m-1})$.

Continuing this process with $j   := m-2, \dotsc,   4$ we obtain  some
$$r_{m-2},\dotsc,r_{4},r_{<4}$$
 with
$\pi(\xi)=\pi(r_m)\cdot \dotsc \cdot \pi(r_4)\cdot\pi(r_{<4})$.
We take  $r_3 := r_{<4}$,  which  concludes the proof of the proposition.
\end{proof}

\begin{prop}\label{prop:1-2}
  There exist $r_1,r_2\models q(y)|_{\xi_e}$ such that $f_{r_1}(a_1)=e_2$,
  $f_{r_2}(e_1)=a_2$ and $\pi(r_2)\cdot \pi(r_1)=\pi(\xi)$.
  \end{prop}
  \begin{proof}

  We choose $r_1\models q(y)$  with $f_{r_1}(a_1)=e_2$ (possible  by generic transitivity: as $a_1 \ind e_2$, hence $a_1 e_2 \equiv a_1 a_2$ by stationarity of types over $\emptyset$; and as $f_{\xi}(a_1) = a_2$, we can take $r_1$ to be the image of $\xi$ under the automorphism of $\MM$ sending $(a_1, a_2)$ to $(a_1,  e_2)$). We also have $r_1\ind \xi_e$ ($a_1 \ind \vec{e}$ and $e_2 \ind \bar{e}$  by the choice of $\vec{e}$, so $a_1 e_2 \ind \bar{e}$;
 as $r_1 \in \acl(a_1,e_2), \xi_e \in \acl(\bar{e})$, we conclude $r_1 \ind \xi_e$), hence $r_1 \models q|_{\xi_e}$  by stationarity again.

 Similarly $\xi\ind \xi_e r_1$, hence
 $\xi \ind \lceil  f_{r_1}^{-1}
{\circ} f_{\xi_e}\rceil$. By
Lemma~\ref{lem:compind} we also  have $r_1 \ind  \lceil  f_{r_1}^{-1}
{\circ} f_{\xi_e}\rceil$. By stationarity of $q$ this implies $\xi \equiv_{\lceil  f_{r_1}^{-1}
{\circ} f_{\xi_e}\rceil}  r_1$,  so there exists some $r_2 \models q$ such that $\xi r_2 \equiv_{\lceil  f_{r_1}^{-1}
{\circ} f_{\xi_e}\rceil}  r_1 \xi_e$. Hence
\[\tilde f_\xi^{-1} {\circ} \tilde f_{r_2}= \tilde f_{r_1}^{-1}
{\circ} \tilde f_{\xi_e },\]
equivalently
\begin{equation}
  \label{eq:4}
  \tilde f_{r_2}= \tilde  f_\xi{\circ} \tilde f_{r_1}^{-1}
{\circ} \tilde f_{\xi_e}.
\end{equation}
In particular, $r_2\in \acl(\xi,r_1,\xi_e)$.

We claim that $e_1\ind r_2 \xi r_1 \xi_e$.
Since $\xi_e\in \acl(e_1,e_2)$, $r_1\in \acl(a_1,e_2)$ and $\{a_1,e_1,e_2\}$ is an
independent set, we have $r_1\ind_{e_2} e_1 \xi_e$. Using $r_1\ind e_2$ we
deduce $r_1 \ind e_1\xi_e$. As $\xi_e\ind e_1$, it implies that $\{r_1,
e_1,\xi_e\}$ is an independent set.  We have $r_1,
e_1,\xi_e \in \acl(a_1,e_1,e_2)$ and $\xi\in\acl(a_1,a_2)$. Using
independence  of $a_1,a_2,e_1,e_2$ we obtain  $\xi\ind_{a_1}
e_1\xi_e r_ 1$.
Since $\xi\ind a_1$, we have that  $\xi\ind e_1r_1\xi_e$, hence
$\{\xi,e_1,r_1,\xi_e\}$ is an independent set and $e_1\ind
\xi r_1 \xi_e$. As $r_2\in \acl(\xi, r_1, \xi_e)$ we can conclude
$e_1\ind r_2 \xi r_1 \xi_e$.

It  then follows from \eqref{eq:4} that
\[   f_{r_2}(e_1)=  (f_\xi{\circ} f_{r_1}^{-1}
{\circ} f_{\xi_e})(e_1)=a_2, \]
so $f_{r_2}(e_1)=a_2$.

It also follows from~\eqref{eq:4} that
\[ \Bigl((\tilde f_{\xi_e}^{-1}\co \tilde  f_{r_2})\co
  (\tilde f_{\xi_e}^{-1}
  \co \tilde f_{r_1})\Bigr)
  =\tilde f_{\xi_e}^{-1}\co  \tilde f_{\xi}. \]

We let
\[
c_1 := \pi(r_{1}) =\lceil  f_{\xi_e}^{-1}
{\circ} f_{r_{1}}\rceil \text{ and } c_2  := \pi(r_{2})  =\lceil  f_{\xi_e}^{-1}
{\circ} f_{r_{2}}\rceil.  \]
To show that $c_2\cdot c_1=\pi(\xi)$ and finish the proof of  the proposition it
is sufficient to show that $c_1\ind c_2$.

Since $r_1\in \acl(a_1,e_2)$, $r_2\in \acl(e_1,a_2)$ (by Remark \ref{rem: trans inv technical}, as by the above we have $r_2 \models q, e_1\ind r_2$ and $f_{r_2}(e_1) = a_2$) and
$\xi_e\in \acl(e_1,e_2)$, we obtain $r_1\ind_{e_2} r_2\xi_e$.
Using $r_1\ind e_2$ we deduce $r_1\ind r_2\xi_e$,
hence $r_1\ind_{\xi_e} r_2$. It follows then that $c_1\ind_{\xi_e}
c_2$
and, as $c_1\ind \xi_e$, we obtain $c_1\ind c_2$.
\end{proof}

Combining Propositions \ref{prop:1-2} and \ref{prop:finding-r},
we obtain some $r_1,\dotsc,r_m\models q(y)|_{\xi_e}$ such that each $r_i$ is
inter-algebraic with $a_i$ over $\{  e_1,\dotsc,e_m \}$ and
\[ \pi(r_2)\cdot \pi(r_1)= \pi(r_m)\cdot\dotsc\cdot \pi(r_3).\]
Obviously each $r_i$ is also inter-algebraic over $\{e_1,\dotsc,e_m\}$ with
$\pi(r_i)$.

Thus, after a base  change to $\{ e_1, \dotsc,  e_m \}$ and
inter-algebraically replacing $a_1$ with $\pi(r_1)^{-1}$,
$a_2$ with $\pi(r_2)^{-1}$, and  $a_i$ with $\pi(r_i)$ for $i \in \{  3,\dotsc,m \}$, and using that permuting the elements of an abelian  $m$-gon we  still  obtain an abelian $m$-gon, we achieve  the
following.

\medskip
\noindent\textbf{Reduction 3.} \emph{
We may assume that $a_1,\dotsc,a_m$ realize the generic type $s(z)$ of
a connected group  $G$  that  is  type-definable over the  empty set,
with $a_1\cdot a_2\cdot a_m\cdot\dotsc \cdot a_3= 1_{G}$.}
\medskip

To finish the proof of Theorem~\ref{thm: main ab mgon gives grp} it only
remains to show that the group $G$ is abelian.
We  deduce it from the Abelian Group Configuration Theorem, more precisely
\cite[Lemma~C.1]{bays2017model}.

\begin{claim}\label{claim:abelian}
Let $G$ be a connected group type-definable over the  empty set,
$m\geq 4$ and
$g_1,\dotsc,g_m$ are generic elements of $G$ such that $g_1,\dots,g_m$
form an abelian $m$-gon and $g_1\cdot\dotsc \cdot g_m=1_{G}$. Then the
group $G$ is abelian.
\end{claim}
\begin{proof}
  Let $B :=  \acl(g_5,\dotsc,g_m)$.
We have that $g_1, \ldots, g_4$ are generics of  $G$ over $B$, and they   form an
abelian $4$-gon over $B$. Since $g_4$ is inter-algebraic over $B$  with
$g_1{\cdot} g_2 {\cdot} g_3$, we have that $g_1,g_2,g_3,g_1{\cdot}g_2{\cdot}g_3$ form an abelian
$4$-gon over $B$.  Let $D  := \acl_{B}(g_1,g_3)\cap
\acl_{B}(g_2,g_1{\cdot}g_2{\cdot}g_3)$.
We have $g_1,g_3\ind_D g_2, g_1{\cdot}g_2{\cdot}g_3$, hence
\[g_1{\cdot}g_2{\cdot}g_3\in \acl_{B}(g_2,D)=\acl_{B} \big(g_2, \acl_{B}(g_1,g_3)\cap
\acl_{B}(g_2,g_1{\cdot}g_2{\cdot}g_3) \big). \]
By \cite[Lemma~C.1]{bays2017model}, the group $G$ is abelian.
  \end{proof}

\section{Main theorem in the stable case}\label{sec: main thm stable}

Throughout the section  we work in  a complete theory $T$ in a language $\mathcal{L}$. We fix an $|\CL|^+$-saturated model $\CM = (M, \ldots )$ of $T$,
and also choose a large saturated elementary extension $\MM$ of
$\CM$. We say that a subset $A$ of $\CM$ is \emph{small} if $|A| \leq |\CL|$.
Given a definable set $X$ in $\CM$, we will often view it as a definable subset of $\mathbb{M}$, and sometimes write explicitly $X(\mathbb{M})$ to denote the set of tuples in $\mathbb{M}$ realizing the formula defining $X$.

\subsection{On the notion of $\mfp$-dimension}
\label{sec:notion-mfp-dimension}

We introduce a basic notion of dimension in an arbitrary theory imitating the topological definition of dimension in $o$-minimal structures, but localized at a given tuple of commuting definable global types.
We will see that it enjoys definability properties that may fail for Morley rank even in nice theories such as $\operatorname{DCF}_0$.

\begin{defn}
	If $X$ is a definable set in $\CM$ and $\CF$ is a family of subsets of $X$, we say that $\CF$ is a \emph{definable family} (over a set of parameters $A$) if there exists a definable set $Y$ and a definable set $D \subseteq X \times Y$ (both defined over $A$) such that $\CF = \{D_b : b \in Y \}$, where $D_b = \{a \in X : (a,b) \in D\}$ is the fiber of $D$ at $b$.
\end{defn}


\begin{defn}\label{def: p-pairs}
\begin{enumerate}
	\item By a \emph{$\mfp$-pair} we
mean a pair $(X,\mfp_X)$  where $X$ is an $\emptyset$-definable set and $\mfp_X  \in S(\CM)$ is an $\emptyset$-definable
stationary type on $X$.
\item Given $s \in \mathbb{N}$, we say that $(X_i,\mfp_i)_{i \in [s]}$ is a  \emph{$\mfp$-system} if each $(X_i,\mfp_i)$ is a $\mfp$-pair and the types $\mfp_1, \ldots, \mfp_s$ commute, i.e.~$\mfp_i \otimes \mfp_j = \mfp_j \otimes \mfp_i$ for all $i,j \in [s]$.
\end{enumerate}
	
\end{defn}

\begin{sample}
	Assume $T$ is a stable theory, $(\mfp_i)_{i \in [s]}$ are arbitrary types over $\CM$ and $X_i \in \mfp_i$ are arbitrary definable sets.  By local character we can choose a model $\CM_0 \preceq \CM$ with $|\CM_0| \leq |\CL|$ such that each $\mfp_i$ is definable (and stationary) over $\CM_0$ and $X_i, i \in [s]$ are definable over $\CM_0$. The types $(\mfp_i)_{i \in [s]}$ automatically commute in a stable theory. Hence, naming the elements of $\CM_0$ by constants, we obtain a $\mfp$-system.
\end{sample}

Assume now that $(X_i,\mfp_i)_{i \in [s]}$ is a $\mfp$-system. Given $u \subseteq [s]$, we let $\pi_u : \prod_{i \in [s]} X_i \to \prod_{i \in u} X_i$ be the projection map. For $i \in [s]$, we let $\pi_i := \pi_{\{i\}}$. Given $u,v \subseteq [s]$ with $u \cap v = \emptyset$, $a = (a_i : i \in u) \in \prod_{i \in u} X_i$ and $b = (b_i : i \in v) \in \prod_{i \in v} X_i$, we write $a \oplus b$ to denote the tuple $c = (c_i : i \in u \cup v) \in \prod_{i \in u \cup v} X_i$ with $c_i = a_i$ for $i \in u$ and $c_i = b_i$ for $i \in v$.
Given $Y \subseteq \prod_{i \in [s]} X_i$, $u \subseteq [s]$ and $a  \in \prod_{i \in u} X_i$, we write $Y_a := \{ b \in \prod_{i \in [s] \setminus u} X_i : a \oplus b \in Y\}$ to denote the fiber of $Y$ above $a$.

\begin{sample}
	If $\CF$ is a definable family of subsets of $\prod_{i \in [s]} X_i$ and $u \subseteq [s]$, then $ \left\{\pi_u(F) : F \in \CF\right\}$ and $\left\{F_a : F \in \CF, a \in \prod_{i \in [s] \setminus u} X_i \right\}$ are  definable families of subsets of $\prod_{i \in u} X_i$ (over the same set of parameters).
\end{sample}

\begin{defn}\label{def: p-gen, dim, etc}  Let $\bar a = (a_1,\dotsc, a_s)\in X_1\times
  \dotsb\times X_s$ and $A$  a small subset of $\CM$.
  \begin{enumerate}
  \item We say that $\bar a$  \emph{is  $\mfp$-generic in  $X_1\times \dotsb\times X_s$
  over $A$}  if
$(a_1,\dotsc,a_s)\models \mfp_1\otimes \dotsb\otimes \mfp_s {\restriction} A$.

\item
\begin{enumerate}
\item   For $k \leq s$  we write $\dimp(\bar a /A) \geq  k$ if
  for \emph{some} $u \subseteq [s]$ with $|u| \geq k$ the tuple
  $\pi_u(\bar a)$  is $\mfp$-generic (with respect to the corresponding $\mfp$-system $\{(X_i, \mfp_i) : i \in u \}$).

\item     As usual,  we  define $\dimp(\bar a /A) =k$  if  $\dimp(\bar a/A)\geq k$
  and it is not true that $\dimp(\bar a /A) \geq k+1$.
\end{enumerate}

\item If $q(\bar{x}) \in S(A)$ and $q(\bar{x}) \vdash \bar{x} \in X_1 \times \ldots \times X_s$, we write $\dim_{\mfp}(q) := \dim_{\mfp}(\bar{a}/A)$ for some (equivalently, any) $\bar{a} \models q$.

\item
  For a subset $Y \subseteq X_1\ttimes \dotsb \times X_s$ definable
  over $A$,
  we define
  \begin{gather*}
  	\dimp(Y) := \max\left\{ \dimp(\bar a/A)  \colon \bar a \in Y\right\}\\
  	= \max\left\{ \dimp(q)  \colon q \in S(A), Y \in q \right\},
  \end{gather*}
  note that this does not depend on the set $A$ such that $Y$ is $A$-definable.

 \item
As usual, for a definable subset $Y \subseteq X_1\ttimes \dotsb \times X_s$
we  say that $Y$
is \emph{a $\mfp$-generic subset of $X_1\times \dots \times X_s$} if
$\dimp(Y)=s$ (equivalently, $Y$ is contained in $\mfp_1\otimes
\dotsb\otimes \mfp_s.$)
\end{enumerate}
If $A=\emptyset$ we will omit it.
\end{defn}

\begin{rem}\label{rem: p-gen tuple def}
	It follows from the definition that for a  definable set $Y \subseteq X_1\ttimes
\dotsb \ttimes X_s$,  $\dimp(Y)$ is the maximal $k$ such that the
projection of $Y$ onto some $k$ coordinates is $\mathfrak{p}$-generic.
As usual, for a  definable $Y \subseteq X_1\tdt X_s$  and small $A\subseteq \CM$
we say that an element $a\in Y$ is \emph{generic in $Y$ over $A$} if
$\dimp( a/A)=\dimp(Y)$.
\end{rem}

\begin{rem}
	It also follows that if $\CN \succeq \CM$ is an arbitrary $|\CL|^+$-saturated model and $\mfp'_i := \mfp_i|_{\CN} \in  S (\CN)$ is the unique definable extension, for $i \in [s]$,  then $(X_i(\CN), \mfp'_i)_{i \in [s]}$ is a $\mfp$-system  in $\CN$,  and for every definable subset $Y \subseteq X_1 \times \ldots \times X_s$  in $\CM$ we have $\dimp(Y) = \dimp(Y(\CN))$, where  the latter is calculated in $\CN$ with respect to this $\mfp$-system.
\end{rem}

\begin{claim}\label{claim:p-dim-def}
 Let  $\CF$ be   a definable (over $A$) family of subsets of $X_1\times \dotsb
 \times X_s$ and $k\leq s$. Then
the family
  \[ \left\{ F\in \CF \colon \dimp(F) = k \right\} \]
  is definable (over $A$ as well).
\end{claim}
\begin{proof}
Assume that $\CF = \{ D_b : b \in Y \}$ for some definable $Y$ and definable $D \subseteq (X_1 \times \ldots \times X_s) \times Y$.
Given $0 \leq k \leq s$, let $Y_k := \{ b \in Y : \dim_{\mfp}(D_b) = k \}$, it suffices to show that $Y_k$ is definable. As every $\mfp_i$ is definable, for every $u \subseteq [s]$, the type $\mfp_u = \bigotimes_{i \in u} \mfp_i$ is also definable. In particular, there is a definable (over any set of parameters containing the parameters of $Y$ and $D$) set $Z_u \subseteq Y$ such that for any $b \in Y$, $\pi_{u}(D_b) \in \mfp_u  \iff b \in Z_u$. Then $Y_k$ is definable as
\begin{gather*}
	Y_k = \left( \bigvee_{u \subseteq [s], |u| = k} b \in Z_u \right) \land \left( \bigwedge_{u \subseteq [s], |u| > k} b \notin Z_u \right).\qedhere
\end{gather*} 
\end{proof}

The following lemma shows that  $\mfp$-dimension is ``super-additive''.

\begin{lem}\label{lem:sub}
  Let $Y\subseteq  X_1\tdt X_s$ be definable and $u \subseteq [s]$. Assume that $0 \leq n \leq [s]$ is such that for  every  $a\in \pi_u(Y)$ we have
  $\dimp(Y_a) \geq n$. Then $\dimp(Y)\geq \dimp(\pi_u(Y)) +n$.

\end{lem}
\begin{proof}
%
%
%
%
%
Assume that $Y$ is definable over a small set of parameters $A$, and let $m := \dim_{\mfp}(\pi_u(Y))$. Then there is some $u^* \subseteq u, |u^*| = m$ such that 
$$\pi_{u^*}(Y)(\pi_u(Y)) = \pi_{u^*}(Y) \in \mfp_{u^*} = \bigotimes_{i \in u^*} \mfp_i.$$
Let $b_{u^*} = (b_{i} : i \in u^*) \models \mfp_{u^*}|_{A}$. As $b_{u^*} \in \pi_{u^*}(Y)(\pi_u(Y))$, there exist some $(b_i : i \in u \setminus u^*)$ so that $b_u := (b_i : i \in u) \in \pi_u(Y)$. Then by assumption $\dimp(Y_{b_{u}}) \geq n$, that is for some $v^* \subseteq v := [s] \setminus u$ with $|v^*| \geq  n$ we have $\pi_{v^*}(Y_{b_{u}}) \in  \mfp_{v^*} := \bigotimes_{i \in v^*} \mfp_i$. Let $b_{v^*} = (b_i : i \in v^*) \models \mfp_{v^*}|_{A b_u}$, and let $w := u^* \sqcup v^*$.
Since the types $(\mfp_{i} : i \in w)$ are stationary and commuting, it follows that $b_{w} := (b_i : i \in w) \models \mathfrak{p}_{w}|_{A}$ for $p_w := \bigotimes_{i \in u^* \sqcup v^*} \mfp_i$. As $b_{v^*} \in \pi_{v^*}(Y_{b_{u}})$, there exists some $(b_i : i \in v \setminus v^*)$ so that $(b_i : i \in v) \in Y_{b_u}$, hence $(b_i : i \in [s]) \in Y$. Thus $b_w \in \pi_{w}(Y)$, hence $\pi_{w}(Y) \in \mathfrak{p}_{w}$, and $|w| \geq m + n$ --- which shows that $\dim_{\mfp}(Y) \geq m + n$, as required.
\end{proof}
	
\subsection{Fiber-algebraic relations and $\mfp$-irreducibility}\label{sec: pirreducible}

\begin{defn}
	Given a definable set $Y \subseteq \prod_{i \in [s]}X_i$  and a small set of parameters $C\subseteq \CM$ so that $Y$ is defined over $C$, we
say that $Y$ is \emph{$\mfp$-irreducible over $C$} if there do not exist disjoint sets $Y_1,Y_2$ definable over $C$ with $Y = Y_1 \cup Y_2$ and $\dimp(Y_1) = \dimp(Y_2) = \dimp(Y)$.

We say that $Y$ is \emph{absolutely $\mfp$-irreducible} if it is irreducible
over any small set $C \subseteq \CM$ such that $Y$ is defined over $C$.
\end{defn}

\begin{rem}\label{rem: irred iff unique gen type} It follows from the definition of $\mfp$-dimension that a definable set $Y \subseteq  X_1 \times \ldots \times X_s$ is $\mfp$-irreducible over $C$ if and only if any two
  tuples generic in $Y$  over $C$ have the same type over $C$.
\end{rem}

\begin{lem}\label{lem: fin many types of full dim in Q}
	If $Q(\bar{x}) \subseteq X_1\times \dotsc \times X_s$ is fiber-algebraic of degree $\leq d$, then the set
	\begin{gather*}
		 \left\{q \in S_{\bar{x}}(\CM) :  Q \in q \textrm{ and } \dimp(q) \geq s-1 \right\}
	\end{gather*}
	has cardinality at most $sd$.
\end{lem}
\begin{proof}
	Assume towards a contradiction that $q_1, \ldots, q_{sd+1}$ are pairwise different types in this set. Then there exist some formulas $\psi_i(\bar{x})$ with parameters in $\CM$ such that $\psi_i(\bar{x}) \in q_i$ and $\psi_i(\bar{x}) \rightarrow \neg \psi_j(\bar{x})$ for all $i \neq j \in [sd+1]$. Let $C \subseteq \CM$ be the (finite) set of the parameters of $Q$ and $\psi_i, i \in [sd+1]$. For each $i \in [sd+1]$, as $\left( \psi_i(\bar{x}) \land Q(\bar{x}) \right) \in q_i$, we have $\dimp \left(\psi_i(\bar{x})  \land Q (\bar{x})\right) \geq s-1$, which by definition of $\mfp$-dimension implies $\exists x_k \left( \psi_i(\bar{x}) \land Q(\bar{x}) \right) \in \bigotimes_{\ell \in [s] \setminus \{k\}} \mfp_{\ell}$ for at least one $k \in [s]$. By pigeonhole, there must exist some $k' \in [s]$ and some $u \subseteq [sd+1]$ such that $|u| \geq d+1$ and $\exists x_{k'} \left( \psi_i(\bar{x}) \land Q(\bar{x}) \right) \in \bigotimes_{\ell \in [s] \setminus \{k' \}} \mfp_{\ell}$ for all $i \in u$. Now let $\bar{a} = (a_{\ell} : \ell \in [s] \setminus \{k'\})$ be a tuple in $\CM$ satisfying $\bar{a} \models \left( \bigotimes_{\ell \in [s] \setminus \{k'\}} \mfp_{\ell} \right)|_{C}$. By the choice of $u$, for each $i \in u$ there exists some $b_i$ in $\CM$ such that $\left( \psi_i \land Q \right) (a_1, \ldots, a_{k'-1}, b_i, a_{k'+1}, \ldots, a_{s})$ holds. By the choice of the formulas $\psi_i$, the elements $(b_i : i \in u)$ are pairwise distinct, and $|u| > d$ --- contradicting that $Q$ is fiber-algebraic of degree $d$.
\end{proof}

\begin{cor}\label{cor: decomp into abs irred sets}
	Every fiber-algebraic $Q \subseteq X_1\times \dotsc \times X_s$ of degree $\leq d$ is a union of at most $sd$ absolutely $\mfp$-irreducible sets (which are  then automatically fiber-algebraic, of degree $\leq d$).
\end{cor}
\begin{proof}
Let $(q_i : i \in [D])$ be an arbitrary enumeration of the set
$$ \left\{q \in S_{\bar{x}}(\CM) :  Q \in q \land  \dimp(q) \geq s-1 \right\},$$ we have $D \leq sd$
 by Lemma \ref{lem: fin many types of full dim in Q}. We can choose formulas $(\psi_i(\bar{x}) : i \in [D])$  with parameters over $\CM$ such that $\psi_i(\bar{x}) \in q_i$ and $\psi_i(\bar{x}) \rightarrow \neg \psi_j(\bar{x})$ for all $i \neq j \in [D]$. Let $Q_i(\bar{x}) := Q(\bar{x}) \land \psi_i(\bar{x})$, then $Q = \bigsqcup_{i \in [D]} Q_i$ and  each $Q_i$ is absolutely $\mfp$-irreducible (by Remark \ref{rem: irred iff unique gen type}, as every generic tuple in $Q_i$ over a small set $C$ has the type $q_i|_C$).
 \end{proof}

 \begin{lem}\label{lem: irred implies coordwise irred}
 	If $Q \subseteq \prod_{i \in [s]}X_i$ is $\mfp$-irreducible over a small set of parameters $C$ and $\dimp(Q) = s-1$, then for any $i \in [s]$ and any tuple $\bar{a} = (a_j : j \in [s] \setminus \{i\})$ which is $\mfp$-generic in $\prod_{j \in [s] \setminus \{i\} } X_j$ over $C$ (i.e.~$\bar{a} \models (\bigotimes_{j\in [s] \setminus \{i\} } \mfp_j)|_C$), if $Q(a_1,\dotsc,a_{i-1}, x_i,a_{i+1},\dotsc,a_s)$ is consistent then it implies a complete type over $C\cup\left\{a_j : j \in [s] \setminus \{i\}\right\}$.
 	\end{lem}
\begin{proof}
Otherwise there exist two types $r_t \in S_{x_i}(C \bar{a}), t \in \{1,2\}$ such that $r_1 \neq r_2$ and $Q(a_1, \ldots, a_{i-1}, x_i, a_{i+1}, \ldots, a_s) \in r_t$ for both $t \in \{1,2\}$. Then there exist some formulas $\varphi_t(\bar{x}), t \in \{1,2\}$ with parameters in $C$ such that $\varphi_t(a_1, \ldots, a_{i-1}, x_i, a_{i+1}, \ldots, a_s) \in r_t$, $\varphi_1(\bar{x}) \rightarrow \neg \varphi_{2}(\bar{x})$ and $\varphi_2(\bar{x}) \rightarrow \neg \varphi_{1}(\bar{x})$. In particular, by assumption on $\bar{a}$,
$$\dimp \left( Q(\bar{x}) \land \varphi_{t}(\bar{x}) \right) \geq s-1$$
 for both $t \in \{1,2\}$ --- contradicting irreducibility of $Q$ over $C$.\end{proof}

\subsection{On general position}
\label{sec:general-position-1}

We recall the notion of general position from Definition \ref{defn: gen pos intro}, specialized to the case of $\mfp$-dimension.

\begin{defn}
	Let  $(X,\mfp)$ be a $\mfp$-pair, and  let $\CF$ be a
definable family of subsets of $X$.
For $\nu\in \NN$, we say that a set $A\subseteq X$ is in \emph{$(\CF,\nu)$-general position}
if for every $F\in \CF$    with $\dimp(F) = 0$  we have $|A\cap F|
\leq \nu$.
\end{defn}
We extend this notion to  cartesian products of $\mfp$-pairs.
\begin{defn}
For sets $X_1\ttimes X_2 \ttimes \dotsb \ttimes X_s$ and an integer
$n\in  \NN$,  by an \emph{$n$-grid on $X_1\ttimes \dotsb \ttimes X_s$}
we mean a set of the form $A_1\ttimes A_2\ttimes\dotsb \ttimes A_s$
with $A_i \subseteq X_i$ and $|A_i| \leq n$ for all $i\in [s]$.	
\end{defn}

\begin{defn}\label{def: p-gen pos}
	Let $s\in \NN$ and $(X_i,\mfp_i)$,  $i\in [s]$,  be $\mfp$-pairs.
Let $\vec \CF$ be a definable system of subsets of  $(X_i)$,
$i\in [s]$, i.e.~$\vec\CF=(\CF_1,\dotsc,\CF_s)$  where each
 $\CF_i$ is a definable family of subsets of $X_i$.
 For  $\nu\in \NN$, we say
that a grid $A_1 \times \dotsb \ttimes A_s$ on $X_1\ttimes\dotsb\ttimes X_s$
is in
\emph{$(\vec\CF,\nu)$-general position} if each $A_i$ is in
  $(\CF_i,\nu)$-general position.
\end{defn}

We will need a couple of auxiliary lemmas bounding the size of the intersection of sets in a definable family with finite grids in terms of their $\mfp$-dimension.

\begin{lem}\label{lem:general-position}
  Let $s \in \mathbb{N}_{\geq 1}$,
$(X_i,\mfp_i)_{i\in [s]}$ a $\mfp$-system,  and $\CG$ a definable family of subsets of
$X_1\tdt X_s$ such that $\dimp(G)=0$ for every $G\in \CG$. Then  there is a definable system of
subsets    $\vec \CF=(\CF_1,\dotsc, \CF_s)$ such that: for any finite
grid $A=A_1\tdt A_s$ on $X_1\times \dotsb \times X_s$  in $(\vec
\CF,\nu)$-general position and any $G\in \CG$
 we have    $|G\cap  A| \leq \nu^s$.
\end{lem}

\begin{proof}
%
%
%
%
%
%
%
%
%

Assume that $\CG$  is a definable family of subsets
 $X_1\tdt X_{s}$ with $\dimp(G)=0$ for all $G\in \CG$. For $i \in [s]$ and $G \in \mathcal{G}$, we let $G_i := \pi_i(G)$, note that still $\dimp(G_i) = 0$. Let  $\mathcal{F}_i := \left\{G_i : G \in \mathcal{G} \right\}$, we claim that then $\vec{\mathcal{F}} := \left(\mathcal{F}_1, \ldots, \mathcal{F}_s \right)$ satisfies the requirements.
 
Indeed, let $A=A_1\tdt A_{s}$ be a finite grid on $X_1\times \dotsb \times X_{s}$  in $(\vec
\CF,\nu)$-general position.
Let $G\in \CG$ be arbitrary. As $G_i \in \mathcal{F}_i$ with $\dimp(G_i) = 0$, by assumption we have $|G_i \cap A_i|\leq \nu$ for every $i \in [s]$. As $G \subseteq \prod_{i \in [s]} G_i$, we have 
$$G \cap  \left( \prod_{i \in [s]}A_i \right) \subseteq \left( \prod_{i \in [s]} G_i \right) \cap \left(  \prod_{i \in [s]} A_i \right) =  \prod_{i \in [s]} (G_i \cap A_i),$$
hence $\left \lvert G \cap  \prod_{i \in [s]}A_i \right \rvert \leq \nu^s$, as required.
\end{proof}

\begin{lem}\label{lem:general-position-2}
 Let $s \in \mathbb{N}_{\geq 1}$ and $(X_i,\mfp_i)_{i\in [s]}$  be a $\mfp$-system,  and $\CG$ a definable family of subsets of
$X_1\tdt X_s$.   Assume that for some $0 \leq k\leq s$ we have
$\dimp(G)\leq k$ for every $G\in \CG$.  Then  there is a definable system  $\vec \CF=(\CF_1,\dotsc, \CF_s)$ of subsets of $X_1 \times \ldots \times X_s$ such that: for any $\nu$ and any 
$n$-grid $A=A_1\tdt A_s$ on $X_1\times \dotsb \times X_s$  in $(\vec
\CF,\nu)$-general position,  for every  $G\in \CG$
 we have    $|G\cap  A| \leq s^k \nu^{s-k} n^k $.
\end{lem}
\begin{proof}

 Given $s \geq k$ and $\nu$, we let $C(k,s,\nu)$ be the smallest number in $\mathbb{N}$ (if it exists) so that the bound $|G\cap  A| \leq C(k,s, \nu) n^k $ holds (with respect to all possible $\mfp$-systems $(X_i,\mfp_i)_{i\in [s]}$ and definable families $\CG$). We will show that  $C(k,s,\nu) \leq s^k \nu^{s-k}$ for all $s \geq k \geq 0$ and $\nu$.

  For any $s \in \mathbb{N}_{\geq 1}$ and $k=0$, the claim holds by Lemma \ref{lem:general-position} with $C(0,s,\nu) = \nu^{s}$.
  For any $s \in \mathbb{N}_{\geq 1}$ and $k=s$, the claim trivially holds with $C(s,s, \nu) = 1$ (and $\CF_i = \emptyset, i \in [s]$).

  We fix $s > k \geq 1$ and assume that the claim holds for all pairs $s' \geq k' \geq 0$ with either $s' < s$ or $k' < k$.
 Assume that $\dim_{\mfp}(G) \leq k$ for every $G \in \CG$.
 Given $G \in \CG$, let $G' := \{g \in \pi_1(G) : \dim_{\mfp}(G_g) \geq k \}$. Then $\CF_1 := \{G' : G \in \CG \}$ is a definable family of subsets of $X_1$ by Claim \ref{claim:p-dim-def}. By assumption and Lemma \ref{lem:sub} we have $\dim_{\mfp}(G') = 0$ for every $G \in \CG$.
  Let
 \begin{gather*}
 	 \CG^* := \{ G_g : G \in \CG \land g \in \pi_1(G) \},\\
 	 \CG^*_{< k} := \{ G_g : G \in \CG \land g \in \pi_1(G) \land \dim_{\mfp}(G_g) < k \}.
 \end{gather*}
\noindent Both $\CG^*$ and $\CG^*_{<k}$ (by Claim \ref{claim:p-dim-def}) are definable families of subsets of $\prod_{2 \leq i \leq s} X_i$, all sets in $\CG^*$ have $\mfp$-dimension $\leq k$, and all sets in $\CG^*_{<k}$ have $\mfp$-dimension $\leq k-1$. Applying the $(k,s-1)$-induction hypothesis,
let $\vec \CF^* = \left(\CF^*_i : 2 \leq i \leq s \right)$ be a definable system of subsets of $X_2 \times \ldots \times X_s$ satisfying the conclusion of the lemma with respect to $\CG^*$. Applying the $(k-1, s-1)$-induction hypothesis, let $\vec \CF^{*}_{<k} = \left(\CF^*_{<k,i} : 2 \leq i \leq s \right)$ be a definable system of subsets of $X_2 \times \ldots \times X_s$ satisfying the conclusion of the lemma with respect to $\CG^*_{<k}$.
 We let $\vec{\CF} = (\CF_i : i \in [s])$ be a definable system of subsets of $X_1 \times \ldots \times X_{s}$, with $\CF_1$ defined above and $\CF_i := \CF^{*}_i \cup \CF^{*}_{<k,i}$ for $2 \leq i \leq s$.

Let now $\nu \in \mathbb{N}$ and $A=A_1\tdt A_{s}$ be a finite grid on $X_1\times \dotsb \times X_{s}$  in $(\vec
\CF,\nu)$-general position.
Let $G\in \CG$ be arbitrary.
As $G' \in \CF_0$, we have in particular that $|G' \cap A_1| \leq \nu$, and by the choice of $\vec{\CF}^*$, for every $g \in G' \cap A_1$ we have $|G_g \cap (A_2 \times \ldots \times A_s)| \leq C(k, s-1, \nu) n^{k}$.
And by the choice of $\vec{\CF}^*_{<k}$, for every $g \in A_1 \setminus G'$, we have $|G_g \cap (A_2 \times \ldots \times A_s)| \leq C(k-1, s-1, \nu) n^{k-1}$.
Combining, we get
\begin{gather*}
	|G \cap (A_1 \times \ldots \times A_s)| \leq \\
	 \nu C(k, s-1, \nu) n^{k} + (n-\nu)C(k-1, s-1, \nu) n^{k-1}
	 \leq \\
	 \big( \nu C(k, s-1, \nu) + C(k-1, s-1, \nu) \big) n^{k}.
	\end{gather*}
 This establishes a recursive bound on $C(k,s,\nu)$. Given $s \geq k \geq 1$, we can repeatedly apply this recurrence for $s, s-1, \ldots, k$, and using that $C(s,s,\nu) = 1$ for all $s,\nu$ we get that
 $$C(k,s,\nu) \leq \nu^{s-k} + \sum_{i=1}^{s-k} \nu^{i-1} C(k-1, s-i, \nu)$$
  for any $s \geq k \geq 1$. Using that $C(0,s,\nu) = \nu^s$ for all $s,\nu$ and iterating this inequality for $0,1, \ldots, k$, it is not hard to see that $C(s,k,\nu) \leq  s^k \nu^{s-k} $ for all $s,k,\nu$.
\end{proof}

\subsection{Main theorem:  the statement and some reductions}
\label{sec:setting-3}

From now on we will assume additionally that the theory $T$ is stable and
eliminates imaginaries, i.e.~$T = T^{\eq}$ (we refer to e.g.~\cite{tent2012course} for a general exposition of stability).
As before, $\CM$ is an $|\CL|^+$-saturated  model of $T$, $\mathbb{M}$ is a monster model of $T$, and we assume that $(X_i, \mfp_i)_{i \in [s]}$ is a $\mfp$-system in $\CM$, with each $\mfp_i$ non-algebraic. ``Definable'' means ``definable with parameters in $\mathcal{M}$''. As usual, if $X$ is a definable set, a family $\mathcal{F}$ of subsets of $X$ is definable if there exist definable sets $Y, F \subseteq X \times Y$ so that $\mathcal{F} = \left\{ F_b : b \in Y \right\}$.

\begin{rem}
Note that if $Q\subseteq X_1\tdt X_s$ is a
fiber-algebraic  relation of degree $d$, then for any $n$-grid
$A \subseteq \prod_{i \in [s]}X_i$ we have
\[ |Q\cap  A |\leq  d n^{s-1} =O_d(n^{s-1}). \]	
\end{rem}

  \gdef\CQ{\mathcal{Q}}

\begin{defn}\label{def: power saving} Let $\mathcal{Q}$ be a definable family of subsets of $X_1\tdt X_s$.
\begin{enumerate}
  \item  Given a real $\varepsilon > 0$, we say that $\mathcal{Q}$ admits \emph{$\varepsilon$-power saving}
if there exist 
definable families
  $\CF_i$ on $X_i$,   such
  that for  $\vec\CF=(\CF_i) _{i\leq s}$ and  any $\nu \in \NN$, for any
$n$-grid $A=A_1\times \dotsb \times A_s$ on $X_1 \times \dotsb \times X_s$ in
$(\vec\CF,\nu)$-general position and any $Q \in \mathcal{Q}$ we have
\[|Q\cap A | = O_\nu \left(n^{
    (s{-}1)-\varepsilon} \right). \] 

 \item   We say that $\mathcal{Q}$ \emph{admits power saving}\footnote{We are following the terminology in \cite{Bays}.} if it admits $\varepsilon$-power saving for some $\varepsilon > 0$.
 \item     We say that a relation $Q \subseteq X_1 \times \ldots \times X_s$ admits $(\varepsilon$-)power saving if the family $\mathcal{Q} := \{ Q\}$ does.
\item  We say that $Q$ is \emph{special} if it  is fiber-algebraic and does not admit power-saving.
  \end{enumerate}
  \end{defn}
  
\begin{lemma}\label{lem: union power saving}
	Assume $\CQ, \CQ_1, \ldots, \CQ_m$ are definable families of subsets of $X_1\tdt X_s$ and $\varepsilon >0$ is such that each $\CQ_t$ satisfies $\varepsilon$-power saving. Assume that for every $Q \in \CQ$, $Q = \bigcup_{t \in [m]}Q_t$ for some $Q_t \in \CQ_t$. Then $\CQ$ also satisfies $\varepsilon$-power saving.
\end{lemma}
\begin{proof}
	Assume each $\CQ_t, t \in [m]$ satisfies $\varepsilon$-power saving, i.e.~there exist
	definable families
  $\CF_{t,i}$ on $X_i$ and functions $C_t: \mathbb{N} \to \mathbb{N}$ so that letting $\vec\CF_t=(\CF_{t,i})_{i\leq s}$, 
	for every grid $A$ in $\left(\vec\CF_t , \nu \right)$-general position and every $Q_t \in \CQ_t$ we have  $|Q_t\cap A | \leq C_t(\nu) n^{
    (s{-}1)-\varepsilon}$. Let $\CF_{i} := \bigcup_{t \in [m]} \CF_{t,i}$, $\vec\CF=(\CF_{i})_{i\leq s}$ and $C := \sum_{t \in [m]}C_t$. Then for every grid $A$ in $\left(\vec\CF , \nu \right)$-general position and every $Q \in \CQ$ we have  $|Q\cap A | \leq C(\nu) n^{
    (s{-}1)-\varepsilon}$, as required.
    \end{proof}

 We recall Definition \ref{def: gen corresp intro}, specializing to $\mfp$-dimension.

 \begin{defn}\label{def: group corresp}
 	Let $Q \subseteq \prod_{i \in [s]} X_i$ be a definable relation and $(G, \cdot, 1_{G})$ a type-definable group in $\CM$ (over a small set of parameters $A$). We say that $Q$ is in a \emph{$\mfp$-generic correspondence with $G$} (over  $A$) if there exist elements $g_1, \ldots, g_s \in G(\CM)$ such that:
\begin{enumerate}
\item   $g_1\cdot
  \dotsc\cdot  g_s =1_{G}$;
 \item  $g_1,\dotsc,g_{s-1}$ are independent generics in $G$ over $A$ (in the usual sense of stable group theory);
 \item  for each $i \in [s]$ there is a generic element $a_i \in X_i$   realizing $\mfp_i|_{A}$ and inter-algebraic with $g_i$ over $A$,  such that $\CM \models Q(a_1, \ldots, a_s)$.
\end{enumerate}
 \end{defn}

\begin{rem}
	 If $Q$ is $\mfp$-irreducible over $A$, then (3) holds for all $g_1, \ldots, g_s \in G$ satisfying (1) and (2), providing a definable generic finite-to-finite correspondence between $Q$ and the graph of the $(s-1)$-fold multiplication in $G$.
\end{rem}

The following is the main theorem of the section characterizing special fiber-algebraic relations in stable reducts of distal structures.

 \begin{thm}\label{thm: stab main ineff}
 Assume that $\CM$ is an $|\CL|^+$-saturated $\CL$-structure, and $\Th(\CM)$ is stable and admits a distal expansion. Assume that $s \geq 3$, $(X_i, \mfp_i)_{i \in [s]}$ is a $\mfp$-system with each $\mfp_i$ non-algebraic and $Q\subseteq X_1\tdt X_s$ is a definable fiber-algebraic  relation.
  	Then at least one of the following holds.
 	\begin{enumerate}
 		\item $Q$ admits power saving.
 		
 		\item $Q$ is in a $\mfp$-generic correspondence with an abelian group $G$ type-definable in $\CM^{\eq}$ over a set of parameters of cardinality $\leq |\CL|$.
 	\end{enumerate}
 \end{thm}

  The only property of distal structures actually used is that every definable binary relation in $\CM$ satisfies the $\gamma$-ST property (Definition  \ref{def: gamma ST property}) for some $\gamma > 0$, by Proposition \ref{prop: ind ES} and Fact \ref{fac: def in dis impl distal cell decomp}. 
   In fact, Theorem \ref{thm: stab main ineff} follows from the following more precise version with the additional uniformity in families and explicit bounds on power saving.
   \gdef\CQ{\mathcal{Q}}
\begin{defn}
Let $\CQ$ be a definable family of subsets $X_1\tdt X_s$.

\begin{enumerate}
	\item  We say that $\CQ$ is a \emph{fiber-algebraic} family if each
  $Q\in \CQ$ is fiber-algebraic.
\item  We say that $\CQ$ is an \emph{absolutely $\mfp$-irreducible fiber-algebraic family} if each
  $Q\in \CQ$ is $\mfp$-irreducible and fiber-algebraic
\end{enumerate}
 \end{defn}

\begin{rem} Let $\CQ$ be a definable fiber-algebraic family. By saturation of $\CM$ there is $d\in \NN$ such that every $Q\in \CG$  has degree $\leq d$. In this case we say that $\CQ$ is \emph{of degree $\leq d$}. 
  \end{rem}
  \gdef\CQ{\mathcal{Q}}
  \begin{thm}\label{rem: power saving bound, main stable}
  Assume that $\CM$ is an $|\CL|^+$-saturated $\CL$-structure and $\Th(\CM)$ is stable. Assume that $s \geq 4$, $(X_i, \mfp_i)_{i \in [s]}$ is a $\mfp$-system with each $\mfp_i$ non-algebraic, and let
  $\CQ$ be a fiber-algebraic definable family, and fix $0 < \gamma \leq 1$.
  \begin{itemize}
  	\item If $s \geq 4$, assume that there exist $m \in \mathbb{N}$ and definable families $\mathcal{Q}_i, i \in [m]$ of absolutely $\mfp$-irreducible sets so that for every $Q \in \CQ$ we have $Q = \bigcup_{i \in [m]}Q_i$ for some $Q_i \in \CQ_i$.  Assume also that for each $i \in [m], t_1 \neq t_2 \in [s]$, the family $\mathcal{Q}_i$ viewed as a definable family of subsets of $\left( X_{t_1} \times X_{t_2} \right) \times \left( \prod_{k \in [s] \setminus \{t_1, t_2\}} X_k\right)$ satisfies the $\gamma$-ST property.

  	\item If $s = 3$,   for each $i \in [m]$ and $\CQ_i$ as above, we additionally consider the definable family $\CQ^*_i := \left\{ Q^* : Q \in \CQ_i \right\}$ of subsets of $X_1 \times X_2 \times X_3 \times X_4$, where 
  		\begin{gather*}
  		Q^* := \Big\{(x_2,x'_2,x_3,x'_3) \in X_2 \times X_2 \times X_3 \times X_3 : \\
	 \exists x_1 \in X_1 \, \big((x_1,x_2,x_3) \in Q \land (x_1, x'_2, x'_3) \in Q \big) \Big\}.
  	\end{gather*}
  	Assume moreover that there exist $m_i \in \mathbb{N}, i \in [m]$ and definable families $\CQ_{i,j}$ for $i \in [m], j \in [m_i]$ so that for every $i \in [m], Q^* \in \CQ^*_{i}$ we have $Q^* = \bigcup_{j \in [m_i]} Q_{i,j}$ for some $Q_{i,j} \in \CQ_{i,j}$. Assume also that for each $i \in [m], j \in [m_i], t_1 \neq t_2 \in [4]$, the family $\mathcal{Q}_{i,j}$ viewed as a definable family of subsets of $\left( X_{t_1} \times X_{t_2} \right) \times \left( \prod_{k \in [4] \setminus \{t_1, t_2\}} X_k\right)$ satisfies the $2 \gamma$-ST property.

  \end{itemize}
  	Then there is a definable subfamily $\CQ'\subseteq \CQ$  such that the family 
$\CQ'$ admits $\gamma$-power saving, and for each $Q \in \CQ\setminus \CQ'$  the relation  $Q$ is in a $\mfp$-generic correspondence with an abelian group $G_Q$ type-definable in $\CM^{\eq}$ over a set of parameters of cardinality $\leq |\CL|$.
 \end{thm}

 To see that Theorem \ref{thm: stab main ineff} follows from Theorem \ref{rem: power saving bound, main stable}, assume that a definable relation $Q$ is as in Theorem \ref{thm: stab main ineff}, and consider the definable family $\CQ := \left\{ Q \right\}$ consisting of a single element $Q$. 
  By Proposition \ref{prop: ind ES} and Fact \ref{fac: def in dis impl distal cell decomp} every definable family of binary relations in $\CM$ satisfies the $\gamma$-ST property (Definition  \ref{def: gamma ST property}) for some $\gamma > 0$. Moreover, by Corollary \ref{cor: decomp into abs irred sets}, if $Q \subseteq X_1 \times \ldots \times X_s$ is definable and fiber-algebraic of degree $\leq d$, we have $Q = \bigcup_{i \in [sd]} Q_i$ for some definable absolutely $\mfp$-irreducible sets $Q_i$. By distality, each $Q_i$ satisfies the $\gamma_i$-ST-property for some $\gamma_i > 0$.  Hence, taking $\CQ_i := \{ Q_i\}$, $m := sd$ and $\gamma := \min \{\gamma_i : i \in [m]\} > 0$, the assumption of Theorem \ref{rem: power saving bound, main stable} is satisfied for $s \geq 4$. If $s = 3$, note that each $Q_i$ is still fiber-algebraic of degree $d$, hence each $Q'_i \subseteq X_1 \times \ldots \times X_4$ is fiber-algebraic, of degree $\leq d^2$ by Lemma \ref{lem: Q' is fiber alg}. By Corollary \ref{cor: decomp into abs irred sets} again, for each $i$ we have $Q'_i = \bigcup_{j \in [4d^2]} Q_{i,j}$ for some definable absolutely $\mfp$-irreducible sets $Q_{i,j}$, each satisfying the $\gamma_{i,j}$-ST-property for some $\gamma_{i,j} > 0$. Hence, taking $m_i := 4d^2$, $\CQ_{i,j} := \left\{ Q_{i,j}\right\}$ and $\gamma := \min \{\gamma_{i,j} : i \in [m], j \in [m_i]\} > 0$, the assumption of Theorem \ref{rem: power saving bound, main stable} is satisfied for $s =3$. In either case, let $\CQ'$ be as given by applying Theorem \ref{rem: power saving bound, main stable}. If $\CQ' = \CQ$, then $Q$ is in Case (1) of Theorem \ref{thm: stab main ineff}. Otherwise $\CQ' = \emptyset$, and $Q$ is in Case (2) of Theorem \ref{thm: stab main ineff}.
%
%

In the rest of the section we give a proof of Theorem \ref{rem: power saving bound, main stable} (which will also establish Theorem \ref{thm: stab main ineff}).
In fact, first we will prove a special case of Theorem \ref{rem: power saving bound, main stable} for definable families of absolutely $\mfp$-irreducible sets and $s \geq 4$ (Theorem \ref{thm: stab main ser}), and then derive full Theorem \ref{rem: power saving bound, main stable} from it in Section \ref{sec: MainThm for s geq 4} (for $s \geq 4$) and Section \ref{sec: stable main thm ternary} (for $s = 3$). We begin with some auxiliary observations and reductions.

%

\begin{ass}\label{ass: stab ass 1}
For the rest of Section \ref{sec: main thm stable},   we assume that $s \in \mathbb{N}_{\geq 3}$ (even though some of the results below make sense for $s \in \mathbb{N}_{\geq 1}$),  $\CM$ is $|\CL|^+$-saturated, $(X_i, \mfp_i)_{i \in [s]}$ is a $\mfp$-system with each $\mfp_i$ non-algebraic, and $X_i$ is a $\emptyset$-definable. ``Definable'' will mean ``definable with parameters in $\mathcal{M}$''
\end{ass}

\begin{lem}\label{lem: fib alg impl dim s-1}
If $Q\subseteq X_1\tdt X_s$ is fiber-algebraic then $\dim_{\mfp}(Q) \leq s-1$.

\end{lem}
\begin{proof}
	Let $(a_1, \ldots, a_{s-1}) \models \bigotimes_{i \in [s-1]} \mfp_i|_{A}$, where $A$ is some finite set such that $Q$ is $A$-definable. The type $\mfp_s$ is non-algebraic by Assumption \ref{ass: stab ass 1}, and $Q(a_1, \ldots, a_{s-1},x_s)$ has at most $d$ solutions.  Hence necessarily 
	$$Q(a_1, \ldots, a_{s-1},x_s) \notin \mfp_s,$$
	 so $Q(x_1, \ldots,x_s) \notin  \bigotimes_{i \in [s]} \mfp_i$.
\end{proof}

  The following is straightforward by definition of fiber-algebraicity.
  \begin{lem}\label{lem: bdd fib alg grid using proj}  Let $Q\subseteq X_1\tdt X_s$ be a
fiber-algebraic  relation of degree $\leq d$ and $u \subseteq [s]$ with $|u| = s-1$.
    Let $\pi_{u}$ be the projection from $X_1\tdt X_s$ onto $\prod_{i \in u} X_i$.   Let $A=A_1\tdt A_s$ be a grid on
    $X_1\tdt X_s$.   Then
\[ |Q\cap A| \leq d \left\lvert \pi_{u}(Q) \cap \prod_{i \in u} A_i \right\rvert. \]
  \end{lem}


\begin{prop}\label{prop:1-save-ser}
  Let $\CQ$ be a definable family of fiber-algebraic subsets  of $X_1\tdt X_s$. Let $u\subseteq [s]$ with $u=s-1$.
  Assume that for every $Q\in \CQ$ the projection $\pi_u(Q)$ onto $\prod_{i\in u} X_i$ is not $\mfp$-generic.
  Then $\CQ$ admits $1$-power saving. 
\end{prop}

\begin{proof}
By Lemma \ref{lem:general-position-2}  there exists a definable system  $\vec \CF_u^*=(\CF_i : i \in u)$ of subsets of $\prod_{i \in u} X_i$ such that for any $\nu \in  \mathbb{N}$, for any
$n$-grid $A^*$ on $\prod_{i \in u}X_i$  in $(\vec
\CF_u^*,\nu)$-general position, for any $Q\in \CQ$
we have    $|\pi_u(Q) \cap  A^*| \leq s^{s-2} \nu^{2} n^{s-2}$.
Let $d\in \NN$ be such that $\CQ$ is of degree $\leq d$. 
Taking $\CF_i := \emptyset$ for $i \in [s] \setminus u$, let $\vec \CF_u := \{\CF_i : i \in [s]\}$. Then by Lemma \ref{lem: bdd fib alg grid using proj}, for any
$n$-grid $A$ on $\prod_{i \in [s]}X_i$  in $(\vec
\CF,\nu)$-general position, for any $Q\in \CQ$
 we have    $|Q\cap  A| \leq d s^{s-2} \nu^{2} n^{s-2} = O_{\nu}(n^{s-2})$, hence the family  $\CQ$ admits $1$-power saving.
\end{proof}

The following is the main theorem for definable families of absolutely irreducible sets:
\begin{thm}\label{thm: stab main ser}
  Assume that $\CM$ is an $|\CL|^+$-saturated $\CL$-structure and $\Th(\CM)$ is stable. Assume that $s \geq 4$, $(X_i, \mfp_i)_{i \in [s]}$ is a $\mfp$-system with each $\mfp_i$ non-algebraic, and let
  $\CQ$ be a fiber-algebraic definable family of absolutely $\mfp$-irreducible subsets of $X_1\tdt X_s$. Assume that  for some $0 < \gamma \leq 1$, for each $i \neq j \in [s]$, $\mathcal{Q}$ viewed as a definable family of subsets of $\left( X_i \times X_j \right) \times \left( \prod_{k \in [s] \setminus \{i,j\}} X_k\right)$ satisfies the $\gamma$-ST property.
  	Then there is a definable subfamily $\CQ'\subseteq \CQ$  such that the family 
$\CQ'$ admits $\gamma$-power saving, and for each $Q \in \CQ\setminus \CQ'$  the relation  $Q$ is in a $\mfp$-generic correspondence with an abelian group $G_Q$ type-definable in $\CM^{\eq}$ over a set of parameters of cardinality $\leq |\CL|$.
 \end{thm}

 In the rest of this section we give a proof of Theorem~\ref{thm: stab main ser} (and then of Theorem \ref{rem: power saving bound, main stable}).

We fix  a fiber-algebraic definable family $\CQ$  of absolutely $\mfp$-irreducible subsets of $X_1\tdt X_s$.

 Let $\CQ_0$ be the set of all $Q\in \CQ$ such that for some $u\subseteq [s]$ with $|u|=s-1$ for the projection $\pi_u(Q)$   of $Q$ onto $\prod_{i\in u} X_i$  we have $\dimp(\pi_u(Q) < s-1$.  
 
 By Claim~\ref{claim:p-dim-def}, the family $\CQ_0$ is definable and it follows from Proposition~\ref{prop:1-save-ser}
 that the family $\CQ_0$ admits $1$-power saving.
 Hence replacing $\CQ$ with $\CQ\setminus \CQ_0$, if needed, we will assume the following:
\begin{ass}\label{ass: two}
 $\CQ$ is a fiber-algebraic definable family  of absolutely $\mfp$-irreducible subsets of $X_1\tdt X_s$. For any $Q\in \CQ$ the projection of $Q$ onto any $s-1$
coordinates is $\mfp$-generic.  In particular, $\dimp(Q)=s-1$ (by Lemma \ref{lem: fib alg impl dim s-1}).
\end{ass}

\begin{prop}\label{prop:q-gen-ser}
Let $C$ be a small set in $\CM$, $Q\in \CQ$  and
  let  $\bar a =(a_1,\dotsc,a_s)$ be a $\mfp$-generic in $Q$
  over $C$ (see Remark \ref{rem: p-gen tuple def} for the definition).
  Then for any $i \in [s]$
  we have
  \[ \left(a_j : j \in [s] \setminus \{ i\} \right)\models \bigotimes_{j \in [s] \setminus \{i\} } \mfp_j |_C.   \]
\end{prop}
\begin{proof}
  Since $Q$ is absolutely $\mfp$-irreducible,  it has unique $\mfp$-generic type over $C$. By our assumption for any $i\in [s]$
  the projection of $Q$ onto $[s]\setminus \{i\}$ is $\mfp$-generic.
  Hence any realization of $\bigotimes_{j \in [s] \setminus \{i\} } \mfp_j |_C$ can be extended to a $\mfp$-generic of
  $Q$.
\end{proof}

Next we observe that the assumption that the projection of $Q$ onto any $s-1$
coordinates is $\mfp$-generic in Proposition \ref{prop:q-gen-ser} was necessary, but could be replaced by the assumption that $Q$ does not admit $1$-power saving (this will not be used in the proof of the main theorem).

\begin{prop}\label{prop:q-gen}
Assume that $Q$ is absolutely $\mfp$-irreducible, $\dimp(Q) = s-1$ (but no assumption on the projections of $Q$), and $Q$ does not admit $1$-power saving. Let $C$ be a small set in $\CM$ and
  let  $\bar a =(a_1,\dotsc,a_s)$ be a generic in $Q$
  over $C$.
  Then for any $i \in [s]$
  we have
  \[ (a_j : j \in [s] \setminus \{ i\} )\models \bigotimes_{j \in [s] \setminus \{i\} } \mfp_j |_C.   \]
\end{prop}
\begin{proof} Let $\bar{a}$ be a generic in $Q$ over $C$. Permuting the variables if necessary and using that the types $\mfp_i$ commute, we may assume
  $$(a_1,\dotsc,a_{s-1})\models   \mfp_1\otimes
  \dotsb\otimes  \mfp_{s-1} |_{C}.$$

  We only consider the case $i=1$, i.e.~we need to show that
  $$(a_2, \dotsc,a_s)\models  \mfp_2\otimes  \dotsb\otimes  \mfp_s |_C,$$
 the other cases are analogous.

Assume this does not hold, then there is a relation $G_1\subseteq X_2\tdt X_s$
definable over $C$ such that $\dimp(G_1)< s-1$ and
$(a_2,\dotsc,a_s)\in G_1$.

Since $Q$ is $\mfp$-irreducible over $C$,  the formula $Q(a_1,\dotsc,a_{s-1},x_s)$ implies a complete type over $C \cup \{a_1, \ldots, a_{s-1} \}$ by Lemma \ref{lem: irred implies coordwise irred}.
Hence we have
\[
  Q(a_1,\dotsc,a_{s-1},x_s) \vdash \tp(a_s/C \cup\{a_1,\dotsc,a_{s-1}\}),
\]
so in particular
\begin{gather*}
   Q(a_1,\dotsc,a_{s-1},x_s) \rarr  G_1(a_2,\dotsc,a_{s-1},x_s),
\end{gather*}
which implies
\begin{gather*}
   \{Q(x_1,\dotsc,x_{s-1},x_s)\} \cup  \left(\mfp_1 \otimes \ldots \otimes \mfp_{s-1}\right) |_{C}(x_1, \ldots, x_{s-1}) \\
     \rarr  G_1(x_2,\dotsc,x_{s-1},x_s) .
\end{gather*}
Then, by saturation of $\CM$, there exists some $\mfp$-generic  set $G_2\subseteq
X_1\tdt X_{s-1}$ definable over $C$ (given by a finite conjunction of formulas from $\left(\mfp_1 \otimes \ldots \otimes \mfp_{s-1}\right) |_{C}$)
such that
\[
Q(x_1,\dotsc,x_{s-1},x_s) \wedge G_2(x_1,\dotsc,x_{s-1}) \rarr  G_1(x_2,\dotsc,x_{s-1},x_s),
\]
hence
\[
Q(x_1,\dotsc,x_{s-1},x_s) \rarr  \big(\neg G_2(x_1,\dotsc,x_{s-1}) \vee  G_1(x_2,\dotsc,x_{s-1},x_s)  \big).
\]
Let $H_2 := (\neg G_2 )\times X_s$ and $H_1 := X_1 \times  G_1$. Then $\dim_{\mfp} \left( \pi_{[s-1]}(H_2) \right) = \dim_{\mfp}(\neg G_2) < s-1$ and $\dim_{\mfp} \left( \pi_{[s] \setminus \{1\}}(H_1) \right) = \dim_{\mfp}(\neg G_1) < s-1$.
Thus $Q$ is covered by the union of $H_1$ and $H_2$, each
with $1$-power saving  by Proposition~\ref{prop:1-save-ser}, which implies that $Q$ admits $1$-power-saving.
\end{proof}

\begin{rem}
The assumption that $Q$ has no $1$-power saving is necessary in Proposition \ref{prop:q-gen}, and the assumption that the projection of $Q$ onto any $s-1$
coordinates is $\mfp$-generic in necessary in Proposition \ref{prop:q-gen-ser}. For example let $s=2$  and assume $Q(x_1,x_2)$ is the graph of a bijection from $X_1$ to some  $\emptyset$-definable set $Y_2 \subseteq X_2$  with $Y_2 \notin \mfp_2$. Then $Q$ is clearly fiber algebraic, absolutely $\mfp$-irreducible, with $\dim_{\mathfrak{p}}(Q)=1$. But for a generic $(b_1, b_2) \in Q$, $b_2$ does not realize $\mfp_2|_{\emptyset}$. Note that $Q$ has $1$-power saving. Indeed, let $\vec{\mathcal{F}} := (\mathcal{F}_1, \mathcal{F}_2)$ with $\mathcal{F}_1 := \emptyset, \mathcal{F}_2 := \{ Y_2\}$. Then, given any $n,\nu \in \mathbb{N}$ and an $n$-grid $A_1 \times A_2$ in $\left(\vec{\mathcal{F}},\nu \right)$-general position, as $\dim_{\mathfrak{p}}(Y_2) = 0$ we must have $|A_2 \cap Y_2| \leq \nu$, hence, by definition of $Q$, $|Q \cap (A_1 \times A_2)| \leq \nu = O_{\nu}(1) =  O_{\nu}\left(n^{(2-1)-1} \right)$. Also note that $\pi_{\{2\}}(Q)$ is not $\mathfrak{p}$-generic.
\end{rem}

  We can now state the key structural dichotomy at  the core of Theorem \ref{thm: stab main ser}:
  \begin{thm}\label{thm:ser-main} Let $\CQ=\{Q_\alpha\colon \alpha\in \Omega\}$ be a definable family of absolutely $\mfp$-irreducible fiber-algebraic
    subsets of $\prod_{i\in s} X_i$  satisfying the Assumption \ref{ass: stab ass 1} above.  Assume the family  
$\CQ$, 
as a family of binary relations on $$\left(\prod_{i \in [s-2]} X_i \right) \times \left( X_{s-1} \times X_{s} \right),$$ satisfies the $\gamma$-ST property for some $0< \gamma \leq 1$.

  Then there is a definable $\Omega_1 \subseteq \Omega$ such that the family $\{Q_\alpha\colon \alpha\in \Omega_1 \}$, 
admits $\gamma$-power-saving, and for every $\alpha\in \Omega\setminus \Omega_1$,  
for every tuple $(a_1,\dotsc, a_s)\in Q_\alpha$ generic over $\alpha$ there exists some tuple 
$$\xi\in \acl(a_1,\dotsc,a_{s-2},\alpha) \cap
   \acl(a_{s-1},a_s,\alpha)$$
    of length at most $|\CL|$ such that
   $$(a_1,\dotsc,a_{s-2})\ind_\xi (a_{s-1},a_s).$$
 \end{thm}

%

 \begin{rem}\label{rem: triv for s=3} Theorem \ref{thm:ser-main} is trivial for $s=3$ with $\Omega_1 = \emptyset$, as $a_1 \ind_{\xi} (a_2, a_3)$
   always holds with $\xi :=a_1 \alpha$.
 \end{rem}
 
  First we show how the above theorem, combined with the reconstruction of abelian groups from abelian $s$-gons in Theorem  \ref{thm: main ab mgon gives grp},  implies Theorem~\ref{thm: stab main ser}. Then we use theorem  Theorem~\ref{thm: stab main ser} to deduce Theorem \ref{thm: stab main ineff} for $s \geq 4$ (along  with the bound in Theorem \ref{rem: power saving bound, main stable}) in Section \ref{sec: MainThm for s geq 4}. The case $s=3$ of Theorem \ref{thm: stab main ineff} requires a separate argument reducing to the case $s=4$ of Theorem \ref{thm: stab main ineff}  given in Section \ref{sec: stable main thm ternary}.

\begin{proof}[Proof of Theorem~\ref{thm: stab main ser}] 
From the reductions described above, we assume that $\mathcal{Q}$ and $(X_i,\mfp_i)_{i \in [s]}$ satisfy Assumptions \ref{ass: stab ass 1} and \ref{ass: two}, and that for some $0 < \gamma \leq 1$, for each $i \neq j \in [s]$, $\mathcal{Q}$ viewed as a definable family of subsets of $\left( X_i \times X_j \right) \times \left( \prod_{k \in [s] \setminus \{i,j\}} X_k\right)$ satisfies the $\gamma$-ST property.

It follows that for every permutation  of $[s]$, the  family $\mathcal{Q}$ and the $\mfp$-system  obtained from $\mathcal{Q}$ and $(X_i,\mfp_i)_{i \in [s]}$ by permuting the variables accordingly still satisfy the assumption of Theorem \ref{thm:ser-main}. Applying Theorem~\ref{thm:ser-main}  
to  every permutation  of $[s]$, and taking (definable) $\Omega' \subseteq \Omega$ to be the  union of the corresponding $\Omega_1$'s, 
we have that the family $\CQ'=\{Q_\alpha \colon \alpha\in \Omega'\}$ admits $\gamma$-power saving  and for any $\alpha\in \Omega\setminus \Omega'$,  for every tuple $(a_1, \ldots, a_s)$ generic in $Q_\alpha$
over $\alpha$, after any  permutation of $[s]$ we have 
$$a_1 a_2  \ind_{\acl(a_1 a_2 \alpha)  \cap \acl(a_3 \ldots  a_s \alpha)} a_3  \ldots a_s.$$
Together with fiber-algebraicity of $Q_\alpha$ this implies that $(a_1, \ldots, a_s)$ is an abelian $s$-gon  over $\alpha$.

Applying Theorem  \ref{thm: main ab mgon gives grp},  we obtain that for any $\alpha\in \Omega\setminus \Omega'$ 
there exists a small set $A_\alpha \subseteq \CM$ and a connected abelian group $G_\alpha$ type-definable over $A_\alpha$ and such that $Q_\alpha$ is in a $\mfp$-generic correspondence with $G_\alpha$ over $A_\alpha$.
(As  stated, Theorem \ref{thm: main ab mgon gives grp} only guarantees the existence of an appropriate set of parameters $A_{\alpha}$ of size $\leq |\CL|$  and  $G_{\alpha}$ \emph{in $\MM$}, however by $|\CL|^+$-saturation of $\CM$ there  exists a set $A'_{\alpha}$ in $\CM$ with the same type as $A_{\alpha}$, hence we obtain the required group applying an automorphism of $\MM$ sending $A_{\alpha}$ to $A'_{\alpha}$.)
%
\end{proof}

In the remainder of the section we prove Theorem~\ref{thm:ser-main}.

%

\subsection{Proof of Theorem~\ref{thm:ser-main}}
\label{sec:proof-theor-ser-main}

Theorem \ref{thm:ser-main} is trivial in the case $s=3$ by Remark \ref{rem: triv for s=3}, so we will assume $s\geq 4$.

Let $U := X_1\times \dotsc  \times X_{s-2}$ and $V := X_{s-1}\times X_s$.  We view each
$Q\in \CQ$ as a binary relation $Q\subseteq U\times V$.

We fix a formula $\varphi(u;v;w) \in \CL$ such that for $\alpha\in \Omega$ the formula $\varphi(u;v;\alpha)$
defines $Q_\alpha$, with the variables $u$ corresponding to $U$ and $v$ to $V$.

We also fix $d\in \NN$ such that $\CQ$ is of degree $\leq d$.

\medskip

\begin{defn} For $\alpha\in \Omega$  and   $a\in U$,  let $Z_\alpha(a)$ be the set
\[ Z_\alpha(a):=\left\{ a'\in U  \colon  \dimp \left( \varphi(a;v; \alpha)\cap \varphi(a';v; \alpha) \right) = 1\right\}.\]
\end{defn}

\begin{claim}\label{claim:za-def} The family $\{ Z_\alpha(a)\colon \alpha\in \Omega, a\in U\}$  is a
  definable family of subsets of $U$.
\end{claim}
\begin{proof} By Claim~\ref{claim:p-dim-def}, the set
\begin{gather*}
	D :=  \{(a,a',\alpha) \in U \times U\times \Omega : a' \in Z_\alpha(a) \} \\
	= \left\{ (a,a',\alpha) \in U \times U\times \Omega \colon  \dimp \left( \varphi(a;v;\alpha)\cap \varphi(a';v;\alpha) \right) = 1\right\}
\end{gather*}
is definable, hence the family $\{ Z_\alpha(a)\colon \alpha\in \Omega, a\in U\}$ is definable.
 \end{proof}

  \begin{claim}
 	For any $\alpha\in \Omega$ and $a \in U$, we have that
$Z_\alpha(a)\neq \emptyset$ if and only if $a\in Z_\alpha(a)$, if and only if
$\dimp(\varphi(a;v;\alpha))=1$.
 \end{claim}
 \begin{proof}
Let $\alpha\in \Omega$ and $a\in U$. As $Q_\alpha$ is fiber-algebraic, we also have that the binary relation $\varphi(a;v;\alpha) \subseteq X_{s-1} \times X_s$ is fiber-algebraic, hence
$\dimp(\varphi(a;v;\alpha))\leq 1$ (by Lemma \ref{lem: fib alg impl dim s-1}).	The claim follows as, by definition of $\mfp$-dimension, $\dim_{\mfp}(\varphi(a;v;\alpha) \cap \varphi(a;v;\alpha)) = \dim_{\mfp}(\varphi(a;v;\alpha)) \geq \dim_{\mfp}(\varphi(a;v;\alpha) \cap \varphi(a';v;\alpha))$ for any $a' \in U$.
 \end{proof}

\begin{claim}\label{claim-za-fiberalg-ser}
For every $\alpha\in \Omega$ and $a\in U$ the set $Z_\alpha(a) \subseteq X_1\tdt X_{s-2}$ is fiber-algebraic, of degree $\leq 2d^2$.
\end{claim}
\begin{proof}
  We fix $\alpha\in \Omega$ and $a\in U$. Assume $Z_\alpha(a)\neq \emptyset$. Since $\varphi(a;v;\alpha)$ is fiber-algebraic of degree $\leq d$ (by fiber-algebraicity of $Q_\alpha$),
the set $S$ of types $q \in S_v(\CM)$ with
$\varphi(a;v;\alpha) \in q$  and $\dimp(q)=1$ is finite, of size $\leq 2d$ (by Lemma \ref{lem: fin many types of full dim in Q}); and for any $a' \in U$ we have $a'\in Z_\alpha(a)$ if and
only if $\varphi(a',v;\alpha)$ belongs to one of these types (by definition of $\mfp$-dimension).  Thus
\begin{gather*}
Z_\alpha(a) =\{ a' \in U \colon \varphi(a',v,\alpha) \in q
\text{ for some } q\in
  S\}.
\end{gather*}

   Let $q_1,\dotsc, q_t, t \leq 2d$ list all types in $S$.                                                                                                                                                        We then have
$Z_\alpha(a)=\bigcup_{i \in [t]} d_\varphi(q_i)$, where    $d_\varphi(q_i)=\{ a'\in U
\colon \varphi(a',v;\alpha)\in q_i \}$. It is sufficient to show
that each $d_\varphi(q_i)$ is fiber-algebraic of degree $\leq d$.  Choose a realization
$\beta_i$ of $q_i$ in $\mathbb{M}$.  Obviously $d_\varphi(q_i) \subseteq
\varphi(\MM,\beta_i;\alpha)$.  As $\CM \preceq \mathbb{M}$ and $\alpha \in \CM$, the set
$\varphi(\MM;\beta_i;\alpha) \subseteq \prod_{i \in [s-2]} X_i(\mathbb{M})$ is fiber-algebraic of degree $\leq d$, hence the set $d_\varphi(q_i)$ is
fiber-algebraic of degree $\leq d$ as well.
\end{proof}

By Claim \ref{claim-za-fiberalg-ser} and Lemma \ref{lem: fib alg impl dim s-1}, each  $Z_\alpha(a)$ is not a $\mfp$-generic subset of
$X_1\tdt X_{s-2}$, hence we have that $\dimp(Z_\alpha(a))\leq s-3$ for any $\alpha\in \Omega$ and $a \in U$.

\begin{defn}
Let $Z_\alpha \subseteq U$ be the set
\[ Z_\alpha := \{ a\in U \colon \dimp(Z_\alpha (a)) =s-3 \}. \]
\end{defn}

Note that the family $\{Z_\alpha : \alpha\in \Omega\}$ is definable by Claim \ref{claim:p-dim-def}.

\medskip
Let $\Omega_1 := \{\alpha\in \Omega \colon \dimp(Z_\alpha) <s-2\}$.  By Claim~\ref{claim:p-dim-def} the set $\Omega_1$ is definable.
  We will show that the family $\CQ_1 := \{ Q_\alpha \colon \alpha\in \Omega_1\}$  admits $\gamma$-power saving for the required $\gamma$.

  To show that the family $\{Q_\alpha\colon \alpha\in \Omega_1\}$ admits $\gamma$-power saving, it suffices to show that both families 
  $\{Q_\alpha\cap (Z_\alpha\times V)\colon \alpha\in \Omega_1\}$ and $\{Q_\alpha\cap (\bar{Z}_\alpha \times V)\colon \alpha\in \Omega_1\}$ admit $\gamma$-power saving, where $\bar{Z}_\alpha := U \setminus Z_\alpha$ is the complement of $Z_\alpha$ in $U$.

Since for any $\alpha\in \Omega_1$ the set  $Z_\alpha$ is not a $\mfp$-generic subset of $X_1\tdt X_{s-2}$,  for
the projection $\pi_{[s-1]} \colon X_1\tdt X_s \to X_1\tdt  X_{s-1}$ we have that
$\pi_{[s-1]}(Q_\alpha\cap (Z_\alpha \times V))$ is not a $\mfp$-generic subset of $X_1 \times \ldots \times X_{s-1}$. Hence, by
Proposition~\ref{prop:1-save-ser},   the family $\{ Q_\alpha \cap (Z_\alpha \times V) \colon \alpha\in \Omega_1\}$ admits $1$-power
saving.

\medskip

Next we show that the family  $\{ Q_\alpha \cap (\bar Z_\alpha\times V)\colon \alpha\in \Omega_1\}$ admits $\gamma$-power
saving.
By the definition of $Z_\alpha$,  for any $\alpha\in \Omega_1$ and $a\in \bar{Z}_{\alpha}$ we have
$\dimp(Z_\alpha(a)) \leq s-4$. 
By Lemma~\ref{lem:general-position-2}, there is a definable system of
sets
$\vec\CF_1=(\CF_1, \ldots, \CF_{s-2})$ on $X_1 \times \ldots \times X_{s-2}$ such that for any $n$-grid
$A_1\tdt A_{s-2}$ in $(\vec\CF_1,\nu)$-general position we have
  \[ |Z_\alpha(a) \cap (A_1\tdt A_{s-2})|\leq (s-2)^{s-4} \nu^{2} n^{s-4},   \]
  for any $\alpha\in \Omega_1$ and $a\in \bar Z_\alpha$.

  Applying Lemma~\ref{lem:general-position} to the definable family
  \begin{gather*}
  	\CG :=\big\{\varphi(a_1; v;\alpha )\cap \varphi(a_2; v;\alpha) \colon \alpha\in \Omega_1, a_1,a_2\in U,\\
    \dimp(\varphi(a_1; v;\alpha)\cap \varphi(a_2; v;\alpha) )=0 \big\},
  \end{gather*}
  we obtain that there is a definable system of sets $\vec{\CF}_2=(\CF_{s-1},
\CF_s)$ on $X_{s-1} \times X_s$ such that for any $n$-grid
$A_{s-1} \times  A _{s}$ in $(\vec\CF_2,\nu)$-general position and any $G\in \CG$ we have
  \[ |G \cap \left( A_{s-1} \times A_s \right)|\leq  \nu^2.\]

\medskip

Then $\vec \CF := \vec \CF_1\cup \vec\CF_2=(\CF_1,\dotsc \CF_s)$ is a definable system of sets on $X_1\tdt X_s$.

Let $A=A_1\tdt A_s$ be an $n$-grid on $X_1\tdt X_s$ in
$(\vec\CF,\nu)$-general position and $\alpha\in \Omega_1$.
We need to estimate from above $|Q_\alpha \cap (\bar Z_\alpha \times V) \cap
A|$.   
Let  $A_u := A_1\tdt A_{s-2}, A'_u := A_u \cap \bar{Z}_{\alpha}$ and $A_v := A_{s-1}\times A_s$, then
$|A'_u| \leq |A_u|\leq n^{s-2}$ and $|A_v|\leq n^2$.   Let $Q'_\alpha$ be $Q_\alpha$ viewed as a binary relation on $U \times V$,
we have 
$$|Q_\alpha\cap (\bar Z_\alpha\times V) \cap
A| = |Q'_\alpha \cap (\bar{Z}_{\alpha} \times V) \cap (A_u \times A_v)| \leq |Q'_\alpha \cap (A'_u \times A_v)|,$$
so it suffices to obtain the desired upper bound on $|Q'_\alpha \cap (A'_u \times A_v)|$.

From the $(\vec\CF,\nu)$-general position assumption and the choice of $\vec{\CF}$ we have: for any $a\in A'_u$ there are at most $(s-2)^{s-4} \nu^{2} n^{s-4}$ elements $a'\in A'_u$ such that
$|Q'_\alpha(a,v)\cap Q'_\alpha(a',v) \cap A_v|\geq  \nu^2$.

By assumption on $\CQ$ the definable family $\mathcal{Q}'_1 := \{Q'_{\alpha} : \alpha \in \Omega_1 \}$ of subsets of $U \times V$ satisfies the $\gamma$-ST property, and let $C': \mathbb{N} \to \mathbb{N}$ be as given by Definition \ref{def: gamma ST property} for $C(\nu) := (s-2)^{s-4} \nu$.
Then we have $|Q'_\alpha \cap (A'_u \times A_v)| \leq C'(\nu^2) n^{(s-1)-\gamma}$ (independently of $\alpha$), as required.

Thus the family $\CQ_1 = \{ Q_\alpha\colon \alpha\in \Omega_1\}$ admits $\gamma$-power saving.

\medskip

We now fix   $\alpha\in \Omega\setminus \Omega_1$.

By absolute irreducibility of $Q_\alpha$ and Remark \ref{rem: irred iff unique gen type} it is sufficient to show that there exists a  tuple $(a_1,\dotsc a_s)\in Q_\alpha$ generic
over $\alpha$ and  some tuple $\xi\in \acl(a_1,\dotsc,a_{s-2},\alpha) \cap
   \acl(a_{s-1},a_s,\alpha)$ of length at most $|\CL|$ such that
   $$(a_1,\dotsc,a_{s-2})\ind_\xi (a_{s-1},a_s).$$

  \emph{We  add $\acl(\alpha)$ to our language if needed and assume that $\alpha\in \dcl(\emptyset)$.}

By $|\CL|^+$-saturation of $\CM$, let $e=(e_1,\dotsc,e_{s-2})$ be a tuple in $\CM$ which is $\mfp$-generic
in $Z_\alpha$, namely $e\in Z_\alpha$ with $\dimp(e/\emptyset)=s-2$ (note that $Z_\alpha$ is $\emptyset$-definable). Let $\CM_0 = (M_0, \ldots) \preceq \CM$ be a model containing $e$ with $|\CM_0| \leq |\CL|$.

Let $\beta=(\beta_1,\dotsc,\beta_{s-2})\in U$ be a $\mfp$-generic point
in $Z_\alpha(e)$ over $M_0$, i.e.~$\beta\in Z_\alpha(e)$ and $\dimp(\beta/M_0)=s-3$.

Let $\delta=(\delta_1,\delta_2)$ be a tuple  in $\varphi(e,\CM,\alpha)\cap
\varphi(\beta,\CM,\alpha)$  with $\dimp(\delta/M_0\beta)=1$.
Without loss of generality  we assume that $\dimp(\delta_1/M_0\beta)=1$, namely
$\delta_1\models \mfp_{s-1}\restriction _{M_0\beta}$.
Note that such $\beta$ and $\delta$ can be chosen in $\CM$ by $|\CL|^+$-saturation.

We now collect  some properties of  $\beta$ and $\delta$.

\begin{claim}\label{claim:alpha-beta}
  \begin{enumerate}
  \item  $(e,\delta)$ is generic in $Q_\alpha$ over $\emptyset$.
   \item $\delta_1\ind_{M_0} \delta_2$  and $(\delta_1,\delta_2)\models
     \mfp_{s-1}\otimes \mfp_s |_{\emptyset}$.
   \item   $\beta\ind_{M_0} \delta$.
   \item $(\beta,\delta)$ is generic in $Q_\alpha$ over $\emptyset$.
  \end{enumerate}
\end{claim}
\begin{proof}
  (1)  We have, by our assumption above,   that
  $\dimp(\delta_1/M_0 \beta)=1$,  hence in particular $\dimp(\delta_1/e)=1$.
Since $\dimp(e/\emptyset)=s-2$ we have
 $\dimp((e,\delta)/\emptyset) \geq  s-1$ (as $(e,\delta_1) \models \left( \bigotimes_{i \in [s-2]} \mfp_i \right) \otimes \mfp_{s-1}|_{\emptyset}$ using that the types $\mfp_i, i \in [s-1]$ commute).   Since  $Q_\alpha$ is
 fiber-algebraic and $(e,\delta) \in Q_\alpha$, we also have
  $\dimp((e,\delta)/\emptyset) \leq s-1$ by Lemma \ref{lem: fib alg impl dim s-1}.

\medskip
\noindent (2) Since  $(e,\delta)$ is generic in $Q_{\alpha}$ over $\emptyset$ by (1),
by Proposition~\ref{prop:q-gen-ser} we have $(\delta_1,\delta_2)\models
  \mfp_{s-1}\otimes \mfp_s |_{\emptyset}$.

\medskip
\noindent   (3) As  $\beta\ind_{M_0} \delta_1$ and $\delta_2\in
  \acl(e\delta_1)\subseteq \acl(M_0\delta_1)$, we have   $\beta \ind_{M_0}
  (\delta_1,\delta_2)$.

\medskip
\noindent (4)  We have $(\beta,\delta)\in Q_\alpha$.  Since $\dimp(\beta/M_0)=s-3$ and
$\beta\ind_{M_0} \delta$, we have $\dimp(\beta/M_{0}\delta)=s-3$ (as $\beta \models \bigotimes_{i \in [s-3]} \mfp_i |_{M_0 \delta}$ by stationarity of non-forking over models), hence in particular
$\dimp(\beta/\delta)\geq s-3$.  Also, since
$\dimp(\delta/\emptyset)=2$ by (2),  we have $\dimp \left((\beta,\delta) /\emptyset \right) \geq
s-1$.  Since $Q_\alpha$ is fiber-algebraic we  also have
$\dimp((\beta,\delta)/\emptyset) \leq
s-1$, hence $\dimp \left( (\beta,\delta)/\emptyset \right)=
s-1$.
\end{proof}

Let $p(u) := \tp(\beta/M_0)$ and $q(v) := \tp(\delta/M_0)$, both are definable types over $M_0$ by stability.

We choose  canonical bases $\xi_p$ and $\xi_q$ of $p$ and
$q$, respectively; i.e.~$\xi_p, \xi_q$ are tuples of length $\leq |\mathcal{L}|$ in $\mathcal{M}_0^{\eq}$,   and for any
automorphism $\sigma$ of $\CM$ we have $\sigma(p|
\CM)=p{| \CM}$ if and only if
$\sigma(\xi_p)=\xi_p$ (pointwise); and    $\sigma(q | \CM)=q | \CM$ if and only if
$\sigma(\xi_q)=\xi_q$.

Note that $p$ does not fork over $\xi_p$ and
$q$ does not fork over $\xi_q$.

\begin{claim}\label{claim:alpha-beta2}
We  have:
\begin{enumerate}[(a)]
\item  $\xi_q\in \acl(\beta)$;
\item  $\xi_p\in \acl(\delta)$;
\item $\xi_q\in \acl(\xi_p)$;
\item $\xi_p\in \acl(\xi_q)$.
\end{enumerate}
\end{claim}
\begin{proof}
 (a) Assume not, then the orbit of $\xi_q$ under the automorphisms of
  $\CM$ fixing $\beta$ would be infinite. Hence we can   choose a model
  $\CN = (N, \ldots) \preceq \CM$ containing  $M_0\beta$ with $|N| \leq |\CL|$, and distinct types  $q_i \in S_v(N), i\in
  \omega$, each conjugate to $q | N$ under an automorphism
  of $\CN$ fixing $\beta$.

Let $\delta'_1\models \mfp_{s-1} |  N$.   For each $i\in
\omega$ we choose $\delta^i_2$ such that  $(\delta'_1,
\delta^i_2)\models  q_i$.   We have that $(\beta, \delta_1',\delta_2^i)\in Q_\alpha$,
hence, by fiber-algebraicity,  $|\{ \delta_2^i \colon i\in \omega \}|
\leq d$.   But all $q_i$ are pairwise distinct types, a contradiction.

\medskip

\noindent (b)  Since $\dimp(\beta/M_0\delta)=s-3$,
permuting variables if needed, we may assume that
$(\beta_1,\dotsc,\beta_{s-3})\models \mfp_1\otimes\dotsb \otimes
\mfp_{s-3}|_{M_0 \delta}$.

Assume (b) fails. Then  we can   find  a model
  $\CN \preceq \CM, |\CN| \leq |\CL|$ containing  $M_0\delta$, and distinct types  $p_i \in S(N), i\in
  \omega$, each conjugate to $p{\restriction} N$ under an automorphism
  of $\CN $ fixing $\beta$.
 Let
  \[ (\beta'_1,\dotsc,\beta'_{s-3}) \models \mfp_1\otimes\dotsc
    \otimes \mfp_{s-3} | N \]
    in $\CM$.
  For each $i\in
\omega$ we choose $\beta^i_{s-2}$ in $\CM$ such that
$$(\beta'_1,
\dotsc,\beta_{s-3}',\beta^i_{s-2})\models  p_i,$$
 and get a
contradiction as in (a).

\medskip

\noindent (c)  Since $\xi_q\in M_0$  and $p$ does not fork over $\xi_p$, we have
$\xi_q\ind_{\xi_p} \beta$, which by part (a) implies $\xi_q\in \acl(\xi_p)$.

\medskip

\noindent(d) Similar to (c).
\end{proof}

We have  that the tuple $(\beta,\delta)$ is generic in $Q_\alpha$  by Claim~\ref{claim:alpha-beta}(4). Let $\xi := \xi_p \cup \xi_q$, then $\xi \in \acl(\beta) \cap \acl(\delta)$ by Claim \ref{claim:alpha-beta2}. Finally $\delta \ind_{M_0} \beta$ by Claim \ref{claim:alpha-beta}(3), $\beta \ind_{\xi_p} M_0$ by the choice of $\xi_p$, hence $\delta \ind_{\xi_p} \beta$, and as $\xi_q \in \acl(\beta)$ we conclude $\beta \ind_{\xi} \delta$.

This finishes the proof of Theorem~\ref{thm:ser-main}, and hence of Theorem \ref{thm: stab main ser}.

\subsection{Proof of Theorem \ref{rem: power saving bound, main stable} for $s \geq 4$}\label{sec: MainThm for s geq 4}

Let $\mathcal{Q} = \left\{Q_{\alpha} : \alpha \in \Omega \right\}$ be a definable family of subsets of $X_1\tdt X_s$  satisfying the assumption of Theorem \ref{rem: power saving bound, main stable}, and say $\CQ$ is fiber-algebraic of degree $\leq d$. In particular, there exist $m \in \mathbb{N}$ and definable families $\mathcal{Q}_i, i \in [m]$ of absolutely $\mfp$-irreducible subsets of $X_1\tdt X_s$, so that for every $Q \in \mathcal{Q}$ we have $Q = \bigcup_{i \in [m]} Q_i$ for some $Q_i \in \mathcal{Q}_i$. Note that each $\mathcal{Q}_i$ is automatically fiber-algebraic, of degree $\leq d$.
By assumption each $\mathcal{Q}_i$ satisfies the $\gamma$-ST property for some fixed $\gamma > 0$, under any partition of the variables into two groups of size $2$ and $s-2$.

For each $i \in [m]$, let the definable family $\CQ'_i$ be as given by Theorem \ref{thm: stab main ser} for $\CQ_i$. That is, for each $i \in [m]$ the family 
$\CQ'_i$ admits $\gamma$-power saving, and for each $Q_i \in \CQ_i\setminus \CQ'_i$  the relation  $Q_i$ is in a $\mfp$-generic correspondence with an abelian group $G_{Q_i}$ type-definable in $\CM^{\eq}$ over a set of parameters $A_i$ of cardinality $\leq |\CL|$.
Consider the definable family 
$$\CQ' := \left\{Q \in \CQ : Q = \bigcup_{i \in [m]} Q_i \textrm{ for some } Q_i \in \CQ'_i \right\} \subseteq \CQ.$$
By Lemma \ref{lem: union power saving}, $\CQ'$ satisfies $\gamma$-power saving. On the other hand, from Definition \ref{def: group corresp}, if $Q \in \CQ$, $Q = \bigcup_{i \in [m]} Q_i$  with $Q_i \in \CQ_i$, and at least one of the $Q_i$ is in a $\mfp$-generic correspondence with a type-definable group, then $Q$ is also in a $\mfp$-generic correspondence with the same group. Hence every element $Q \in \CQ \setminus \CQ'$ is in  a $\mfp$-generic correspondence with a  group type-definable over some $A := \bigcup_{i \in [m]}A_i, |A| \leq |\CL|$.

\subsection{Proof of Theorem  \ref{rem: power saving bound, main stable} for ternary $Q$}\label{sec: stable main thm ternary}

In this subsection we reduce the remaining case $s=3$ of Theorem \ref{rem: power saving bound, main stable} to the case $s=4$.

Let $(X_i, \mfp_i)_{i  \in [3]}$ and a definable fiber-algebraic (say, of degree $\leq d$) family $\CQ$ of subsets of $X_1 \times X_2 \times X_3$ satisfy the assumption of Theorem \ref{rem: power saving bound, main stable} with some fixe $\gamma > 0$. In particular, there exist $m \in \mathbb{N}$ and fiber-algebraic (of degree $\leq d$) definable families $\mathcal{Q}_i, i \in [m]$ of absolutely $\mfp$-irreducible subsets of $X_1\tdt X_s$, so that for every $Q \in \mathcal{Q}$ we have $Q = \bigcup_{i \in [m]} Q_i$ for some $Q_i \in \mathcal{Q}_i$. By the same reduction as in Section \ref{sec: MainThm for s geq 4}, it suffices to establish the theorem separately for each $\CQ_i$, so we may assume from now on that additionally all sets  in $\CQ$ are absolutely $\mfp$-irreducible.

Consider the definable family $\CQ^* := \left\{ Q^* : Q \in \CQ \right\}$ of subsets of $X_1 \times X_2 \times X_3 \times X_4$, where 
  		\begin{gather*}
  		Q^* := \Big\{(x_2,x'_2,x_3,x'_3) \in X_2 \times X_2 \times X_3 \times X_3 : \\
	 \exists x_1 \in X_1 \, \big((x_1,x_2,x_3) \in Q \land (x_1, x'_2, x'_3) \in Q \big) \Big\}.
  	\end{gather*}

\begin{lem}\label{lem: Q' is fiber alg}
The definable family $\CQ^*$ of subsets of $X_2 \times X_2 \times X_3 \times X_3$ is fiber algebraic, of degree $\leq d^2$.	
\end{lem}
\begin{proof}
	We consider the case of fixing the first three coordinates of $Q^* \in \CQ^*$, all other cases are similar. Let $Q \in \CQ$, $(a_2,a'_2) \in X_2 \times X_2$ and $a_3 \in X_3$ be fixed. As $Q$ is fiber algebraic of degree $\leq d$, there are at most $d$ elements $x_1 \in X_1$ such that $(x_1,a_2,a_3) \in Q$; and for each such $x_1$, there are at most $d$ elements $x_3' \in X_3$ such that $(x_1, a'_2,x'_3) \in Q$. Hence, by definition of $Q^*$, there are at most $d^2$ elements $x'_3 \in X_3$ such that $(a_2, a'_2, a_3, x'_3) \in Q^*$.
\end{proof}

\begin{rem}\label{rem: p-syst from Q to Q'}
	Note that $\left(X'_i, \mfp'_i \right)_{i \in [4]}$ with $X'_1=X'_2 :=  X_2, X'_3  = X'_4 := X_3$ and $\mfp'_1 = \mfp'_2 :=  \mfp_2, \mfp'_3 = \mfp'_4 := \mfp_3$ is a $\mfp$-system with each $\mfp$ non-algebraic.
\end{rem}

The following lemma will be used to show that power saving for $\CQ^*$ implies power saving for $\CQ$ (this is a version of \cite[Proposition 3.10]{StrMinES} for families, which in turn is essentially \cite[Lemma 2.2]{raz}). We include a proof for completeness.

\begin{lem}\label{lem: bound Q from Q'}
For any finite $A_i \subseteq X_i, i \in [3]$ and $Q \in \CQ$, taking $\tilde{Q} := Q \cap \left(A_1 \times A_2 \times A_3\right) $ and $\tilde{Q}^* :=  Q^* \cap \left(A_2 \times A_2 \times A_3 \times A_3 \right)$ we have
\begin{gather*}
	\left \lvert \tilde{Q} \right \rvert \leq d \left \lvert A_1 \right \rvert^{\frac{1}{2}} \left \lvert \tilde{Q}^* \right \rvert^{\frac{1}{2}}.
\end{gather*}
\end{lem}
\begin{proof}
	Consider the (definable) set
	\begin{gather*}
		W  := \big\{(x_1, x_2, x'_2, x_3, x'_3) \in X_1 \times X^2_2 \times X_3^2 : \\
		 (x_1,x_2,x_3) \in Q \land  (x_1,x'_2,x'_3) \in Q \big \},
	\end{gather*}
	and let $\tilde{W} := W \cap \left( A_1 \times A_2^2 \times A_3^2\right)$. As usual, for arbitrary sets $S \subseteq B \times C$ and $b \in B$, we denote by $S_b$ the fiber $S_b = \{ c \in C : (b,c) \in S \}$.
	
	Note that $|\tilde{Q}| = \sum_{a_1 \in A_1} |\tilde{Q}_{a_1}|$ and $|\tilde{W}| = \sum_{a_1 \in A_1} |\tilde{Q}_{a_1}|^2$, which by the Cauchy-Schwarz inequality implies
	\begin{gather*}
		|\tilde{Q}| \leq |A_1|^{\frac{1}{2}} \left(\sum_{a_1 \in A_1}|\tilde{Q}_{a_1}|^2 \right)^{\frac{1}{2}} = |A_1|^{\frac{1}{2}}|\tilde{W}|^{\frac{1}{2}}.
	\end{gather*}
	For any tuple $\bar{a} := (a_2,a'_2,a_3,a'_3) \in \tilde{Q}^*$, the fiber $\tilde{W}_{\bar{a}} \subseteq A_1$ has size at most $d$ by fiber algebraicity of $Q$. Hence $|\tilde{W}| \leq d |\tilde{Q}^*|$, and so $|\tilde{Q}| \leq d |A_1|^{\frac{1}{2}}|\tilde{Q}^*|^{\frac{1}{2}}$.
\end{proof}

\begin{lem}\label{lem: pwr save Q' implies Q}
Assume that $\gamma' > 0$ and $\CQ^*$ admits $\gamma'$-power saving (with respect to the $\mfp$-system $(X'_i, \mfp'_i)_{i\in [4]}$ in Remark \ref{rem: p-syst from Q to Q'}).
 Then $\CQ$ admits $\gamma$-power saving  for $\gamma := \frac{\gamma'}{2}$.
	\end{lem}
\begin{proof}
	 By assumption there exist $\vec{\CF}' = (\CF'_i)_{i \in [4]}$ with $\CF'_1, \CF'_2$ definable families on $X_2$ and $\CF'_3, \CF'_4$ definable families on $X_3$, and a function $C': \mathbb{N} \to \mathbb{N}$, such that for any $Q^* \in \CQ^*, \nu,n \in \mathbb{N}$ and an $n$-grid $A' = \prod_{i \in [4]}A'_i$ on $X_2 \times X_2 \times X_3 \times X_3$ in $(\vec{\CF}',\nu)$-general position we have $|Q^* \cap A'| \leq C'(\nu) n^{3 - \gamma'}$.
	
	We take $\CF_1 := \emptyset$, $\CF_2 := \CF'_1 \cup \CF'_2$, $\CF_3 := \CF'_3 \cup \CF'_4$, $C(\nu) := d \cdot C'(\nu)^{\frac{1}{2}}$ and $\gamma := \frac{\gamma'}{2}$.

Assume we are given $Q \in \CQ, \nu,n \in \mathbb{N}$ and  $A_i \subseteq X_i, i  \in [3]$ with $|A_i| = n$  in $(\vec{\CF}, \nu)$-general  position.  By the choice of $\vec{\CF}$   it  follows  that the  grid  $A_2  \times A_2 \times  A_3 \times A_3$ is  in $\vec{\CF}'$-general  position, hence $|Q^*  \cap (A_2^2 \times  A_3^2)|  \leq C'(\nu)  n^{3  - \gamma'}$. By Lemma \ref{lem: bound Q from Q'} this implies
\begin{gather*}
	|Q \cap (A_1 \times A_2 \times A_3)| \leq d |A_1|^{\frac{1}{2}} |Q^* \cap (A_2^2 \times A_3^2)|^{\frac{1}{2}} \\
	\leq  d n^{\frac{1}{2}}  C'(\nu)^{\frac{1}{2}} n^{\frac{3}{2} - \frac{\gamma'}{2}} \leq C(\nu) n^{2 - \gamma}.
\end{gather*}
Hence $\CQ$ satisfies $\gamma$-power saving.
\end{proof}

We are ready to finish the proof of Theorem \ref{rem: power saving bound, main stable} (and hence of Theorem \ref{thm: stab main ineff}), the required bound on power saving follows from the proof.
\begin{proof}[Proof of Theorem \ref{rem: power saving bound, main stable} for $s=3$]
By the reduction explained above we may assume that $\CQ$ is a definable family of absolutely $\mfp$-irreducible sets and does not satisfy $1$-power saving.
Applying the case $s=4$ of Theorem  \ref{rem: power saving bound, main stable}  to the family $\CQ^*$ (note that $\CQ^*$ satisfies the assumption of Theorem \ref{rem: power saving bound, main stable}  by the reduction above and since $\CQ$ satisfies the $s=3$ assumption of Theorem \ref{rem: power saving bound, main stable}), we find a definable subfamily $\left(\CQ^* \right) '\subseteq \CQ^*$  such that the family 
$\left(\CQ^* \right) '$ admits $\gamma$-power saving, and for each $Q^* \in \CQ^* \setminus \left(\CQ^* \right) '$  the relation  $Q^*$ is in a $\mfp$-generic correspondence with an abelian group $G_{Q^*}$ type-definable in $\CM^{\eq}$ over a set of parameters of cardinality $\leq |\CL|$.

 Let $\CQ_0$ be the set of all $Q\in \CQ$ such that for some $u\subseteq [3]$ with $|u|=2$ for the projection $\pi_u(Q)$   of $Q$ onto $\prod_{i\in u} X_i$  we have $\dimp(\pi_u(Q)) < 2$.  
  By Claim~\ref{claim:p-dim-def}, the family $\CQ_0$ is definable and it follows from Proposition~\ref{prop:1-save-ser}
 that the family $\CQ_0$ admits $1$-power saving.

Consider the definable subfamily $\CQ' := \left\{ Q \in \CQ : Q^* \in \left(\CQ^* \right) ' \right\} \cup \CQ_0$ of $\CQ$. By Lemma \ref{lem: pwr save Q' implies Q}, as $\gamma \leq 1$, $\CQ'$ admits $\frac{\gamma}{2}$-power saving. On the other hand, if $Q \in \CQ \setminus \CQ'$, then $Q^* \in \CQ^* \setminus \left(\CQ^* \right) '$, hence there exists a small set $A \subseteq M$ and an abelian group $(G, \cdot, 1_G)$ type-definable over $A$ so that $Q^*$ is in a $\mfp$-generic correspondence with $G$.

This means that there exists a tuple $(g_2, g'_2,g_3,g'_3) \in G^4$ so  that $g_2 \cdot g'_2 \cdot g_3 \cdot g'_3 = 1_{G}$,  $g_2, g_3, g'_3$ are independent generics over $A$  and a tuple $(a_2,a'_2,a_3, a'_3) \in Q^* $ so that each of the elements $a_2, a'_2, a_3, a'_3$ is $\mfp$-generic over $A$ and each of the pairs $(g_2,a_2), (g'_2, a'_2), (g_3, a_3), (g'_3, a'_3)$ is inter-algebraic over $A$.

By definition of $Q^*$ there exists some $a_1 \in X_1$ such that $(a_1, a_2, a_3) \in Q$ and $(a_1, a'_2, a'_3) \in Q$.
We let $A' := A a'_3$ and $g_1 := g'_2 \cdot  g'_3$, and make the following observations.
\begin{enumerate}
	\item $g_1 \cdot g_2 \cdot g_3 = 1_{G}$ (using that $G$ is abelian).
	\item Each of the pairs $(a_1, g_1), (a_2,g_2), (a_3, g_3)$ is inter-algebraic over $A'$.
	
\noindent	The pairs $(a_2, g_2),  (a_3, g_3)$ are inter-algebraic over $A$ by assumption. Note that $a_1$ and $a'_2$ are inter-algebraic over $A'$ as $Q$ is fiber-algebraic,  so it suffices to show that $a'_2$ and $g_1$ are inter-algebraic over $A'$.
By definition $g_1 \in \acl(g'_2, g'_3) \subseteq \acl(a'_2, a'_3, A)  \subseteq \acl(a'_2,A')$. Conversely, as $g'_2 \in \acl(g'_2 \cdot  g'_3, g'_3) \subseteq \acl(g_1, A')$,  we have $a'_2 \in \acl(g'_2, A) \subseteq \acl(g_1, A')$.
\item $g_2$ and $g_3$ are independent generics in $G$ over $A'$.

\noindent By assumption $g_2 \ind_A g_3 g'_3$ and $a'_3$ is inter-algebraic with $g'_3$ over $A$, hence $g_2 \ind_{A'} g_3$.

\item $a_i \models \mfp_i|_{A'}$ for all $i \in [3]$.

\noindent For $i \in \{2,3\}$: as $g_i \ind_{A}  g'_3$ and $g'_3$ is inter-algebraic with $a'_3$ over $A$, we have $a_i \ind_A a'_3$, which by stationarity of $\mfp_i$ implies $a_i \models \mfp_i|_{A'}$.

\noindent For $i = 1$:
as $a_i  \models \mfp_i|_{A'}$ for $i \in \{2,3\}$ and $a_2 \ind_{A'} a_3$, it follows that $(a_2, a_3) \models (\mfp_2 \otimes \mfp_3)|_{A'}$ and the tuple $(a_1, a_2, a_3)$ is generic in $Q$ over $A'$ (as $\dimp(Q) = 2$ by the choice of $\CQ'$). But then $a_1 \models \mfp_1|_{A'}$ by the assumption on $Q$ and Proposition \ref{prop:q-gen-ser} (can be applied by absolute irreducibility of $Q$ and the choice of $\CQ'$).
\end{enumerate}
	It follows that $Q$ is in a $\mfp$-generic correspondence with $G$ over $A'$, witnessed by the tuples $(g_1, g_2, g_3)$ and $(a_1, a_2, a_3)$.
\end{proof}

\subsection{Discussion and some applications}\label{sec: stab applics}

First we observe how Theorem \ref{thm: stab main ineff}, along with some standard facts from model theory of algebraically closed fields, implies a higher arity generalization of the Elekes-Szab\'o theorem for algebraic varieties over $\CC$ similar to \cite{Bays}. Recall from  \cite{Bays} that a \emph{generically finite algebraic correspondence} between irreducible varieties $V$ and $V'$ over $\CC$ is a closed irreducible subvariety $C \subseteq V \times V'$ such that the projections $C \to V$ and $C \to V'$ are generically finite and dominant (hence necessarily $\dim(V) = \dim(V')$). And assuming that $W_i, W'_i$ and $V \subseteq \prod_{i \in [s]} W_i, V' \subseteq \prod_{i \in [s]} W'_i$ are irreducible algebraic varieties over $\CC$, we say that $V$ and $V'$ are in \emph{coordinate-wise correspondence} if there is a generically finite algebraic correspondence $C \subseteq V \times V'$ such that for each $i \in [s]$, the closure of the projection $(\pi_i \times \pi'_i)(C) \subseteq W_i \times W'_i$ is a generically finite algebraic correspondence between the closure of $\pi_i(V)$ and the closure of $\pi'_{i}(V')$.
\begin{cor}\label{cor: ES for any dim in ACF}

	Assume that $s \geq 3$, and $X_i \subseteq \mathbb{C}^{m_i}, i \in [s]$ and  $Q \subseteq \prod_{i \in [s]} X_i$ are irreducible algebraic varieties, with $\dim(X_i)=  d$. Assume also that for each $i \in [s]$, the projection $Q \to \prod_{j \in [s] \setminus \{i\}} X_i$ is dominant and generically finite. Let $m:=(m_1, ..., m_s)$, $t:= \max\{\deg(Q), \deg(X_1), ..., \deg(X_s)\}$.
	Then one of the following holds.
	\begin{enumerate}
		\item For every $\nu$ there exist $D = D(d,s,t,m)$ and $c = c(d,s,t,m,\nu)$ such that:
	 for any $n$ and finite $A_i \subseteq X_i, |A_i| = n$ such that $|A_i \cap Y_i| \leq \nu$ for every algebraic subsets $Y_i$ of $X_i$ of dimension $<d$ and degree $\leq D$, we have
	 $$|Q \cap A| \leq c n^{s-1 -\gamma}$$
	 for $\gamma =\frac{1}{16d-5}$ if $s \geq 4$, and $\gamma =  \frac{1}{2(16d-5)}$ if $s=3$.
		\item There exists a connected abelian complex algebraic group $(G, \cdot)$ with $\dim(G) = d$ such that $Q$ is in a coordinate-wise correspondence with
		$$Q' := \left\{(x_1, \ldots, x_s) \in G^s : x_1 \cdot \ldots \cdot x_s = 1_G \right\}.$$
	\end{enumerate}
\end{cor}

The above Corollary \ref{cor: ES for any dim in ACF} immediately follows from the following slightly more general
statement:

\begin{cor}\label{cor: ES for any dim in ACF new}

	Assume that $s \geq 3$, and $X_i \subseteq \mathbb{C}^{m_i}, i
        \in [s]$ are irreducible algebraic varieties with $\dim(X_i)=
        d$, and  let $\CQ$ be a definable family of subsets of $\prod_{i \in [s]} X_i$, each of
        Morley degree $1$. 
        Assume also that for each $Q \in \CQ$, $i \in [s]$, the projection $Q \to
        \prod_{j \in [s] \setminus \{i\}} X_i$ is Zariski dense  and
        is generically finite to one. Then there is a definable family $\CQ' \subseteq \CQ$ such that:
	\begin{enumerate} 
        \item $\CQ'$ admits $\gamma$-power saving 
          for $\gamma = \frac{1}{16d-5}$ if $s \geq 4$, and $\gamma =  \frac{1}{2(16d-5)}$ if $s=3$.
                \item For every $Q \in \CQ \setminus \CQ'$ there exists a connected abelian complex
                  algebraic group $(G, \cdot)$ with $\dim(G) = d$ such
                  that for some independent generics  
                  $g_1,\dotsc,g_{s-1}\in G$ and generic
                  $(q_1,\dotsc,q_s)\in Q$ we have that $g_i$ is
                  inter-algebraic with  $q_i$ for $i<s$ and $q_s$
                  inter-algebraic with $(g_1 \cdot g_2 \cdot  \dotsc  \cdot g_{s-1})^{-1}$. 
	\end{enumerate}
      \end{cor}
It is not hard to see that Corollary \ref{cor: ES for any dim in ACF new} implies \ref{cor: ES for any dim in ACF}.   Indeed, if  $Q$ is  an irreducible variety then  it has Morley degree
one. Let $\CQ$ be the family of \emph{all} irreducible algebraic varieties contained in $\prod_{i \in [s]}X_i$ of degree $\deg{Q}$, Morley rank $\MR(Q)$ and 
with all projections Zariski dense  and
        generically finite to one. It is a definable family in $\CM$ by definability of Morley rank and irreducibility (see e.g.~\cite[Theorem A.7]{freitag2017differential}), defined by a formula depending only on $m,t,s,d$; and $Q \in \CQ$. Applying Corollary \ref{cor: ES for any dim in ACF new} we can conclude depending on whether $Q \in \CQ'$ or $Q \in \CQ \setminus \CQ'$.

\begin{proof}[Proof of Corollary~\ref{cor: ES for any dim in ACF new}]

Let $\CM :=  (\CC, +, \times, 0,1) \models \ACF$, then $|\CL|= \aleph_0$ and $\CM$  is  an $|\CL|^+$-saturated structure. We recall that $\CM$ is a strongly minimal structure, in particular it is $\omega$-stable and has additive Morley rank $\MR$ coinciding with the Zariski dimension (see e.g.~\cite{MR1678602}).
	
For each $i$, as $X_i$ is irreducible, i.e.~has Morley degree $1$, let
$\mfp_i \in S_{x_i}(\CM)$ be the unique type in $X_i$ with
$\MR(\mfp_i) = \MR(X_i) = d$. By stability,  types are definable,
commute and are stationary after naming a countable elementary
submodel of $\CM$ so that all of the $X_i$'s are defined over
it.

Hence $(X_i, \mfp_i)_{i \in [s]}$ is a $\mfp$-system; and by
the additivity of Morley rank we see that $\MR(Y)  \geq d \dimp(Y)$ for any definable $Y \subseteq \prod_{i \in [s]}X_i$.

For any $Q \in \CQ$, since the projection of $Q$ onto $\prod_{i=1}^{s-1}X_i$ is Zariski
dense  and
generically finite, we have $\MR(Q)=d(s-1)$. 

Let
$Q'\subseteq Q$ be a definable set with $\RM(Q')=d(s-1)$.  Since $Q$
and $Q'$ have the same generic points, the item (2) is equivalent for
$Q$ and $Q'$.
Obviously $\gamma$-power saving for $Q$ implies $\gamma$-power saving
for $Q'$, and we observe that $\gamma$-power saving for $Q'$ with $0<\gamma<1$ implies $\gamma$-power saving for $Q$.   Let $Q'' := Q\setminus Q'$. Then, as $Q$ has Morley degree $1$,
$\MR(Q'')<d(s-1)$, hence $\dimp(Q'')\leq s-2$.  Applying
Lemma~\ref{lem:general-position-2} to $\CG := \{ Y'' \}$ we
obtain that $Y''$ has  $1$-power saving. Since $\gamma<1$, it follows that $Y=Y'\cup Y''$ also has $\gamma$-power saving.

As by assumption every $Q \in \CQ$ has generically finite projections, after removing a
subset of smaller Morley rank we may assume that $Q$ is fiber-algebraic. This can be done uniformly for the family by \cite[Theorem A.7]{freitag2017differential} (however, on this step we have to pass from a family of algebraic sets to a family of constructible sets, that is why we can only use bounds from Corollary \ref{fac: algebraic ST}(2) but not from Fact \ref{fac: stronger bounds for algebraic} in the following), hence we may assume that the family $\CQ$ consists of fiber algebraic sets of fixed degree.

As $\dim(X_i) = d$, it follows that $X_i$ has a  generically
finite-to-one projection onto $\CC^{d}$, hence, after possibly a coordinate-wise correspondence, we may assume that $Q \subseteq \prod_{i \in [s]} \mathbb{C}^d$ --- again, uniformly for the whole family $\CQ$. By Corollary \ref{fac: algebraic ST}(2), every definable family of sets $Y \subseteq \CC^{2d}  \times \CC^{(s-2)d}$ satisfies the $\left(\frac{1}{8d-1} \right)$-ST property. Applying Theorem \ref{rem: power saving bound, main stable} (we are using once more that irreducible components are uniformly definable in families in $\ACF$, see \cite[Theorem A.7]{freitag2017differential}) we find a definable subfamily $\CQ'$ with  $\gamma$-power saving for the stated $\gamma$.

Every  type-definable group in $\CM^{\eq}$ is actually definable (by $\omega$-stability, see e.g.~\cite[Theorem 7.5.3]{MR1924282}),
and every group interpretable in an algebraically closed field is
definably isomorphic to an algebraic group (see e.g.~\cite[Proposition
4.12  + Corollary 1.8]{MR1678602}).
Thus, for $Q \in \CQ \setminus \CQ'$, there exists an abelian connected complex algebraic group $(G, \cdot)$, independent generic elements $g_1, \ldots, g_{s-1} \in G$ and $g_s \in G$ such that $g_1 \cdot \ldots \cdot g_s = 1$ and generic $a_i \in X_i$ inter-algebraic with $g_i$, such that $(a_1, \ldots, a_s) \in Q$. In particular, $\dim(G) = \dim(X_i) = d$. And, by irreducibility of $Q$, hence uniqueness of the generic type, such $a_i$'s exist \emph{for any} independent generics $g_1, \ldots, g_{s-1}$.
As the model-theoretic algebraic closure coincides with the field-theoretic algebraic closure, by saturation of $\CM$ this gives the desired coordinate-wise correspondence.
\end{proof}

 \begin{rem}
 Failure of power saving on arbitrary grids, not necessarily in a general position, does not guarantee coordinate-wise correspondence with an \emph{abelian} group in Corollary \ref{cor: ES for any dim in ACF}.
 For example, let $(H,\cdot)$ be the Heisenberg group of $3\times 3$ matrices over $\CC$, viewed as a definable group in $\CM := (\CC,+,\times)$. For $n \in \NN$, consider the subset of $H$ given by
 \begin{gather*}
 	A_n := \left\{ \begin{pmatrix}
1 & n_1 &n_3\\
0 & 1 & n_2\\
0 & 0 & 1
\end{pmatrix} \colon n_1, n_2, n_3 \in \NN, n_1,n_2 < n, n_3 < n^2 \right\}.
 \end{gather*}
It is not hard to see that $|A_n| = n^4$. For the definable fiber-algebraic relation $Q(x_1, x_2, x_3, x_4)$ on $H^4$ given by $x_1 \cdot x_2 = x_3 \cdot x_4$ we have $|Q \cap A_n^4| \geq \frac{1}{16}(n^4)^3 = \Omega(|A_n|^3)$.

However, $Q$ is not in a generic correspondence with an \emph{abelian} group. Indeed, the sets $A_n \subseteq H, n \in \NN$ are not in an $(\CF,\nu)$-general position for any $\nu$, with respect to the definable family $\CF = \{u_1 - u_2 = c : c \in \CC \}$ of subsets of $H$.
 \end{rem}

However, restricting further to the case $\dim(X_i) = 1$ for all $i \in [s]$, the general position requirement is satisfied automatically: for any definable set $Y \subseteq X_i$, $\dim(Y) < 1$ if and only if $Y$ is finite; and for every definably family $\CF_i$ of subsets of $X_i$ there exists some $\nu_0$ such that for any $Y \in \CF_i$, if $Y$ has cardinality greater than $\nu_0$ then it is infinite.
Hence (using the classification of one-dimensional connected complex algebraic groups) we obtain the following simplified statement.
\begin{cor}\label{cor: ES for 1 dim ACF}
	Assume $s \geq 3$, and let $Q \subseteq \mathbb{C}^{s}$  be an irreducible algebraic variety so that for each $i \in [s]$, the projection $Q \to \prod_{j \in [s] \setminus \{i\}} \mathbb{C}^{s}$ is generically finite.
	Then exactly one of the following holds.
	\begin{enumerate}
		\item There exist $c$ depending only on $s, \deg(Q)$ such that:
	 for any $n$ and $A_i \subseteq \CC_i, |A_i| = n$ we have
	 $$|Q \cap A| \leq c n^{s-1 -\gamma}$$
	 for $\gamma = \frac{1}{11}$ if $s \geq 4$, and $\gamma =  \frac{1}{22}$ if $s=3$.
		\item For $G$ one of $(\mathbb{C}, +)$, $(\mathbb{C}, \times)$ or an elliptic curve, $Q$ is in a coordinate-wise correspondence with
		$$Q' := \left\{(x_1, \ldots, x_s) \in G^s : x_1 \cdot \ldots \cdot x_s = 1_G \right\}.$$
	\end{enumerate}
	
\end{cor}
%
%
%

 \begin{rem}\label{rem: mut excl}
 	We expect that the two cases in Corollary \ref{cor: ES for any dim in ACF} are not mutually exclusive (a potential example is suggested in \cite[Remark 7.14]{breuillard2021model}), however they are mutually exclusive in the $1$-dimensional case in Corollary \ref{cor: ES for 1 dim ACF}. The proof of this for $s=3$ is given in \cite[Proposition 1.7]{StrMinES}, and the argument generalizes in a straightforward manner to an arbitrary $s$.
 \end{rem}

%

We remark that the case of complex algebraic varieties corresponds to a rather special case of our general Theorem \ref{thm: stab main ineff} which also applies e.g.~to the theories of differentially closed fields or compact complex manifolds (see Facts \ref{fac: DCF0 dist exp} and \ref{fac: CCM dist exp}). For example:
\begin{rem}
	 Given definable strongly minimal sets $X_i, i \in [s]$ and  a fiber-algebraic $Q \subseteq \prod_{i \in [s]} X_i$ in a differentially closed field  $\CM$ of characteristic $0$, we conclude that either $Q$ has power saving (however, we do not have an explicit exponent here, see Problem \ref{prob: cell decomp DCF0}), or that $Q$ is in correspondence with one of the following strongly minimal differential-algebraic groups:  the additive, multiplicative or elliptic curve groups over the field of constants  $\mathcal{C}_{\CM}$ of $\CM$, or a \emph{Manin kernel} of a simple abelian variety $A$ that does not descend to $\mathcal{C}_{\CM}$  (i.e.~the Kolchin closure of the torsion subgroup of $A$; we rely here on the Hrushovski-Sokolovic trichotomy theorem, see e.g.~\cite[Section 2.1]{MR3641651}).
\end{rem}

\section{Main theorem in the $o$-minimal case}\label{sec: main thm omin}

\subsection{Main theorem and some reductions}\label{sec: o-min main thm}

In this section we will assume that $\CM = (M, \ldots)$ is an o-minimal, $\aleph_0$-saturated $\CL$-structure expanding a group (or just with definable Skolem functions).
We shall use several times the following basic property of o-minimal structures:
\begin{fact}\label{generic neighborhoods}\cite[Fact 2.1]{peterzil2019minimalist}
Assume that $a\in M^n$ and $A \subseteq B \subseteq M$ are small sets. For every definable open neighborhood $U$ of $a$ (defined over arbitrary parameters), there exists $C \supseteq A$, $\acl$-independent from $aB$ over $A$, and a $C$-definable open $W \subseteq U$ containing $a$. In particular, $\dim(a/A) = \dim(a/C)$ and $\dim(aB/C) = \dim(aB/A)$.
\end{fact}

For the rest of the section we assume that $s \geq 3$ and for $i=1,\ldots, s$, we have $\emptyset$-definable sets $X_i$ with $\dim X_i=m$ for all $i \in [s]$ (throughout the section, $\dim$ refers to the standard notion of dimension in $o$-minimal structures). We also have an $\emptyset$-definable set $Q\sub  \overline{X} := X_1\times \cdots \times X_s$,  with $\dim Q=(s-1)m$, and such  that $Q$ is fiber algebraic of degree $d$, for some $d \in \mathbb N$ (see Definition \ref{def: fiber alg}).

The following is the equivalent of Definitions \ref{def: p-gen pos} and \ref{def: power saving} in the $o$-minimal setting.
\begin{defn}
 For $\gamma \in \mathbb{R}_{>0}$, we say that  $Q\sub \overline{X}$ satisfies \emph{$\gamma$-power saving} if there are definable families $\overrightarrow{\CF}=(\CF_1,\ldots,\CF_s)$, where each $\CF_i$ consists of subsets of $X_i$ of dimension smaller than $m$, such that for every $\nu \in \mathbb N$  there exists a constant $C=C(\nu)$ such that: for every $n \in \mathbb{N}$ and every $n$-grid $\overline{A}:=A_1\times\cdots\times A_k\sub \overline{X}, |A_i| =n$
in $(\overrightarrow{\CF},\nu)$-general position (i.e.~for every $i \in [s]$ and $S \in \CF_i$ we have $|A_i \cap S| \leq \nu$) we have
$$|Q\cap \overline{A}|\leq C n^{(s-1)-\gamma}.$$
\end{defn}
It is easy to verify that if $Q_1,Q_2\sub \overline{X}$ satisfy $\gamma$-power saving  then so does $Q_1\cup Q_2$. Before stating our main theorem in the $o$-minimal case, we define:
\begin{defn}\label{def: inf neighb} Given a finite tuple $a$ in an o-minimal structure $\CM$, we let $\mu_{\CM}(a)$ be the \emph{infinitesimal neighborhood of $a$},  namely the intersection of all $\CM$-definable open neighborhoods of $a$. It can be viewed as a partial type
over $\CM$, or we can identify it with  the set of its realizations in an elementary extension of $\CM$.
\end{defn}

\begin{thm}\label{thm: main thm o-min1} Under the above assumptions, one  of the
following  holds.
\begin{enumerate}
\item The set $Q$ has $\gamma$-power saving, for $\gamma = \frac{1}{8m-5}$ if $s \geq 4$, and $\gamma = \frac{1}{16m-10}$ if $s=3$.

\item There exist (i) a tuple $\bar a=(a_1,\ldots, a_s)$ in $\CM$ generic in $Q$, (ii) a substructure $\CM_0 := \dcl(\bar a)$ of $\CM$ of cardinality $\leq |\CL|$ (iii) a
type-definable abelian group $(G, +)$ of dimension $m$, defined over $M_0$ and (iv) $M_0$-definable bijections
$\pi_i:\mu_{\CM_0}(a_i)\cap X_i\to G, i \in [s]$, sending $a_i$ to $0=0_G$, such that
$$\pi_1(x_1)+\cdots+\pi_s(x_s)=0\Leftrightarrow Q(x_1,\ldots,x_s)$$
for all $x_i \in \mu_{\CM_0}(a_i)\cap X_i, i \in [s]$.
\end{enumerate}
\end{thm}


We begin working towards a proof of Theorem \ref{thm: main thm o-min1}.

\medskip
\noindent {\bf Notation}
\begin{enumerate}
\item For $i,j\in [s]$, we write $\overline{X}_{i,j}$ for the set $\prod_{\ell\neq i,j}X_\ell$.

\item
For $z\in X_1\times X_2$ and $V\sub \overline{X}_{1,2}$ we write $$Q(z,V) := \{w\in V:(z,w)\in
Q\}.$$ We similarly write $Q(U,w)$, for $U\sub X_1\times X_2$ and $w\in \overline{X}_{1,2}$.

\end{enumerate}

\begin{lem} The following are easy to verify:
\begin{enumerate}
\item For every $z\in X_1\times X_2$, $\dim Q(z,\overline{X}_{1,2})\leq (s-3)m$.

\item If $\alpha =(z,w)\in (X_1\times X_2)\times \overline{X}_{1,2}$ is generic in $Q$ then for every neighborhood $U\times V$ of $\alpha$,
$\dim Q(z,V)=(s-3)m$ and $\dim Q(U,w)=m$.

\end{enumerate}
\end{lem}

We will need to consider a certain \emph{local} variant of the property (P2) from Section \ref{sec: bij Q any arity}.

\begin{defn}
Assume that  $\alpha=(z,w)\in Q\cap (X_1\times X_2)\times \overline{X}_{1,2}$.
\begin{itemize}
	\item We say that $Q$
\emph{has the $(P2)_{1,2}$ property near $\alpha$}
 if  for all $U'\sub X_1\times X_2 $ and $V'\sub \overline{X}_{1,2}$ neighborhoods  of
$z,w$ respectively,
\begin{equation} \label{eq2} \dim  Q(U',w)=m \mbox{ and } \dim Q(z,V')=(s-3)m,\end{equation}  and there are open neighborhoods
$U\times V\ni (z,w)$ in $(X_1\times X_2)\times \overline{X}_{i,j}$ such that
\begin{equation}\label{eq1} Q(U,w)\times Q(z,V) \sub Q,\end{equation}
(namely, for every $z_1\in U$ and $w_1\in V$, if $(z_1,w),(z,w_1)\in Q$ then
$(z_1,w_1)\in Q$).

\item We say that {\em $Q$ satisfies  the $(P2)_{i,j}$-property near $\alpha$}, for $1\leq i< j\leq s$, if the above holds under
the coordinate permutation of $1, 2$ and $i,j$, respectively.

\item We say that {\em $Q$ satisfies  the $(P2)$-property near $\alpha$} if it has the $(P2)_{i,j}$-property
 for all $1\leq i<j\leq s$.
\end{itemize}

\end{defn}

\begin{rem}
	Note that if $U,V$ satisfy (\ref{eq1}), then also every $U_1\sub U$ and $V_1\sub V$ satisfy it.
 Note also that under the above assumptions,
we have $\dim(Q(U,w)\times Q(z,V))=(s-2)m$.
\end{rem}

\begin{defn}
\begin{itemize}
	\item Let  $Q_{i,j}^*$ be the set of all $\alpha \in Q$ such that $Q$ satisfies  $(P2)_{i,j}$
near $\alpha$.
\item Let $Q^*=\bigcap_{i\neq j} Q_{i,j}^*$ be the set of all  $\alpha \in Q$ near which $Q$ satisfies $(P2)$.
\end{itemize}
Clearly, $Q^{*}_{i,j}$ and $Q^*$ are $\0$-definable sets.	
\end{defn}

The main ingredient  towards the proof of Theorem \ref{thm: main thm o-min1} is the following:
\begin{thm}\label{main} Assume that $Q$ does not satisfy $\gamma$-power saving  for $\gamma$ as in Theorem \ref{thm: main thm o-min1}(1).
 Then  $\dim Q^*=\dim Q=(s-1)m$.\end{thm}

\subsection{The proof of  Theorem \ref{main}}\label{sec: power saving case proof omin}

The following is an analog of Lemma \ref{lem:general-position-2} in the $o$-minimal setting.
\begin{lem} \label{small dim} Let $\{Z_t:t\in T\}$ be a definable family of subsets of $\overline{X}$, each fiber-algebraic of degree $\leq d$ with
$\dim(Z_t)<(s-1)m$. Then there exist definable families $\CF_i$, $i \in [s]$, each consisting of subsets of $X_i$
of dimension smaller than
$m$, such that for every  $\nu\in \mathbb N$, if $\bar A\sub \overline{X}$ is
 an $n$-grid in $(\overrightarrow{\CF},\nu)$-general position then for every $t\in T$,
$$|\bar A\cap Z_t|\leq s d (\nu-1)n^{s-2}.$$

In particular, each $Z_t, t \in T$ satisfies $1$-power saving.

\end{lem}
\begin{proof}
 For $t\in T$ and $a_1\in X_1$ we let $$Z_{ta_1} := \{(a_2,\ldots,a_s)\in X_2\times\cdots \times X_s:(a_1,a_2,\ldots,a_s)\in Z_t\}.$$ For $i \in [s-1]$, we similarly define $Z_{ta_1\cdots a_i}\sub X_{i+1}\times\cdots\times X_s.$

 \noindent (1)  For $t\in T$, we let
 $$Y_t^1 := \left\{a_1\in X_1:\dim(Z_{t{a_1}})=(s-2)m \right\}.$$ By
our assumption on $\dim Z_t$, $\dim Y_t^1<m$. Let $\CF_1 := \{Y_t^1:t\in T\}$.

\medskip
\noindent (2) For $t\in T$ and $a_1\notin Y_t^1$, let $$Y^2_{ta_1} := \{a_2\in X_2:\dim (Z_{ta_1a_2})=(s-3)m\}.$$ Then define
$\CF_2 := \left\{Y^2_{ta_1}: t \in T, a_1\notin Y_t^1 \right\}.$ Note that whenever $a_1\notin Y^1_t$, $\dim (Z_{ta_1})<(s-2)m$ and
therefore the set $Y^2_{ta_1}$ has dimension smaller than $m$.

For $i=1,\ldots, s-2$, we continue in this way to define a family $\CF_i=\{Y^i_{ta_1\cdots a_{i-1}}\}$ of subsets of $X_i$ as follows: for $a_1\notin Y_t^1$, $a_2\notin Y^2_{ta_1}$, $a_3\notin Y^3_{ta_1a_2},\ldots, a_{i-1}\notin Y^{i-1}_{ta_1\cdots a_{i-2}}$, we let
$$Y^i_{ta_1\cdots a_{i-1}} := \{a_i\in X_i:\dim (Z_{ta_1\cdots a_i})=(s-(i+1))m\},$$ and let
$$\CF_i := \left\{Y^i_{ta_1\cdots a_{i-1}}:t\in T, a_1\notin Y^1_t, a_2\notin Y^2_{ta_1},\ldots,a_{i-1}\notin Y^{i-1}_{ta_1\cdots a_{i-2}}\right\}.$$

Finally, for $a_1,\ldots, a_{s-2}$ such that $a_i\notin Y^i_{ta_1\cdots a_{i-1}}$ for $i=1,\ldots, s-2$, we let
$$Y^{s-1}_{ta_1\cdots a_{s-2}} := \pi_{s-1}(Z_{ta_1\ldots  a_{s-2}})\subseteq X_{s-1},$$ and let
$$\CF_{s-1} := \left\{Y^{s-1}_{ta_1\cdots a_{s-2}} : t\in T, a_1\notin Y_t^1 ,\ldots, a_{s-2}\notin Y^{s-2}_{ta_1\cdots a_{s-3}} \right\} .$$
We provide some details on why the families $\vec{\CF} := (\CF_i : i \in [s])$ satisfy the requirement.

Assume that $n, \nu \in \mathbb{N}$ and  $\bar A\sub\overline{X}$ is an $n$-grid which is in
$(\overrightarrow{\CF},\nu)$-general position, and fix $t\in T$.

Because $|A_1\cap Y_t^1|<\nu$ there are
at most $\nu-1$ elements $a_1\in \pi_1(Z_t \cap \bar A)\cap Y_t^1$, and
for each such $a_1$ there are at most $d n^{s-2}$ elements in $Z\cap \bar A$ which project to it. Indeed, this is true because $Z_{ta_1}$ is fiber-algebraic of degree $\leq d$, so for every tuple $(a_2,\ldots,a_{s-1})\in A_2\times\cdots A_{s-1}$ (and there are at most $n^{s-2}$ such tuples) there are $\leq d$ elements  $a_s\in A_s$ such that $(a_2,\ldots, a_{s-1},a_s)\in \left( A_2\times\cdots \times A_s \right) \cap Z_{ta_1}$.

So, altogether there are at most $d (\nu-1)n^{s-2}$ elements  $(a_1,\ldots, a_s)\in \bar A\cap Z_t$ for which $a_1\in Y_1^t$.
On the other hand, there are at most $n-\nu\leq n$ elements $a_1\notin Y^1_t$. We now compute for how many  $\bar a\in \bar A\cap Z_t$  we have $a_1\notin Y^1_t$.

By definition, $\dim(Z_{ta_1})<(s-2)m$, so now we consider two cases, $a_2\in Y^2_{ta_1}$ and $a_2\notin Y^2_{ta_1}$. In the first case, there are at most $\nu-1$ such $a_2$, by general position, and as above, for each such $a_2$ there are at most $d n^{s-3}$ tuples $(a_3,\ldots,a_s)\in A_3\times\cdots \times A_s$ such that $(a_2,a_3,\ldots,a_s)\in Z_{ta_1}$. Thus all together there are $n(\nu-1)d n^{s-3}= d(\nu-1)n^{s-2}$ elements $\bar a\in \bar A\cap Z_t$ such that $a_1\notin Y_t^1$ and $a_2\in Y_t^2$. There are at most $(n-\nu)\leq n$ elements $a_2\in A_2$ which are not in $Y^2_{ta_1}$. Of course, there are at most $n^2$ pairs $(a_1,a_2)$ such that $a_1\notin Y^1_{t}$ and $a_2\notin Y^2_{ta_1}$, and we now want to compute how many  $\bar a\in \bar A\cap Z_t$ project onto such $(a_1,a_2)$. Repeating the same process along the other coordinates, we see that there are at most $(s-2) d (\nu-1)n^{s-4}$ elements which project into each such $(a_1,a_2)$, so all together there are at most $(s-2) d (\nu-1)n^{s-2}$ tuples $\bar{a} \in \bar A\cap Z_t$ for which $a_1\notin Y_t^1$ and $a_2\notin Y^2_{ta_1}$. If we add it all we get at most $s d (\nu-1)n^{s-2}$ elements in $\bar A\cap Z_t$, which concludes the proof of the lemma.
\end{proof}

\begin{cor} \label{cor-alpha}Assume that  $Q\sub \overline{X}$ does not satisfy $1$-power saving and that  $Z\sub Q$
is a definable set with $\dim Z<(s-1)m$. Then  $Q':=Q\setminus Z$ also does not satisfy $1$-power saving.
\end{cor}
\begin{proof}
	Indeed, Lemma \ref{small dim} (applied to the constant family) implies that $Z$ itself satisfies $1$-power saving,
and since $\gamma$-power saving is preserved under union then it fails for $Q'$.
\end{proof}

 In order to prove Theorem \ref{main}, it is sufficient to
prove the following:
\begin{prop} \label{main-prop}Let $Q'\sub Q$ be a definable set and
assume that  there exist $i\neq j \in [s]$ such that $\dim(Q'\cap
Q^*_{i,j})<(s-1)m$. Then $Q'$ satisfies $\gamma$-power saving  for $\gamma$ as in Theorem \ref{thm: main thm o-min1}(1).
\end{prop}

 Let us first see that indeed Proposition \ref{main-prop} quickly implies Theorem \ref{main}.
Let $\gamma$ be as in Theorem \ref{thm: main thm o-min1}(1). Assuming that $Q$ does not have $\gamma$-power saving, Proposition \ref{main-prop} with $Q' := Q$ implies that
 $\dim(Q^*_{1,2})=(s-1)m$.  Also, if we take $Q'' := Q\setminus Q_{1,2}^*$ then clearly $Q''\cap Q_{1,2}^*=\emptyset$ and therefore by the same proposition $Q''$ satisfies $\gamma$-power saving, and therefore $Q_{1,2}^*$ does not satisfy $\gamma$-power saving. We can thus replace $Q$ by $Q_1:=Q^*_{1,2}$ and retain  the original properties of $Q$ together with the fact that $Q_1$ has $(P2)_{1,2}$ at every $\alpha \in Q_1$. Next we repeat the process with respect to every $(i,j)\neq (1,2)$ and eventually obtain a definable set $Q'\sub Q$ of dimension $(s-1)m$ such that $Q'$ satisfies $(P2)$ at every point --- establishing Theorem \ref{main}.
\medskip

\noindent{\bf Proof of Proposition \ref{main-prop}.}

\medskip

Let $Q'\sub Q$ and $\gamma$ be as in Proposition \ref{main-prop}. It is sufficient to prove the proposition for $Q_{1,2}^*$ (the case of arbitrary $i \neq j \in [s]$ follows by permuting the coordinates accordingly).
  If $\dim Q'<(s-1)m$ then by Lemma \ref{small dim} $Q'$ satisfies $1$-power saving, hence $\gamma$-power saving. Thus we may assume that $\dim Q'=(s-1)m$, and hence, by throwing away a set of smaller dimension,  we may assume that $Q'$ is open in $Q$. It is
then easy to verify that $(Q')^*_{1,2}=Q_{1,2}^*\cap Q'$. Hence, without loss of generality,
$Q=Q'$. We now assume that $\dim Q_{1,2}^*<(s-1)m$ and therefore, by Lemma \ref{small dim}, $Q_{1,2}^*$ has $\gamma$-power saving. Thus, in order to show that $Q$ has $\gamma$-power saving, it is sufficient to prove that  $Q\setminus Q_{1,2}^*$ has $\gamma$-power saving, so we assume from now on that $Q_{1,2}^*=\0$.

We let $U := X_1\times X_2$ and $V := \overline{X}_{1,2}$.

\begin{claim} \label{claim.5}For every $w\in V$, the set
$$X_{w} := \left\{w'\in V:\dim (Q(U,w)\cap Q(U,w'))=m\right\}$$
 has
dimension strictly smaller than $(s-3)m$. Moreover, the set $X_w$ is fiber algebraic in $X_3\times \cdots \times X_s$.\end{claim}
\begin{proof}
We assume that relevant sets thus far (i.e.~$X_i, Q, U,V, Q^*_{i,j}$) are defined over $\0$.  Now, if $\dim(X_w)=(s-3)m$ (it is not hard to see that it cannot be bigger), then by $\aleph_0$-saturation of $\CM$ we may take $w'$ generic in $X_w$ over $w$, and then
$u'$ generic in $Q(U,w)\cap Q(U,w')$ over $w,w'$. Note that the fiber-algebraicity
of $Q$ implies that $\dim(Q(u',V))\leq (s-3)m$, and since $\dim (w'/wu')= \dim(w'/w) =(s-3)m$ it follows that $w'$ is generic in both $X_{w}$ and
$Q(u',V)$ over $wu'$, so in particular, $\dim X_w=\dim Q(u',V)=(s-3)m$.

We claim that $(u',w')\in Q_{1,2}^*$. Indeed, by our assumption, $$\dim(u'/ww')=\dim (Q(U,w)\cap Q(U,w'))=\dim Q(U,w)=m.$$ 
Thus, there exists an open $U_0\sub U$ containing $u'$, such that $U_0\cap Q(U,w)=U_0\cap Q(U,w')$, or, said differently, 
$Q(U_0,w)=Q(U_0,w')$. By Fact \ref{generic neighborhoods}, we may assume that the tuple $(w,w',u')$ is independent from the parameters defining $U_0$ over $\0$.
Thus, without loss of generality, $U_0$ is definable over $\0$.
The set $W_1 := \{v\in V: Q(U_0,w)\sub Q(U,v)\}$ is defined over $w$ and  the set $Q(u',V)$ is defined over $u'$, and both contain $w'$. Since $\dim(w'/w,u')=(s-3)m$ then 
 $\dim(W_1\cap Q(u',V))=(s-3)m$. We can therefore find an open $V_0\sub V$ such that $Q(u',V_0)\sub W_1$.
 Now, by the definition of $W_1$, we have $Q(U_0,w)\times W_1\sub Q$, and hence $Q(U_0,w)\times Q(u',V_0)\sub Q$ and therefore (since $Q(U_0,w)=Q(U_0,w')$), 
  $Q(U_0,w')\times Q(u,V_0)\sub Q$. This shows that $(u',w')\in Q^*_{1,2}$, contradicting our assumption that $Q_{1,2}^*=\0$.

To see that $X_w$ is fiber algebraic, assume towards contradiction that there exists a tuple
$(a_3,\ldots,a_{s-1})\in X_3\times\cdots\times X_{s-1}$ for which there are infinitely many $a_s \in X_s$ such that $(a_3,\dots,a_s)\in X_w$
(the other coordinates  are treated similarly). We can now pick such $a_s$ generic over
$w,a_3,\ldots,a_{s-1}$ and then pick $(a_1,a_2)\in Q(U,w)\cap Q(U,a_3,\ldots, a_s)$ generic over $w,a_3,\ldots,a_s$. 
Because $\dim(a_1,a_2/w)= \dim(a_1,a_2/w, a_3, \ldots, a_s)$ it follows by the additivity of dimension that for any subtuple $a'$ of $a_3, \ldots, a_s$ we have  $\dim(a'/w, a_1,a_2)=\dim(a'/w)$.
It follows that 
$$0<\dim(a_s/w,a_3, \ldots,a_{s-1}) = \dim(a_s/w,a_1,a_2,a_3,\ldots,a_{s-1}).$$ 
Since $Q(a_1,a_2,a_3, \ldots, a_s)$ holds, it follows that $Q(a_1,a_2, a_3, \ldots, a_{s-1}, X_n)$ is infinite --- contradicting the fiber-algebraicity of $Q$.
\end{proof}
We similarly have:
\begin{claim} \label{claim.7} For every $u\in U$, the set
$$X^u := \left\{u'\in U:\dim (Q(u,V)\cap Q(u',V))=(s-3)m \right\}$$
 has
dimension strictly smaller than $m$. Moreover, the set $X^u$ is fiber-algebraic in $X_1\times X_2$.\end{claim}




\begin{lem}\label{claim2} There exist $s$ definable families
$\vec{\CF}=(\CF_1,\ldots,\CF_s)$ of subsets of $X_1,\ldots,X_s$, respectively,
each containing only sets of dimension strictly smaller than $m$, such that for every $\nu\in
\mathbb N$ and every $n$-grid $\bar A\sub \overline{X}$ in
$(\vec {\CF},\nu)$-general position, we have the following.
\begin{enumerate}
\item For all $w,w'\in A_3\times \cdots \times A_s$, if $|Q(A_1\times A_2,w)\cap Q(A_1\times A_2,w' )|\geq d \nu$ then $w'\in X_w$.

\item For all $w\in A_3\times \cdots \times A_s$, there are at most
$C(\nu)n^{s-4}$ elements $w'\in A_3\times \cdots \times A_s$ such that $|Q(A_1\times A_2,w)\cap
Q(A_1\times A_2,w')|\geq d \nu$.

\item $|\bar A\cap Q|\leq C'(\nu)n^{(s-1)-\gamma}$.
\end{enumerate}
\end{lem}
\proof  We choose the definable families in $\vec{\CF}$ as follows. Let
\begin{gather*}
	\CF_1 := \big\{\pi_1(Q(U,w)\cap Q(U,w')): \\
	w,w'\in V\,\&\, \dim \left(
Q(U,w)\cap Q(U,w') \right) < m\big\},
\end{gather*}
and $\CF_2 := \{\0\}$. Clearly, each set in $\CF_1$ has dimension smaller than $m$.
Because $Q$ is fiber algebraic of degree $\leq d$, it is easy to verify that (1) holds independently of the other families.

 For the other families,  we first recall that by Claim \ref{claim.5}, for each $w\in \overline{X}_{1,2}$, the set $X_w\sub \overline{X}_{1,2}$ has dimension smaller than $(s-3)m$.

 We now apply Lemma \ref{small dim} to the family $\{X_w:w\in \overline{X}_{1,2}\}$ (note that $s$ from Lemma \ref{small dim} is replaced here by $s-2$), and obtain definable families $\vec{\CF}'=(\CF_3,,\ldots, \CF_s)$, each $\CF_i$ consisting of subsets of $X_i$ of dimension smaller than $m$, such that
for every $\nu$ and every $n$-grid $A_3\times \cdots \times A_s\sub \overline{X}_{1,2}$ in $(\vec{\CF}',\nu)$-general position  and every $w\in \overline{X}_{1,2}$ we have
$$| \left(A_3\times \cdots \times A_s \right) \cap X_w|\leq C(\nu) n^{s-4}.$$

Let $\vec{\CF} := (\CF_1,\CF_2,\vec{\CF}')$ and assume that $\bar A$ is in $(\vec{\CF},\nu)$-general position.
It follows that for every $w\in A_3\times\cdots\times A_s$ there are at most $C(\nu) n^{s-4}$
elements $w'\in A_3\times \cdots \times A_s$ such that $|Q(A_1\times A_2,w)\cap
Q(A_1\times A_2,w')|\geq d \nu$. This proves (2).

We claim that the relation $Q$, viewed as a binary relation on $(X_1\times X_2)\times \overline{X}_{1,2}$, satisfies the $\gamma$-ST property. Indeed, for $i \in [s]$, 
let $X_i = \bigsqcup_{\ell \in [k_i]} X_{i,\ell} $ be an $o$-minimal cell decomposition of $X_i$, for some $k_i \in \mathbb{N}$, we have $m = \dim(X_i) = \max \left\{ \dim(X_{i, \ell}) : \ell \in [k_i] \right\}$. Then each (definable) cell $X_{i,\ell}$ is in a definable bijection with a definable subset of $M^{\dim \left(X_{i, \ell} \right)}$ (namely, the projection on the appropriate coordinates is a homeomorphism), hence in a definable bijection with a definable subset of $M^{m}$.
For $\bar{\ell} = (\ell_1, \ldots, \ell_s) \in [k_1] \times \ldots \times [k_s]$, let $Q_{\bar{\ell}} := Q \cap \prod_{i \in [s]} X_{i, \ell_i}$. 
Applying these definable bijections coordinate-wise, by Lemma \ref{lem: gamma-ST prop props}(1) we may assume $Q_{\bar{\ell}} \subseteq \prod_{i \in [s]} M^m$ and apply Fact \ref{o-min cutting} to conclude that for each $\bar{\ell}$, $Q_{\bar{\ell}}$ satisfies the $\gamma$-ST property. But then, by Lemma \ref{lem: gamma-ST prop props}(2), $Q$ also satisfies the $\gamma$-ST property. Finally, given an $n$-grid $\bar A\sub (X_1\times X_2)\times \overline{X}_{1,2}$ in $(\vec{\CF},\nu)$-general position, we thus have by the $\gamma$-ST property that (2) implies (3).\qed

This shows that $Q$ has $\gamma$-power saving, in contradiction to our assumption, thus completing the proof of  Proposition \ref{main-prop}, and with it Theorem \ref{main}.
\subsection{Obtaining a nice $Q$-relation}\label{sec: obt nice Q rel omin}
 By Theorem \ref{main} we may assume that $\dim Q=\dim Q^*$. Thus, in order to prove Theorem \ref{thm: main thm o-min1}, we may replace $Q$ by $Q^*$, and assume from now on that $Q=Q^*$.

Using $o$-minimal cell decomposition, we may partition $Q$ into finitely many definable sets such that each is \emph{fiber-definable}, namely for each tuple  $(a_1,\ldots, a_{s-1})\in A_1\times\cdots\times A_{s-1}$, there exists at most
one 
$$a_s=f(a_1,\ldots, a_{s-1})\in X_s$$
 such that $(a_1,\ldots,a_{s-1},a_s)\in Q$, and furthermore $f$ is a continuous function
on its domain. We can do the same for all permutations of the variables. Since $Q$
does not satisfy $\gamma$-power saving by assumption, one of these finitely many sets, of dimension
$(s-1)m$, also does not satisfy $\gamma$-power saving.

Hence from now on we assume that $Q$ is the graph of a
continuous partial function from any of its $s-1$ variables to the remaining one.


By further partitioning $Q$ and changing the sets up to definable bijections, we may
assume that each $X_i$ is an  open subset of $M^m$. Fix $\bar e=(e_1,\ldots,e_s)$ in $\CM$ generic
in $Q$, and let $\CM_0 := \dcl(\bar e)$. Note that for each $(a_3,\ldots, a_s)$ in a neighborhood of $(e_3,\ldots, e_s)$, the set $Q(x_1,x_2,a_3,\ldots, a_s)$ is the graph of a homeomorphism between neighborhoods of $e_1$ and $e_2$.
We let
$\mu_i:=\mu_{\CM_0}(e_i)$ (see Definition \ref{def: inf neighb}) and
identify these partial types  over $\CM_0$ with their sets of realizations in $\CM$.

\begin{lem}\label{lem: inf neighb with comp omin} There exist $\CM_0$-definable relatively open sets $U\sub X_1\times X_2$ and $V\sub \bar X_{1,2}$, containing $(e_1,e_2)$ and $(e_3,\ldots, e_s)$, respectively,
and a relatively open $W\subseteq Q$, containing $\bar e$, such that for every $(u,v)\in W$,
$Q(u,V)\times Q(U,v)\sub Q$.

In particular, for any $u,u'\in \mu_{\CM_0}(e_1,e_2)\cap \left( X_1\times X_2 \right)$ and any $v,v'\in \mu_{\CM_0}(e_3,\ldots,e_s)\cap \bar X_{1,2}$ we have
$$(u,v),(u,v'),(u',v)\in Q\Rightarrow (u',v')\in Q.$$
\end{lem}
\begin{proof} Because the properties of $U,V$ and $W$ are first-order expressible over $\bar e$, it is sufficient to prove the existence of $U,V,W$ in any elementary extension of $\CM_0$.

Because $\bar e\in Q=Q^*$, there are definable, relatively open neighborhoods $U\sub X_1\times X_2$ and $V\sub \bar X_{1,2}$ of $(e_1,e_2)$ and $(e_3,\ldots, e_s)$, respectively, such that
$$Q(U,e_3,\ldots,e_s)\times Q(e_1,e_2,V)\sub Q.$$

By Fact \ref{generic neighborhoods}, we may assume that $U,V$ are definable over $A\sub M$ such that $\bar e$ is still generic in $Q$ over $A$.
It follows that there exists a relatively open $W\sub Q$, containing $\bar e$, such that for every $(u,v)\in W$ (so,  $u\in X_1\times X_2$ and $v\in \bar X_{1,2}$), we have
$Q(U,v)\times Q(u,V)\sub Q$. As already noted, we now can find such $U,V$ and $W$ defined over $\mathcal{M}_0$.

 Note that $\mu_{\CM_0}(e_1,e_2)\cap (X_1\times X_2)\sub U$ and $\mu_{\CM_0}(e_3,\ldots, e_n)\cap \overline{X}_{1,2}\sub V$, and $\mu_{\CM_0}(\bar e)\sub W$.
 Let us see how the last clause follows: assume that  $u,u'\in \mu_{\CM_0}(e_1,e_2)\cap (X_1\times X_2)$, $v,v'\in \mu_{\CM_0}(e_3,\ldots,e_n)\cap \overline{X}_{1,2}$, and
$(u,v), (u,v'), (u',v)\in Q$. 
 We have $u,u'\in U$, $v,v'\in V$ and 
 $$(u,v),(u,v'), (u',v)\in W.$$ 
 By the choice of $U,V,W$, we thus have $(u',v')\in Q$.
\end{proof}

\begin{lem} \label{niceQ} The definable relation $Q$ satisfies properties (P1) and (P2) from Section \ref{sec: bij Q any arity} with respect to the $\CM_0$-type-definable sets $\mu_i \cap X_i, i \in [s]$, namely:

(P1) For any $(a_1,\ldots,a_{s-1})\in \mu_1\times \cdots\times\mu_{s-1}$, there exists exactly one $a_s\in\mu_s$ with $(a_1,\ldots,a_{s-1},a_s)\in Q$, and this remains true under any coordinate permutation.

(P2) Let $\tilde{U} := \mu_1\times \mu_2 \cap X_1 \times X_2$ and $\tilde{V} := \mu_3 \times \ldots \times \mu_s \cap \bar{X}_{1,2}$. Then for every $u,u'\in \tilde{U}$ and $w,w'\in \tilde{V}$,
$$(u;w),(u';w),(u;w')\in
Q \Rightarrow (u';w')\in Q.$$
The same is true when $(1,2)$ is replaced by any $(i,j)$ with $i\neq j \in [s]$.
\end{lem}
\begin{proof} By continuity of the function given by $Q$, for every tuple
$$(a_1,\ldots,a_{s-1}) \in \mu_1\times\cdots\times \mu_{s-1}$$
 there exists a unique $a_s \in \mu_s$ such that
$(a_1,\ldots,a_s)\in Q$. The same is true for any permutation of the variables. This shows (P1).

Property (P2) holds by Lemma \ref{lem: inf neighb with comp omin} for the $(1,2)$-coordinates. The same proof works for the other pairs $(i,j)$.
%
%
%
%
\end{proof}

%

Let us see how Theorem \ref{thm: main thm o-min1} follows. Assume first that  $s\geq 4$, and that $Q$ does not have $\gamma$-power saving for $\gamma=\frac{1}{16s-10}$. By Theorem \ref{main} and the resulting Lemma \ref{niceQ} (see also the choice of the parameters before Lemma \ref{lem: inf neighb with comp omin}),  there is $\bar e=(e_1,\ldots, e_s)$ generic in $Q$ and a substructure $\CM_0=\dcl(\bar e)$ of cardinality $|\mathcal{L}|$ such that $Q\cap \prod_{i \in [s]}\left( \mu_{\CM_0}(e_i) \cap X_i \right)$ satisfies $(P1)$ and $(P2)$ of Theorem \ref{thm: main group config bijections}. Note that $\mu_{\mathcal{M}_0}(e_i)$ is a partial type over $\mathcal{M}_0$ for $i \in [s]$, $\bar{e}$ satisfies the relation, and $\bar{e}$ is contained in $\mathcal{M}_0$.
 Thus, by the ``moreover'' clause of Theorem \ref{thm: main group config bijections}, there exists a type definable abelian group $G$ over $\CM_0$
 and $\CM_0$-definable bijections $\pi'_i:\mu_{\CM_0}(e_i)\cap X_i\to G$ each sending $e_i$ to $0_G$ and satisfying:
 $$Q(a_1,\ldots, a_n) \Leftrightarrow \pi'_1(a_1)+\cdots +\pi'_m(a_m)=0$$
 for all $a_i \in \mu_{\CM_0}(e_i)\cap X_i$.
 This is exactly the second clause of Theorem \ref{thm: main thm o-min1}.
 
 Finally, the case $s=3$ of Theorem \ref{thm: main thm o-min1} reduces to the case $s=4$ as in the stable case, Section \ref{sec: stable main thm ternary}, with the obvious modifications.

\subsection{Discussion and some applications}\label{sec: apps in o-min}

We discuss some variants and corollaries of the main theorem. In particular, we will deduce a variant that holds in an arbitrary $o$-minimal structure, i.e.~without the saturation assumption on $\CM$ used in Theorem \ref{thm: main thm o-min1}.

\begin{defn}(see \cite[Definition 2.1]{goldbring2010hilbert}) A {\em local group} is a tuple $(\Gamma, 1,\iota,p)$, where $\Gamma$ is a Hausdorff topological space,  $\iota:\Lambda\to \Gamma$ (the inversion map) and $p:\Omega\to \Gamma$ (the product map) are continuous functions, with $\Lambda \subseteq \Gamma$ and $\Omega \subseteq \Gamma^2$ open subsets, such that
$1\in \Lambda$, $\{1\}\times \Gamma, \Gamma\times \{1\}\sub \Omega$ and for all $x,y,z\in \Gamma$:
\begin{enumerate}\item $p(x,1)=p(1,x)=x$;
\item if $x\in \Lambda$ then $(x,\iota(x)), (\iota(x),x)\in \Omega$ and $p(x,\iota(x))=p(\iota(x),x)=1$;
\item if $(x,y),(y,z)\in \Omega$ and $(p(x,y),z),(x,p(y,z))\in \Omega$, then
$$p((p(x,y),z)=p(x,p(y,z)).$$
\end{enumerate}

\end{defn}

  Our goal is to show that in Theorem \ref{thm: main thm o-min1} we can replace the type-definable group with {\em a definable} local group. Namely,
\begin{cor}\label{thm:local o-min1} Let $\CM$ be an $\aleph_0$-saturated $o$-minimal expansion of a group, $s \geq 3$,
$Q \subseteq  X_1\times \cdots \times X_s$ are $\emptyset$-definable with $\dim(X_i) = m$, and $Q$ is fiber algebraic. Then one  of the
following  holds.
\begin{enumerate}
\item The set $Q$ has $\gamma$-power saving, for $\gamma = \frac{1}{8m-5}$ if $s \geq 4$, and $\gamma = \frac{1}{16m-10}$ if $s=3$.

\item There exist (i) a finite set $A\sub M$ and  a structure $\CM_0=\dcl(A)$ (ii)   a
definable local  abelian group $\Gamma$ with $\dim(\Gamma) = m$, defined over $M_0$, (iii) definable relatively open $U_i\sub X_i$, a definable open neighborhood $V\sub \Gamma$ of $0=0_\Gamma$, and (iv) definable homeomorphisms
$\pi_i:U_i\to V$, $i \in [s]$, such that for all $x_i \in U_i$,
$$\pi_1(x_1)+\cdots+\pi_s(x_s)=0\Leftrightarrow Q(x_1,\ldots,x_s).$$

\end{enumerate}
\end{cor}
\begin{proof} We assume that (1) fails and apply  Theorem \ref{thm: main thm o-min1} to obtain $\bar a$ generic in $Q$, $\CM_0= \dcl(\bar a)$,  a type-definable abelian group $G$ over $\CM_0$, and bijections $\pi_i:\mu_{\CM_0}(a_i)\to G$ sending $a_i$ to $0$, such that for all $i\in [s]$, and $x_i\in \mu_{\CM_0}(a_i)$,
$$\pi_1(x_1)+\cdots+\pi_s(x_s)=0\Leftrightarrow Q(x_1,\ldots,x_s).$$

By pulling back the group operations via, say, $\pi_1$, we may assume that the domain of $G$ is $\mu_{\CM_0}(a_1)$.
We denote this pull-back of the addition and the inverse operations by $x\oplus y$ and $\ominus y$, respectively. Let us see that $\oplus$ and $\ominus$ are continuous with respect to the induced topology on $\mu_{\CM_0}(a_1)\sub X_1$.
Because $\bar a$ is generic in $Q$, and $Q$ is fiber algebraic, it follows from $o$-minimality that the set $Q(x_1,x_2,x_3, a_4,\ldots,a_s)$ defines a continuous function from any two of the coordinates $x_1,x_2,x_3$ to the third one, on the corresponding infinitesimal types $\mu_{\CM_0}(a_i)\times \mu_{\CM_0}(a_j)$.

The following is easy to verify: for $x',x'',x''' \in \mu_{\CM_0}(a_1)$,
$x'\oplus x''=x'''$ if and only if there exist $x_2\in \mu_{\CM_0}(a_2)$ and $x_3,x_3'\in \mu_{\CM_0}(a_3)$ such that
\[\begin{array}{lll}

Q \left(x',x_2,x_3,a_4,\ldots, a_s \right), & Q \left(x''',a_2,x_3,a_4,\ldots,a_s \right) &\mbox{ and } \\ 
 Q(x'',a_2,x_3',a_4,\ldots, a_s), & Q(a_1,x_2,x_3',a_4,\ldots,a_s).  & 

\end{array} \]

By the above comments, $\oplus$ can thus be obtained as a composition of continuous maps, thus it is continuous. We similarly show that $\ominus$ is continuous.

Applying logical compactness, we may now replace the type-definable $G$ with an  $\CM_0$-definable $\Gamma \supseteq G=\mu_{\CM_0}(a_1)$, with partial continuous group operations, which make $\Gamma$ into a local group (we note that in general, any type-definable group is contained in a definable local group by logical compactness, except for the topological conditions). Similarly, we find $U_i\supseteq \mu_{\CM_0}(a_i)$, $V\sub \Gamma$ and $\pi_i:U_i\to V$ as needed.
\end{proof}

Note that if $\mathbb R_{\omin}$ is an o-minimal expansion of the field of reals  and the $X_i$'s and $Q$ are definable in $\mathbb R_{\omin}$, with $Q$ not satisfying Clause (1) of Corollary \ref{thm:local o-min1}, then taking a sufficiently saturated elementary extension $\CM \succeq \mathbb R_{\omin}$, $Q(\CM)$ still does not satisfy Clause (1) in $\CM$. Hence we may deduce that Clause (2) of Corollary \ref{thm:local o-min1} holds for $Q$ in $\CM$, possibly over additional parameters from $\CM$. However, the definition of a local group is first-order in the parameters defining $\Gamma$, $\iota$ and $p$. Thus, by elementarity, we obtain that Clause (2) of Corollary \ref{thm:local o-min1}  holds for $Q(\mathbb R)$, with $\Gamma$ and the functions $\pi_i$ definable in the original structure $\mathbb R_{\omin}$.

By Goldbring's solution \cite{goldbring2010hilbert} to the Hilbert's 5th problem for local groups, if $\Gamma$ is a \emph{locally Euclidean} local group (i.e.~there is an open neighborhood of $1$ homeomorphic to an open subset of $\mathbb{R}^n$, for some $n$), then there is a neighborhood $U$ of $1$ such that $U$ is isomorphic, as a local group, to an open subset of an actual Lie group $G$. Clearly, if the local group is abelian then the connected component of $G$ is also abelian. Combining these observations with Corollary \ref{thm:local o-min1} we conclude:

\begin{cor}\label{thm:real o-min1} Let $\mathbb R_{\omin}$ be an o-minimal expansion of the field of reals.  Assume $s \geq 3$,
$Q \subseteq  X_1\times \cdots \times X_s$ are $\emptyset$-definable with $\dim(X_i) = m$, and $Q$ is fiber-algebraic. Then one  of the
following  holds.
\begin{enumerate}
\item The set $Q$ satisfies $\gamma$-power saving, for $\gamma = \frac{1}{8m-5}$ if $s \geq 4$, and $\gamma = \frac{1}{16m-10}$ if $s=3$.

\item There exist definable relatively open sets $U_i\sub X_i$, $i\in [s]$,   an abelian Lie group $(G,+)$ of dimension $m$
and an open neighborhood $V\sub G$ of $0$, and definable homeomorphisms
$\pi_i:U_i\to V$, $i \in [s]$, such that for all $x_i \in U_i, i \in [s]$
$$\pi_1(x_1)+\cdots+\pi_s(x_s)=0\Leftrightarrow Q(x_1,\ldots,x_s).$$

\end{enumerate}
\end{cor}

Finally, this takes a particularly explicit form when $\dim(X_i) = 1$ for all $i \in [s]$.

\begin{cor}\label{cor: 1-dim o-min}
Let $\mathbb R_{\omin}$ be an o-minimal expansion of the field of reals. Assume $s \geq 3$ and $Q \subseteq  \mathbb{R}^s$ is definable and fiber-algebraic. Then exactly one  of the
following  holds.
\begin{enumerate}
\item There exists a constant $c$, depending only on the formula defining $Q$ (and not on its parameters), such that: for any finite $A_i \subseteq \mathbb{R}$ with $|A_i| = n$ for $i \in [s]$ we have
$$|Q \cap (A_1 \times \ldots \times A_s)| \leq c n^{s - 1 - \gamma},$$
where $\gamma = \frac{1}{3}$ if $s \geq 4$, and $\gamma = \frac{1}{6}$ if $s=3$.

\item There exist definable open sets $U_i \subseteq \mathbb{R}, i \in [s]$, an open set $V \subseteq \mathbb{R}$ containing $0$, and homeomorphisms $\pi_i: U_i \to V$ such that
$$\pi_1(x_1)+\cdots+\pi_s(x_s)=0\Leftrightarrow Q(x_1,\ldots,x_s)$$
for all $x_i \in U_i, i \in [s]$.
\end{enumerate}
\end{cor}
\begin{proof}
Corollary \ref{thm:real o-min1} can be applied to $Q$.

Assume we are in Clause (1). As the proof of Theorem \ref{thm: main thm o-min1} demonstrates, we can take any $\gamma$ such that $Q$ satisfies the $\gamma$-ST property (as a binary relation, under any partition of its variables into two and the rest) if $s \geq 4$; and such that $Q'$ (as defined in Section \ref{sec: stable main thm ternary}) satisfies the $\gamma$-ST property  if $s = 3$. Applying the stronger bound for definable subsets of $\mathbb R^2 \times \mathbb{R}^{d_2}$ from Fact \ref{o-min cutting}(1), we get the desired $\gamma$-power saving.
Note that in the $1$-dimensional case, the general position requirement is satisfied automatically: for any definable set $Y \subseteq \mathbb{R}$, $\dim(Y) < 1$ if and only if $Y$ is finite; and for every definable family $\CF_i$ of subsets of $\mathbb{R}$, by $o$-minimality there exists some $\nu_0$ such that for any $Y \in \CF_i$, if $Y$ has cardinality greater than $\nu_0$ then it is infinite.

	In Clause (2), we use that every connected $1$-dimensional Lie group $G$ is isomorphic to either $(\mathbb{R}, +)$ or $S^1$, and in the latter case we can restrict to a neighborhood of $0$ and compose the $\pi_i$'s with a local isomorphism from $S^1$ to $(\mathbb{R}, +)$.
	
	Finally, the two clauses are mutually exclusive as in Remark \ref{rem: mut excl}.
\end{proof}

\begin{rem}\label{rem: getting analytic}In the case that definable sets in $\mathbb R_{\omin}$ admit \emph{analytic cell decomposition} (e.g.~in the $o$-minimal structure $\RR_{\operatorname{an, exp}}$, see \cite[Section 8]{van1994real}) then one can strengthen Clause (2) in Corollaries \ref{thm:real o-min1} and \ref{cor: 1-dim o-min}, so that the $U_i$'s are analytic submanifolds and the maps $\pi_i$  are analytic bijections with analytic inverses.\end{rem}
\begin{rem}
If $Q$ is semialgebraic (which corresponds to the case $\RR_{\omin} = \RR$ of Corollary \ref{cor: 1-dim o-min}),  of description complexity $D$ (i.e.~defined by at most $D$ polynomial (in-)equalities, with all polynomials of degree at most $D$), then in Clause (1) the constant $c$ depends only on $s$ and $D$ (as all $Q$'s are defined by the instances of a single formula depending only on $s$ and $D$).
\end{rem}

\begin{rem}\label{rem: optimal power save semilin}
If $Q$ is semilinear, then by Fact \ref{fac: semilin} it satisfies $(1-\varepsilon)$-ST property, for any $\varepsilon > 0$. In this case, in Clause (1) of Corollary \ref{cor: 1-dim o-min} for $s \geq 4$ we get $(1-\varepsilon)$-power saving --- which is essentially the best possible bound. See \cite{makhul2020constructions} concerning the lower bounds on power saving.
\end{rem}

  \bibliography{refs}

\end{document}